\documentclass[10pt]{article}

  \usepackage[latin1]{inputenc}
  \usepackage[T1]{fontenc}
\usepackage{url}                
\usepackage{amsmath,amsfonts,amssymb,amsthm,amscd,array, mathrsfs, stmaryrd,mathrsfs, mathdots}
\usepackage{multirow, blkarray}
\usepackage{xcolor}
\usepackage{graphics, graphicx}
\usepackage{pstricks,pict2e}       
\usepackage{epsfig}
\usepackage{psfrag}
\usepackage{epstopdf}
\usepackage{color}
\usepackage{textcomp}
\usepackage{chngpage}
\usepackage{wrapfig}             
 \usepackage[all]{xy}

\usepackage[colorlinks=true,linkcolor=blue, linktoc=page, citecolor=purple]{hyperref}
\usepackage{vmargin}
\usepackage{fancyhdr}
\usepackage{dsfont}

 \newtheorem{lem}{Lemma}[section]
 \newtheorem{thm}{Theorem}[section]

 \newtheorem{prop}[thm]{Proposition}
 \newtheorem{cor}[thm]{Corollary}

\theoremstyle{definition}

\newtheorem{defn}[thm]{Definition}
\newtheorem{rem}[thm]{Remark}
\newtheorem{ex}[thm]{Example}

 \newtheorem{notation}{Notation}
 \newtheorem{quest}{Q}

\newcommand{\R}{\mathbb{R}}
\newcommand{\Z}{\mathbb{Z}}
\newcommand{\C}{\mathbb{C}}
\newcommand{\CP}{\mathbb{CP}}

\newcommand{\N}{\mathbb{N}}

\newcommand{\bbP}{\mathbb{P}}
\newcommand{\Q}{\mathbb{Q}}

\newcommand{\A}{\mathcal{A}}

\newcommand{\Gc}{{\small \mathcal{G}}}
\newcommand{\F}{\mathcal{F}}

\newcommand{\cC}{\mathcal{C}}

\newcommand{\Qc}{\mathcal{Q}} 
\newcommand{\cM}{\mathcal{M}}

\newcommand{\E}{\mathcal{E}}
\newcommand{\Rc}{\mathcal{R}}
\newcommand{\Sc}{\mathcal{S}}

\newcommand{\Dc}{\mathcal{D}}

\newcommand{\rep}{\mathrm{rep}\,\Qc}
\newcommand{\modQ}{\mathrm{mod }\, \C\Qc}

\newcommand{\id}{\textup{Id}}

\newcommand{\Gr}{\textup{Gr}}

\newcommand{\Td}{\textup{T}}

\newcommand{\SL}{\mathrm{SL}}
\newcommand{\GL}{\mathrm{GL}}

\newcommand{\PGL}{\mathrm{PGL}}
\newcommand{\mA}{\mathrm{ \bf A}}
\newcommand{\mB}{\mathrm{\bf B}}
\newcommand{\mC}{\mathrm{\bf C}}
\newcommand{\mD}{\mathrm{\bf D}}
\newcommand{\mE}{\mathrm{\bf E}}
\newcommand{\mG}{\mathrm{\bf G}}

\newcommand{\udim}{ \underline{\dim}\;}
\newcommand{\op}{\textup{op}}

\def\a{\alpha}

\def\G{\Gamma}

\def\l{\lambda}

\makeindex
\date{}

\begin{document}
\title{Coxeter's frieze patterns at the crossroads of algebra, geometry and combinatorics}

\author{Sophie Morier-Genoud\\
Sorbonne Universit\'es, UPMC Univ Paris 06, UMR 7586, \\
Institut de Math\'ematiques de Jussieu- Paris Rive Gauche,\\ 
Case 247, 4 place Jussieu, F-75005, Paris, France\\
sophie.morier-genoud@imj-prg.fr,
}



\maketitle

\begin{abstract}
Frieze patterns of numbers, introduced in the early 70's by Coxeter,
 are currently attracting much interest
due to connections with the recent theory of cluster algebras.
The present paper aims to review the original work of Coxeter and the new developments around the notion of frieze, focusing on the representation theoretic, geometric and combinatorial approaches.
\end{abstract}

\tableofcontents

\section*{Introduction}
\addcontentsline{toc}{section}{Introduction}

Frieze patterns were introduced by Coxeter in the early 70's, \cite{Cox}.
They are arrays of numbers where neighboring values are connected by a local arithmetic rule.
Coxeter introduces these patterns to understand Gauss' formulas for
the \textit{pentagramma mirificum} and their possible generalizations.
A remarkable property of friezes is the glide symmetry, 
implying their periodicity. 
Coxeter establishes many interesting connections between friezes and various objects:
cross-ratios, continued fractions, Farey series...  
Frieze patterns of positive integers are of special interest. 
Conway and Coxeter~\cite{CoCo} discover a one-to-one correspondence 
between friezes of positive integers and triangulations of polygons. 
This reveals rich combinatorics around frieze patterns.

Frieze patterns appear independently in the 70's in a different
and disconnected context
of quiver representations.
It turns out that the Auslander-Reiten quiver leads to a variant of the Coxeter friezes,
with the local arithmetic rule being an additive analogue of Coxeter's unimodular rule.
The generalization of Coxeter's unimodular rule on Auslander-Reiten quivers
was found recently by Caldero and Chapoton~\cite{CaCh}.

A recent revival of friezes is due to their relation to the 
Fomin-Zelevinsky theory of cluster algebras,
beautiful and unexpected connections with
Grassmannians, linear difference equations, and moduli spaces of points in projective spaces. 
These connections explain recent applications of friezes in integrable systems.

New directions of study and new developments of the notion of frieze have been recently and are currently investigated.
There are mainly three approaches for the study of friezes:

- a representation theoretical and categorical approach, in deep 
connection with the theory of cluster algebras, where entries in the friezes are rational fractions;

- a geometric approach, in connection with moduli spaces of points in projective space and 
Grassmannians, where entries in the friezes are more often real or complex numbers;

- a combinatorial approach, focusing on friezes with positive integer entries.\\

The present article aims to give an overview of the different approaches and recent results about
the notion of friezes. The article is organized in four large independent sections
and a short conclusion.

In Section \ref{FrCC} we review the original work of Coxeter on frieze patterns
\cite{Cox}, and discuss some immediate extensions of the results,
concerning algebraic properties of friezes and links to projective geometry. 

In Section \ref{FrQuiv} we present a generalization of friezes based on representations of quivers and the theory of cluster algebras, introduced in \cite{CaCh}, \cite{ARS}.

In Section \ref{SLtil} we present the variants of $\SL_{k}$-tilings and $\SL_{k}$-friezes introduced in \cite{BeRe}. We review results from \cite{MGOST} where the space of $\SL_{k}$-friezes is identified with the moduli space of projective polygons and the space of superperiodic difference equations.

Section \ref{comb} focuses on friezes of positive integers and their relations with different combinatorial models. In particular,
we give the Conway-Coxeter correspondence with triangulations of polygons \cite{CoCo}
and present some generalizations.

The final section lists the different variants of friezes appearing in the literature.

\section{Coxeter's frieze patterns}\label{FrCC}

Coxeter's frieze patterns \cite{Cox} are arrays of numbers satisfying the following properties:

(i) the array has finitely many rows, all of them being infinite on the right and left,

(ii) the first two top rows are a row of 0's followed by a row of 1's, and the last two bottom rows
are a row 1's followed by a row of 0's\footnote{When representing friezes, one often omits the bordering top and bottom rows of 0's.},

(iii) consecutive rows are displayed with a shift, and every four adjacent entries $a, b,c,d$  forming a diamond $$
 \begin{array}{ccccccc}
 &b&\\
 a&&d\\
 &c&
\end{array}$$ satisfy the \textit{unimodular rule}: $ad-bc=1$.

The number of rows strictly between the border rows of 1's is called the \textit{width} of the frieze (we will use the letter $m$ for the width).
The following array \eqref{exCoCox} is an example of a frieze pattern of width  $m=4$, containing only positive integer numbers.
\begin{equation}\label{exCoCox}
 \begin{array}{lcccccccccccccccccccccccc}
\text{\small{row 0}}&&&1&&1&& 1&&1&&1&&1&&1&& \cdots\\[4pt]
\text{\small{row 1}}&&\cdots&&4&&2&&1&&3&&2&&2&&1&&
 \\[4pt]
\text{\small{row 2}}&&&3&&7&&1&&2&&5&&3&&1&&\cdots&\
 \\[4pt]
\cdots &&\cdots& &5&&3&&1&&3&&7&&1&&2&\\[4pt]
\text{\small{row $m$}}&&& 3&&2&&2&&1&&4&&2&&1&&\cdots\\[4pt]
\text{\small{row $m+1$}}&&\cdots&&1&&1&&1&&1&&1&&1&&1&&
\end{array}
\end{equation}

The definition allows the frieze to take its values in any ring with unit.
Coxeter studies the properties of friezes with entries that are positive real numbers (apart from the border rows of 0's),
 and with a special interest in the case of positive integers. 
The combinatorics related to friezes with positive integer entries will be presented in Section \ref{comb}.
 
 The condition of positivity is quite strong but guarantees a certain genericity of the frieze. We will work with a less restrictive condition.
Throughout this section, we will consider friezes with real or complex entries, and we will assume that they satisfy the following extra condition:
 
 (iv) every adjacent $3\times 3$-submatrix in the array has determinant 0.
 
 Friezes satisfying the condition (iv) are called \textit{tame}, \cite{BeRe}.
 Coxeter's friezes with no zero entries (in particular friezes with positive numbers) are all tame.
The statements established in \cite{Cox} for the friezes with positive entries still hold for tame friezes 
(the proofs can be easily adapted). 

\subsection{Pentagramma Mirificum and frieze patterns of width 2}\label{width2}
The first example of frieze pattern given by Coxeter is the frieze of width 2 made out of the Gauss formulas for the  \textit{pentagramma mirificum}. The pentagramma mirificum is a pentagram drawn on a unit sphere with successively orthogonal great circle arcs (see Figure \ref{mirif}). 
If we denote by $\a_{1}, \ldots, \a_{5}$ the length of the side arcs of the inner pentagon, we obtain the following relations for $c_{i}=\tan^{2}(\a_{i})$:
\begin{equation}\label{pentarec1}
c_{i}c_{i+1}=1+c_{i+3},
\end{equation}
where the indices are taken modulo 5.
In other words, the quantities $c_{i}$'s related to the pentagramma mirificum form
 a frieze pattern of width 2:
\begin{equation}\label{pentafrieze}
 \begin{array}{ccccccccccccccccccc}
&&1&&1&& 1&&1&&1&&1&&\cdots \\[4pt]
&\cdots&&c_{1}&&c_{2}&&c_{3}&&c_{4}&&c_{5}&&c_{1}&&
 \\[4pt]
&&c_{3}&&c_{4}&&c_{5}&&c_{1}&&c_{2}&&c_{3}&&\cdots
 \\[4pt]
&\cdots&&1&&1&&1&&1&&1&&1&&
\end{array}
\end{equation}
Moreover, Gauss observed that the first three equations of \eqref{pentarec1}, i.e. for $i=1,2,3$, imply the last two, i.e. for $i=4,5$. 
This observation implies the 5-periodicity of any frieze pattern of width 2.
It seems that Coxeter's motivation in the study of friezes was to generalize this situation.\\

\begin{figure}
\begin{center}
\includegraphics[width=6cm]{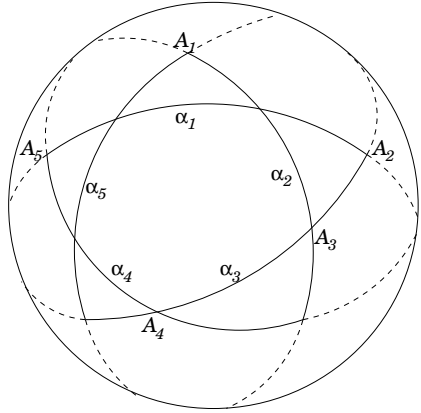}
\end{center}
\caption{Selfpolar pentagram on the sphere (sides are great circles, angles at vertices $A_i$ are right angles). The quantities $c_i:=\tan^2 \a_i$ satisfy the relations \eqref{pentarec1}.}
\label{mirif}
\end{figure}

Given 5 points $p_{1}, p_{2}, p_{3}, p_{4}, p_{5}$ on the (real or complex) projective line, with indices taken cyclically 
($p_{i+5}=p_i$), we form the 5 cross ratios 
\begin{equation*}
c_{i}=[p_{i+1},p_{i+2},p_{i+3},p_{i+4}]=
\frac{(p_{i+4}-p_{i+1})(p_{i+3}-p_{i+2})}{(p_{i+4}-p_{i+3})(p_{i+2}-p_{i+1})}.
\end{equation*}
Assume that two consecutive points, taken cyclically, are distinct
(this guarantees that none of the $c_{i}$'s are infinity).
Then, one checks that the above cross ratios satisfy the relations
\eqref{pentarec1}.

Thus, Coxeter friezes of width 2, with non-zero values, parametrize the moduli space
$$ \cM_{0,5}:=\{p_i\in \mathbb{P}^1, p_{i+5}=p_i, p_{i}\not=p_{j}\}/ \PGL_2.$$
Consider the following frieze of width 2, with a given diagonal of non-zero variables $x_{1}, x_{2}$, that we compute by applying the unimodular rule:
\begin{equation}\label{frieze2}
 \begin{array}{ccccccccccccccccccccc}
\cdots&&1&& 1&&1&&1&&1&&1&&\cdots
 \\[4pt]
&x_1&&\frac{x_2+1}{x_1}&&\frac{x_1+1}{x_2}&&x_2&&\frac{1+x_{2}+x_{1}}{x_{2}x_{1}}&&x_{1}
 \\[4pt]
\cdots&&x_2&&\frac{x_1+x_2+1}{x_1x_2}&&x_1&&\frac{1+x_{2}}{x_{1}}&&\frac{1+x_{1}}{x_{2}}&&x_{2}&&\cdots
 \\[4pt]
&1&&1&&1&&1&&1&&1&&1
\end{array}
\end{equation}
One observes that the entries are all Laurent polynomials in $x_{1}, x_{2}$.
In other words, in the pattern \eqref{pentafrieze}, every entry can be expressed as a Laurent polynomial in any two fixed entries $(c_{i}, c_{i+3})$. 
Therefore, the frieze pattern provides 5 charts
$$
(\C^{*})^{2}\;\longrightarrow\;  \cM_{0,5},
$$
with transition functions written as Laurent polynomials.
These charts cover a slightly bigger space than $\cM_{0,5}$, in which non-consecutive points may coincide,  which is denoted by:
$$\widehat \cM_{0,5}:=\{p_i\in \mathbb{P}^1, p_{i+5}=p_i, p_{i}\not=p_{i+1}\}/ \PGL_2.
$$
Indeed, since consecutive points are distinct
there is necessarily a couple of non-zero cross ratios of the form $(c_{i}, c_{i+3})$.
Conversely, given a frieze pattern \eqref{pentafrieze}, with a non-zero diagonal, say $(c_{1},c_{4})$, one can recover 5 points  of $\mathbb{P}^1$ whose associated cross ratios are 
$c_{1}, \ldots, c_{5}$  by quotienting two consecutive diagonals
$$\textstyle
\frac01,\; \frac{1}{c_{1}}, \;\frac{c_{2}}{c_{4}}, \;\frac{c_{5}}{1},\; \frac10.
$$

Frieze patterns of width $m$ will generalize this situation to configurations of $m+3$ points (see Section~\ref{M0n}).

\begin{figure}
\begin{center}
\setlength{\unitlength}{1844sp}%
\begingroup\makeatletter\ifx\SetFigFont\undefined%
\gdef\SetFigFont#1#2#3#4#5{%
  \reset@font\fontsize{#1}{#2pt}%
  \fontfamily{#3}\fontseries{#4}\fontshape{#5}%
  \selectfont}%
\fi\endgroup%
\begin{picture}(4994,6572)(879,-5708)
\put(4388,-286){\makebox(0,0)[lb]{\smash{{\SetFigFont{10}{14.4}{\rmdefault}{\mddefault}{\updefault}{\color[rgb]{0,0,0}$p_1$}%
}}}}
\thicklines
{\color[rgb]{0,0,0}\put(901,-4111){\line( 5, 6){3780.738}}
}%
{\color[rgb]{0,0,0}\put(901,-3661){\line( 5, 3){4500}}
}%
{\color[rgb]{0,0,0}\put(901,-2986){\line( 3,-1){4927.500}}
}%
{\color[rgb]{0,0,0}\put(901,-2761){\line( 5,-3){4500}}
}%
{\color[rgb]{0,0,0}\put(901,-3211){\line( 1, 0){4950}}
}%
\put(4388,614){\makebox(0,0)[lb]{\smash{{\SetFigFont{10}{14.4}{\rmdefault}{\mddefault}{\updefault}{\color[rgb]{0,0,0}$\mathbb{P}^1$}%
}}}}
\put(4388,-5236){\makebox(0,0)[lb]{\smash{{\SetFigFont{10}{14.4}{\rmdefault}{\mddefault}{\updefault}{\color[rgb]{0,0,0}$p_5$}%
}}}}
\put(4388,-4336){\makebox(0,0)[lb]{\smash{{\SetFigFont{10}{14.4}{\rmdefault}{\mddefault}{\updefault}{\color[rgb]{0,0,0}$p_4$}%
}}}}
\put(4388,-3436){\makebox(0,0)[lb]{\smash{{\SetFigFont{10}{14.4}{\rmdefault}{\mddefault}{\updefault}{\color[rgb]{0,0,0}$p_3$}%
}}}}
\put(4388,-1861){\makebox(0,0)[lb]{\smash{{\SetFigFont{10}{14.4}{\rmdefault}{\mddefault}{\updefault}{\color[rgb]{0,0,0}$p_2$}%
}}}}
{\color[rgb]{0,.56,.56}\put(4276,614){\line( 0,-1){6300}}
}%
\end{picture}%
\end{center}
\end{figure}

\begin{rem}
Renaming the variables
as $x_{1}=c_{1}$, $x_{2}=c_{4}$, $x_{3}=c_{2}$, $x_{4}=c_{5}$, $x_{5}=c_{3}$ lead to the 
 the famous \textit{pentagon recurrence}:
\begin{equation}\label{pentarec2}
x_{i-1}x_{i+1}=1+x_{i}.
\end{equation}
It is well known, and easy to establish, that every sequence $(x_{i})_{i\in \Z}$ with no zero values  satisfying the recurrence \eqref{pentarec2},
is 5-periodic, for details see e.g. \cite{FoRe}.
\end{rem}

\subsection{Periodicity and glide symmetry}\label{sym}

The frieze pattern of width 2 given in \eqref{frieze2} reveals two important properties of friezes:
periodicity and invariance under glide reflection. This is a general fact.

\begin{thm}[\cite{Cox}]\label{thmPerio}
Rows in a frieze of width $m$ are periodic with period dividing $m+3$.
\end{thm}
The period $n=m+3$ is called the order of the frieze in \cite{Cox}.
This periodicity is actually implied by a stronger symmetry.
Recall that a glide reflection is the composition of a reflection about a line and a translation along that line. 

\begin{thm}[\cite{Cox}]\label{thmGlide}
Friezes are invariant under a glide reflection with respect to the horizontal median line of the pattern.
\end{thm}

In other words, friezes consist of a fundamental domain, for instance triangular, that is reflected and translated horizontally.
Extending the pattern \eqref{exCoCox} one observes this property.
\\

\setlength{\unitlength}{1644sp}%
\begingroup\makeatletter\ifx\SetFigFont\undefined%
\gdef\SetFigFont#1#2#3#4#5{%
  \reset@font\fontsize{#1}{#2pt}%
  \fontfamily{#3}\fontseries{#4}\fontshape{#5}%
  \selectfont}%
\fi\endgroup%
\begin{picture}(13787,3173)(2476,-1310)
\put(10351,1289){\makebox(0,0)[lb]{\smash{{\SetFigFont{10}{20.4}{\rmdefault}{\mddefault}{\updefault}{\color[rgb]{0,0,0}1}%
}}}}
\put(15751,1289){\makebox(0,0)[lb]{\smash{{\SetFigFont{10}{20.4}{\rmdefault}{\mddefault}{\updefault}{\color[rgb]{0,0,0}. . .}%
}}}}
\put(15751,389){\makebox(0,0)[lb]{\smash{{\SetFigFont{10}{20.4}{\rmdefault}{\mddefault}{\updefault}{\color[rgb]{0,0,0}. . .}%
}}}}
\put(15751,-511){\makebox(0,0)[lb]{\smash{{\SetFigFont{10}{20.4}{\rmdefault}{\mddefault}{\updefault}{\color[rgb]{0,0,0}. . .}%
}}}}
\put(2476,-961){\makebox(0,0)[lb]{\smash{{\SetFigFont{10}{20.4}{\rmdefault}{\mddefault}{\updefault}{\color[rgb]{0,0,0}. . .}%
}}}}
\put(2476,-61){\makebox(0,0)[lb]{\smash{{\SetFigFont{10}{20.4}{\rmdefault}{\mddefault}{\updefault}{\color[rgb]{0,0,0}. . .}%
}}}}
\put(2476,839){\makebox(0,0)[lb]{\smash{{\SetFigFont{10}{20.4}{\rmdefault}{\mddefault}{\updefault}{\color[rgb]{0,0,0}. . .}%
}}}}
\put(15301,-961){\makebox(0,0)[lb]{\smash{{\SetFigFont{10}{20.4}{\rmdefault}{\mddefault}{\updefault}{\color[rgb]{0,0,0}1}%
}}}}
\put(14401,-961){\makebox(0,0)[lb]{\smash{{\SetFigFont{10}{20.4}{\rmdefault}{\mddefault}{\updefault}{\color[rgb]{0,0,0}1}%
}}}}
\put(13501,-961){\makebox(0,0)[lb]{\smash{{\SetFigFont{10}{20.4}{\rmdefault}{\mddefault}{\updefault}{\color[rgb]{0,0,0}1}%
}}}}
\put(12601,-961){\makebox(0,0)[lb]{\smash{{\SetFigFont{10}{20.4}{\rmdefault}{\mddefault}{\updefault}{\color[rgb]{0,0,0}1}%
}}}}
\put(11701,-961){\makebox(0,0)[lb]{\smash{{\SetFigFont{10}{20.4}{\rmdefault}{\mddefault}{\updefault}{\color[rgb]{0,0,0}1}%
}}}}
\put(10801,-961){\makebox(0,0)[lb]{\smash{{\SetFigFont{10}{20.4}{\rmdefault}{\mddefault}{\updefault}{\color[rgb]{0,0,0}1}%
}}}}
\put(9901,-961){\makebox(0,0)[lb]{\smash{{\SetFigFont{10}{20.4}{\rmdefault}{\mddefault}{\updefault}{\color[rgb]{0,0,0}1}%
}}}}
\put(9001,-961){\makebox(0,0)[lb]{\smash{{\SetFigFont{10}{20.4}{\rmdefault}{\mddefault}{\updefault}{\color[rgb]{0,0,0}1}%
}}}}
\put(8101,-961){\makebox(0,0)[lb]{\smash{{\SetFigFont{10}{20.4}{\rmdefault}{\mddefault}{\updefault}{\color[rgb]{0,0,0}1}%
}}}}
\put(7201,-961){\makebox(0,0)[lb]{\smash{{\SetFigFont{10}{20.4}{\rmdefault}{\mddefault}{\updefault}{\color[rgb]{0,0,0}1}%
}}}}
\put(6301,-961){\makebox(0,0)[lb]{\smash{{\SetFigFont{10}{20.4}{\rmdefault}{\mddefault}{\updefault}{\color[rgb]{0,0,0}1}%
}}}}
\put(5401,-961){\makebox(0,0)[lb]{\smash{{\SetFigFont{10}{20.4}{\rmdefault}{\mddefault}{\updefault}{\color[rgb]{0,0,0}1}%
}}}}
\put(4501,-961){\makebox(0,0)[lb]{\smash{{\SetFigFont{10}{20.4}{\rmdefault}{\mddefault}{\updefault}{\color[rgb]{0,0,0}1}%
}}}}
\put(3601,-961){\makebox(0,0)[lb]{\smash{{\SetFigFont{10}{20.4}{\rmdefault}{\mddefault}{\updefault}{\color[rgb]{0,0,0}1}%
}}}}
\put(14851,-511){\makebox(0,0)[lb]{\smash{{\SetFigFont{10}{20.4}{\rmdefault}{\mddefault}{\updefault}{\color[rgb]{0,0,0}1}%
}}}}
\put(13951,-511){\makebox(0,0)[lb]{\smash{{\SetFigFont{10}{20.4}{\rmdefault}{\mddefault}{\updefault}{\color[rgb]{0,0,0}2}%
}}}}
\put(13051,-511){\makebox(0,0)[lb]{\smash{{\SetFigFont{10}{20.4}{\rmdefault}{\mddefault}{\updefault}{\color[rgb]{0,0,0}4}%
}}}}
\put(12151,-511){\makebox(0,0)[lb]{\smash{{\SetFigFont{10}{20.4}{\rmdefault}{\mddefault}{\updefault}{\color[rgb]{0,0,0}1}%
}}}}
\put(11251,-511){\makebox(0,0)[lb]{\smash{{\SetFigFont{10}{20.4}{\rmdefault}{\mddefault}{\updefault}{\color[rgb]{0,0,0}2}%
}}}}
\put(10351,-511){\makebox(0,0)[lb]{\smash{{\SetFigFont{10}{20.4}{\rmdefault}{\mddefault}{\updefault}{\color[rgb]{0,0,0}2}%
}}}}
\put(9451,-511){\makebox(0,0)[lb]{\smash{{\SetFigFont{10}{20.4}{\rmdefault}{\mddefault}{\updefault}{\color[rgb]{0,0,0}3}%
}}}}
\put(8551,-511){\makebox(0,0)[lb]{\smash{{\SetFigFont{10}{20.4}{\rmdefault}{\mddefault}{\updefault}{\color[rgb]{0,0,0}1}%
}}}}
\put(7651,-511){\makebox(0,0)[lb]{\smash{{\SetFigFont{10}{20.4}{\rmdefault}{\mddefault}{\updefault}{\color[rgb]{0,0,0}2}%
}}}}
\put(6751,-511){\makebox(0,0)[lb]{\smash{{\SetFigFont{10}{20.4}{\rmdefault}{\mddefault}{\updefault}{\color[rgb]{0,0,0}4}%
}}}}
\put(5851,-511){\makebox(0,0)[lb]{\smash{{\SetFigFont{10}{20.4}{\rmdefault}{\mddefault}{\updefault}{\color[rgb]{0,0,0}1}%
}}}}
\put(4951,-511){\makebox(0,0)[lb]{\smash{{\SetFigFont{10}{20.4}{\rmdefault}{\mddefault}{\updefault}{\color[rgb]{0,0,0}2}%
}}}}
\put(4051,-511){\makebox(0,0)[lb]{\smash{{\SetFigFont{10}{20.4}{\rmdefault}{\mddefault}{\updefault}{\color[rgb]{0,0,0}2}%
}}}}
\put(3151,-511){\makebox(0,0)[lb]{\smash{{\SetFigFont{10}{20.4}{\rmdefault}{\mddefault}{\updefault}{\color[rgb]{0,0,0}3}%
}}}}
\put(15301,-61){\makebox(0,0)[lb]{\smash{{\SetFigFont{10}{20.4}{\rmdefault}{\mddefault}{\updefault}{\color[rgb]{0,0,0}2}%
}}}}
\put(14401,-61){\makebox(0,0)[lb]{\smash{{\SetFigFont{10}{20.4}{\rmdefault}{\mddefault}{\updefault}{\color[rgb]{0,0,0}1}%
}}}}
\put(13501,-61){\makebox(0,0)[lb]{\smash{{\SetFigFont{10}{20.4}{\rmdefault}{\mddefault}{\updefault}{\color[rgb]{0,0,0}7}%
}}}}
\put(12601,-61){\makebox(0,0)[lb]{\smash{{\SetFigFont{10}{20.4}{\rmdefault}{\mddefault}{\updefault}{\color[rgb]{0,0,0}3}%
}}}}
\put(11701,-61){\makebox(0,0)[lb]{\smash{{\SetFigFont{10}{20.4}{\rmdefault}{\mddefault}{\updefault}{\color[rgb]{0,0,0}1}%
}}}}
\put(10801,-61){\makebox(0,0)[lb]{\smash{{\SetFigFont{10}{20.4}{\rmdefault}{\mddefault}{\updefault}{\color[rgb]{0,0,0}3}%
}}}}
\put(9901,-61){\makebox(0,0)[lb]{\smash{{\SetFigFont{10}{20.4}{\rmdefault}{\mddefault}{\updefault}{\color[rgb]{0,0,0}5}%
}}}}
\put(9001,-61){\makebox(0,0)[lb]{\smash{{\SetFigFont{10}{20.4}{\rmdefault}{\mddefault}{\updefault}{\color[rgb]{0,0,0}2}%
}}}}
\put(8101,-61){\makebox(0,0)[lb]{\smash{{\SetFigFont{10}{20.4}{\rmdefault}{\mddefault}{\updefault}{\color[rgb]{0,0,0}1}%
}}}}
\put(7201,-61){\makebox(0,0)[lb]{\smash{{\SetFigFont{10}{20.4}{\rmdefault}{\mddefault}{\updefault}{\color[rgb]{0,0,0}7}%
}}}}
\put(6301,-61){\makebox(0,0)[lb]{\smash{{\SetFigFont{10}{20.4}{\rmdefault}{\mddefault}{\updefault}{\color[rgb]{0,0,0}3}%
}}}}
\put(5401,-61){\makebox(0,0)[lb]{\smash{{\SetFigFont{10}{20.4}{\rmdefault}{\mddefault}{\updefault}{\color[rgb]{0,0,0}1}%
}}}}
\put(4501,-61){\makebox(0,0)[lb]{\smash{{\SetFigFont{10}{20.4}{\rmdefault}{\mddefault}{\updefault}{\color[rgb]{0,0,0}3}%
}}}}
\put(3601,-61){\makebox(0,0)[lb]{\smash{{\SetFigFont{10}{20.4}{\rmdefault}{\mddefault}{\updefault}{\color[rgb]{0,0,0}5}%
}}}}
\put(14851,389){\makebox(0,0)[lb]{\smash{{\SetFigFont{10}{20.4}{\rmdefault}{\mddefault}{\updefault}{\color[rgb]{0,0,0}1}%
}}}}
\put(13951,389){\makebox(0,0)[lb]{\smash{{\SetFigFont{10}{20.4}{\rmdefault}{\mddefault}{\updefault}{\color[rgb]{0,0,0}3}%
}}}}
\put(13051,389){\makebox(0,0)[lb]{\smash{{\SetFigFont{10}{20.4}{\rmdefault}{\mddefault}{\updefault}{\color[rgb]{0,0,0}5}%
}}}}
\put(12151,389){\makebox(0,0)[lb]{\smash{{\SetFigFont{10}{20.4}{\rmdefault}{\mddefault}{\updefault}{\color[rgb]{0,0,0}2}%
}}}}
\put(11251,389){\makebox(0,0)[lb]{\smash{{\SetFigFont{10}{20.4}{\rmdefault}{\mddefault}{\updefault}{\color[rgb]{0,0,0}1}%
}}}}
\put(10351,389){\makebox(0,0)[lb]{\smash{{\SetFigFont{10}{20.4}{\rmdefault}{\mddefault}{\updefault}{\color[rgb]{0,0,0}7}%
}}}}
\put(9451,389){\makebox(0,0)[lb]{\smash{{\SetFigFont{10}{20.4}{\rmdefault}{\mddefault}{\updefault}{\color[rgb]{0,0,0}3}%
}}}}
\put(8551,389){\makebox(0,0)[lb]{\smash{{\SetFigFont{10}{20.4}{\rmdefault}{\mddefault}{\updefault}{\color[rgb]{0,0,0}1}%
}}}}
\put(7651,389){\makebox(0,0)[lb]{\smash{{\SetFigFont{10}{20.4}{\rmdefault}{\mddefault}{\updefault}{\color[rgb]{0,0,0}3}%
}}}}
\put(6751,389){\makebox(0,0)[lb]{\smash{{\SetFigFont{10}{20.4}{\rmdefault}{\mddefault}{\updefault}{\color[rgb]{0,0,0}5}%
}}}}
\put(5851,389){\makebox(0,0)[lb]{\smash{{\SetFigFont{10}{20.4}{\rmdefault}{\mddefault}{\updefault}{\color[rgb]{0,0,0}2}%
}}}}
\put(4951,389){\makebox(0,0)[lb]{\smash{{\SetFigFont{10}{20.4}{\rmdefault}{\mddefault}{\updefault}{\color[rgb]{0,0,0}1}%
}}}}
\put(4051,389){\makebox(0,0)[lb]{\smash{{\SetFigFont{10}{20.4}{\rmdefault}{\mddefault}{\updefault}{\color[rgb]{0,0,0}7}%
}}}}
\put(3151,389){\makebox(0,0)[lb]{\smash{{\SetFigFont{10}{20.4}{\rmdefault}{\mddefault}{\updefault}{\color[rgb]{0,0,0}3}%
}}}}
\put(14851,1289){\makebox(0,0)[lb]{\smash{{\SetFigFont{10}{20.4}{\rmdefault}{\mddefault}{\updefault}{\color[rgb]{0,0,0}1}%
}}}}
\put(13951,1289){\makebox(0,0)[lb]{\smash{{\SetFigFont{10}{20.4}{\rmdefault}{\mddefault}{\updefault}{\color[rgb]{0,0,0}1}%
}}}}
\put(15301,839){\makebox(0,0)[lb]{\smash{{\SetFigFont{10}{20.4}{\rmdefault}{\mddefault}{\updefault}{\color[rgb]{0,0,0}1}%
}}}}
\put(14401,839){\makebox(0,0)[lb]{\smash{{\SetFigFont{10}{20.4}{\rmdefault}{\mddefault}{\updefault}{\color[rgb]{0,0,0}2}%
}}}}
\put(13501,839){\makebox(0,0)[lb]{\smash{{\SetFigFont{10}{20.4}{\rmdefault}{\mddefault}{\updefault}{\color[rgb]{0,0,0}2}%
}}}}
\put(12601,839){\makebox(0,0)[lb]{\smash{{\SetFigFont{10}{20.4}{\rmdefault}{\mddefault}{\updefault}{\color[rgb]{0,0,0}3}%
}}}}
\put(11701,839){\makebox(0,0)[lb]{\smash{{\SetFigFont{10}{20.4}{\rmdefault}{\mddefault}{\updefault}{\color[rgb]{0,0,0}1}%
}}}}
\put(10801,839){\makebox(0,0)[lb]{\smash{{\SetFigFont{10}{20.4}{\rmdefault}{\mddefault}{\updefault}{\color[rgb]{0,0,0}2}%
}}}}
\put(9901,839){\makebox(0,0)[lb]{\smash{{\SetFigFont{10}{20.4}{\rmdefault}{\mddefault}{\updefault}{\color[rgb]{0,0,0}4}%
}}}}
\put(9001,839){\makebox(0,0)[lb]{\smash{{\SetFigFont{10}{20.4}{\rmdefault}{\mddefault}{\updefault}{\color[rgb]{0,0,0}1}%
}}}}
\put(8101,839){\makebox(0,0)[lb]{\smash{{\SetFigFont{10}{20.4}{\rmdefault}{\mddefault}{\updefault}{\color[rgb]{0,0,0}2}%
}}}}
\put(7201,839){\makebox(0,0)[lb]{\smash{{\SetFigFont{10}{20.4}{\rmdefault}{\mddefault}{\updefault}{\color[rgb]{0,0,0}2}%
}}}}
\put(6301,839){\makebox(0,0)[lb]{\smash{{\SetFigFont{10}{20.4}{\rmdefault}{\mddefault}{\updefault}{\color[rgb]{0,0,0}3}%
}}}}
\put(5401,839){\makebox(0,0)[lb]{\smash{{\SetFigFont{10}{20.4}{\rmdefault}{\mddefault}{\updefault}{\color[rgb]{0,0,0}1}%
}}}}
\put(4501,839){\makebox(0,0)[lb]{\smash{{\SetFigFont{10}{20.4}{\rmdefault}{\mddefault}{\updefault}{\color[rgb]{0,0,0}2}%
}}}}
\put(3601,839){\makebox(0,0)[lb]{\smash{{\SetFigFont{10}{20.4}{\rmdefault}{\mddefault}{\updefault}{\color[rgb]{0,0,0}4}%
}}}}
\put(13051,1289){\makebox(0,0)[lb]{\smash{{\SetFigFont{10}{20.4}{\rmdefault}{\mddefault}{\updefault}{\color[rgb]{0,0,0}1}%
}}}}
\put(12151,1289){\makebox(0,0)[lb]{\smash{{\SetFigFont{10}{20.4}{\rmdefault}{\mddefault}{\updefault}{\color[rgb]{0,0,0}1}%
}}}}
\put(11251,1289){\makebox(0,0)[lb]{\smash{{\SetFigFont{10}{20.4}{\rmdefault}{\mddefault}{\updefault}{\color[rgb]{0,0,0}1}%
}}}}
\put(3151,1289){\makebox(0,0)[lb]{\smash{{\SetFigFont{10}{20.4}{\rmdefault}{\mddefault}{\updefault}{\color[rgb]{0,0,0}1}%
}}}}
\put(5851,1289){\makebox(0,0)[lb]{\smash{{\SetFigFont{10}{20.4}{\rmdefault}{\mddefault}{\updefault}{\color[rgb]{0,0,0}1}%
}}}}
\put(4951,1289){\makebox(0,0)[lb]{\smash{{\SetFigFont{10}{20.4}{\rmdefault}{\mddefault}{\updefault}{\color[rgb]{0,0,0}1}%
}}}}
\put(4051,1289){\makebox(0,0)[lb]{\smash{{\SetFigFont{10}{20.4}{\rmdefault}{\mddefault}{\updefault}{\color[rgb]{0,0,0}1}%
}}}}
\put(6751,1289){\makebox(0,0)[lb]{\smash{{\SetFigFont{10}{20.4}{\rmdefault}{\mddefault}{\updefault}{\color[rgb]{0,0,0}1}%
}}}}
\put(7651,1289){\makebox(0,0)[lb]{\smash{{\SetFigFont{10}{20.4}{\rmdefault}{\mddefault}{\updefault}{\color[rgb]{0,0,0}1}%
}}}}
\put(8551,1289){\makebox(0,0)[lb]{\smash{{\SetFigFont{10}{20.4}{\rmdefault}{\mddefault}{\updefault}{\color[rgb]{0,0,0}1}%
}}}}
\put(9451,1289){\makebox(0,0)[lb]{\smash{{\SetFigFont{10}{20.4}{\rmdefault}{\mddefault}{\updefault}{\color[rgb]{0,0,0}1}%
}}}}
\thinlines
{\color[rgb]{0,0,0}\put(2926,1289){\line( 1,-1){2587}}
\put(5513,-1298){\line( 1, 1){3149.500}}
\put(8663,1851){\line( 1,-1){3149.500}}
\put(11813,-1298){\line( 1, 1){3149.500}}
\put(14963,1851){\line( 1,-1){1012.500}}
}%
\end{picture}%

\begin{notation}
We label the elements in a frieze by using couples of indices $(i,j)\in\Z\times\Z$ such that $i\leq j \leq i+m-1$, where $m$ is the width. By periodicity the indices are often considered modulo $n=m+3$. 
We denote by $I\subset \Z^{2}$ the set of indices.
When representing a frieze in the plane, the first index $i$ remains constant on a diagonal directed South-East, and the second index $j$ constant on a  North-East diagonal.
\begin{equation}\label{label}
 \begin{array}{lcccccccccccccccccccccccc}
\cdots&&1&& 1&&1&&1&&1&&1&& \cdots\\[4pt]
&e_{1,1}&&e_{2,2}&&\cdots&&e_{i,i}&&\cdots&&e_{n,n}&&e_{1,1}&&
 \\[4pt]
\cdots&&e_{1,2}&&e_{2,3}&&\cdots&&e_{i,i+1}&&\cdots&&e_{n,1}&&\cdots&\
 \\[4pt]
 &e_{n,2}&&e_{1,3}&&e_{2,4}&&\cdots&&e_{i,i+2}&&\cdots&&e_{n,2}&&\\[4pt]
\cdots&&\cdots&&\cdots&&\cdots&&\cdots&&\cdots&&\cdots&&\cdots\\[4pt]
&1&&1&&1&&1&&1&&1&&1&&
\end{array}
\end{equation}
By extension we set $e_{i,i-1}=e_{i,i+m}=1$ and $e_{i,i-2}=e_{i,i+m+1}=0$, for all $i$.

From now on, friezes are considered as evaluations $e: I \to \A$, where $\A$ is a commutative ring with unit.
Let us stress on that the frieze $(e_{i,j})$ and the frieze $(e'_{i,j})$ related by $e'_{i,j};=e_{i+1,j+1}$ 
have the same representations in the plane but are considered as two different friezes since the mappings $ I \to \A$ are different.
\end{notation}

\subsection{Linear recurrence relations}\label{recRelcox}

A key feature of Coxeter frieze patterns is that the diagonals satisfy linear recurrence relations with coefficients given by the entries of the first row of the pattern. 
For a frieze of width $m=n-3$, we denote by $a_{1}, a_{2}, \ldots, a_{n}$ the cycle of $n$ consecutive entries on the first row, so that $e_{i,i}=a_{i}$.

\begin{prop}[\cite{Cox}]\label{propRec}
For any fixed $j$, the sequence  of numbers $V_{i}:=e_{j,i}$ along the $j$-th South-East diagonal satisfies,
\begin{equation}\label{recRel}
V_{i}=a_{i}V_{i-1}-V_{i-2}, \; 
\end{equation}
 for all $i$.
\end{prop}

This statement is easy to establish using the unimodular rule, and is a key point in the proof of Theorem \ref{thmGlide}, 
see also \cite{CoRi}.
This also makes connections between friezes and continued fractions \cite{Cox}.

\subsection{Polynomial entries as continuants}\label{laurent}

Consider a frieze of width $m=n-3$ with first row consisting in the cyclic sequence $a_{1}, a_{2}, \ldots, a_{n}$ (so that $e_{i,i}=a_{i}$).
$$
 \begin{array}{lcccccccccccccccccccccccc}
1&&1&& 1&&1&&1&&1&& \\[4pt]
&a_{1}&&a_{2}&&\cdots&&a_{n}&&a_{1}&&
 \\[4pt]
\cdots&&a_{1}a_{2}-1&&a_{2}a_{3}-1&&\cdots&&a_{n}a_{1}-1&&\cdots&\
 \\[6pt]
&\cdots&&\cdots&&\cdots&&\cdots&&\cdots&&\\[4pt]
1&&1&&1&&1&&1&&1
\end{array}
$$
If one uses the unimodular rule in the frieze to compute the values rows after rows,
one would expect to obtain rational functions in $a_{i}$'s. 
However, the entries are actually polynomials in $a_{i}$'s (this is a consequence of Proposition \ref{propRec}). 

\begin{thm}[\cite{Cox}]\label{thmPoly}
All entries in the frieze are polynomials in the entries $a_{i}$'s of the first row; explicit expressions are given by the following determinants:
\begin{equation} \label{matdet}
e_{i,j}=\left|\begin{array}{cccccccc}
a_{i}&1&&&\\\
1&a_{i+1}&1&&\\
&\ddots&\ddots&\ddots&\\
&&1&a_{j-1}&1\\
&&&1&a_{j}
\end{array}\right|.
\end{equation}
\end{thm}

An arbitrary $n$-periodic sequence $(a_i)$ does not define a frieze pattern with $e_{i,i}:=a_{i}$, 
for all $i$. 
There is no guarantee that
the bottom boundary condition, $e_{i,i+n-3}=1$ will be satisfied.
Three polynomial equations in $a_{i}$'s have to be satisfied in order to define a frieze. 
These equations can be written in terms of the determinants \eqref{matdet}, cf Theorem \ref{eqfri}.

\begin{rem}
The determinants \eqref{matdet} appear in the theory of \textit{continuants}, see \cite{Mui}. 
They are also a first example of  \textit{Andr\'e's determinants}
used to solve linear finite difference equations, \cite{And}, \cite{Jor}.
Note also that in the case of constant coefficients $a_{i}=2x$, the determinant \eqref{matdet} of order  $k$ defines
 the $k$-th Chebyshev polynomial of 2nd kind $U_{k}(x)$.
\end{rem}

\subsection{Laurent phenomenon}
Given a frieze of width $m$,  denote by $x_{0},x_{1}, \ldots,x_{m+1}$ the entry on the $0$-th South-East diagonal (note that $x_{0}=x_{m+1}=1$). 
$$
 \begin{array}{lcccccccccccccccccccccccc}
&1&&1&& 1&&1&&1&&1&&\cdots\\[4pt]
\cdots&&x_{1}&&e_{1,1}&&e_{2,2}&&\cdots&&e_{m,m}&&x_{m}&&
 \\[4pt]
&&&x_{2}&&e_{1,2}&&\cdots&&e_{m-1,m}&&x_{m-1}&&\
 \\[4pt]
&& &&\ddots&&\cdots&&\cdots&&\iddots&&\\[4pt]
& &&&&x_{m}&&e_{1,m}&&x_{1}&&\\[4pt]
\cdots&&1&&1&&1&&1&&1&&\cdots&&
\end{array}
$$
The unimodular rule allows us to compute the rest of the frieze diagonal after diagonal.
One expects to express the entries as rational functions in $x_{i}$'s. Surprisingly all the entries simplify to Laurent polynomials.

\begin{thm}[\cite{CoRi}]\label{thmLaurent}
Entries in a frieze are Laurent polynomials, with positive integer coefficients, in the entries $x_{i}$, $1\leq i \leq m$, placed on a diagonal. Furthermore, one has the explicit formula\begin{equation}
e_{i,j}=x_{i-1}x_{j+1}\left(\frac{1}{x_{i-1}x_{i}}+\frac{1}{x_{i}x_{i+1}}+\cdots+\frac{1}{x_{j}x_{j+1}}\right).\end{equation}
\end{thm}

\begin{rem}\label{remLaurent}
The statement that entries are Laurent polynomials is not formulated in \cite{Cox}, but is easy to deduce from the results of \textit{cite loc}.
Indeed, from Proposition \ref{propRec} one gets $a_{i}=\frac{x_{i-1}+x_{i+1}}{x_{i}}$, $1\leq i\leq m$ (formula given in \S6 and \S7 of \cite{Cox}), then using Theorem ~\ref{thmPoly} one can express all the entries as Laurent polynomials, but this does not ensure the positivity of the coefficients. 
This phenomenon of simplification of the rational expressions is known as \textit{Laurent phenomenon} and occurs in a more general framework
  \cite{FZ1}, \cite{FZLaurent}. 
Using this general framework one can improve the statement of Theorem~ \ref{thmLaurent}:
 \begin{thm}\label{thmLaurentGen}
Entries in a frieze are Laurent polynomials, with positive integer coefficients, in the entries $x_{i}$, $1\leq i \leq m$, placed in any zig-zag shape in the frieze.\end{thm}
Here ``zig-zag shape'' means piecewise linear path from top to bottom where $x_{i+1}$ is placed immediately at the right or at the left under $x_i$
(without necessary alternating right and left).\end{rem}

\begin{ex}\label{friseA3}
Laurent polynomials obtained in a frieze of width 3:
$$
 \begin{array}{ccccccccccccccccccc}
&&1&&1&& 1&&1&&\cdots \\[4pt]
&\cdots&&x_{1}&&\frac{1+x_{2}+x_{1}x_{3}}{x_{1}x_{2}}&&\frac{1+x_{2}}{x_{3}}&&x_{3}&&
 \\[4pt]
&&x_{2}&&\frac{1+x_{1}x_{3}}{x_{2}}&&\frac{(1+x_{2})^{2}+x_{1}x_{3}}{x_{1}x_{2}x_{3}}&&x_{2}&&\cdots
 \\[4pt]
 &\cdots&&x_{3}&&\frac{1+x_{2}+x_{1}x_{3}}{x_{2}x_{3}}&&\frac{1+x_{2}}{x_{1}}&&x_{1}&&
 \\[4pt]
&&1&&1&&1&&1&&\cdots
\end{array}
$$
\end{ex}

\subsection{From infinite friezes to the variety of tame friezes}\label{inffri}

In this section, we explain the structure of algebraic variety on the set of Coxeter's friezes.

The following idea is used in \cite{MGOTaif}.
We consider the formal infinite frieze pattern $F(A_i)_{i\in \Z}$, where $(A_i)_{i\in \Z}$ is a sequence of indeterminates, placed on the first row:
$$
 \begin{array}{lcccccccccccccccccccccccc}
&1&&1&& 1&&1&&1&&1&& \\[4pt]
\cdots&&A_{1}&&A_{2}&&\cdots&&A_{n}&&A_{n+1}&&\cdots
 \\[4pt]
&\cdots&&A_{1}A_{2}-1&&A_{2}A_{3}-1&&\cdots&&A_{n}A_{n+1}-1&&\cdots&\
 \\[6pt]
&&\cdots&&\cdots&&\cdots&&\cdots&&\cdots&&\\[4pt]
\end{array}
$$
The entries in the frieze are computed row by row using the unimodular rule. 
The computations are \textit{a priori} made in the fractions field $\Q(A_i, i\in \Z)$, but similarly to Theorem~\ref{thmPoly}, one shows the entries are
actually in the polynomial ring $\Z[A_i, i\in \Z]$. In particular the frieze $F(A_i)$ is well defined from its first row of indeterminates.
One can show that the entries in the frieze $F(A_i)$ can be computed diagonal by diagonal using recurrence  relations of type \eqref{recRel}, or by direct computation of determinants of 
type~\eqref{matdet}.

For a sequence of numbers $(a_i)_{i\in \Z}$ in any unital commutative ring, we define the infinite frieze $F(a_i)$ from the formal frieze $F(A_i)$ by evaluating all the entries at $A_i=a_i$, $i\in \Z$.

\begin{defn}
We say that the frieze $F(a_i)$ is \textit{closed of width} $m$ if the $(m+1)$-th row is a row of 1's and the $(m+2)$-th row is a row of 0's. 
\end{defn}

\begin{ex}\label{wrong_per}The following frieze is a closed frieze of width 2 if and only if $x=-1$.
$$
 \begin{array}{ccccccccccccccccccccc}
\cdots&&1&& 1&&1&&1&&1&&1&&1&&1&&
 \\[4pt]
&-1&&-1&&-1&&-1-x&&0&&x&&-1&&-1&& \cdots
 \\[4pt]
\cdots&&0&&0&&x&&-1&&-1&&-1-x&&0&&0
 \\[4pt]
&1&&1&&1&&1&&1&&1&&1&&1&& \cdots
\end{array}
$$
Indeed, in the 4th row of the frieze $F(A_i)$ one has the entry $e_{1,4}=A_1A_2A_3A_4-A_1A_2-A_1A_4-A_3A_4+1$. 
If one evaluates $F(A_i)$ with $A_1=A_2=A_3=-1$ and $A_4=-1-x$ one obtains on the fourth row $e_{1,4}=-1-x$.  Hence $x=-1$ is a necessary condition for the above frieze to be closed of width 2. Then one checks that it is also sufficient. 
\end{ex}

\begin{rem}

Friezes coming from an evaluation of $F(A_i)$ are generic in a wide sense.
The evaluation allows us to have a well-defined frieze from its first row even if the rows contain 0 entries.
Such friezes are all tame (see the discussion in introduction to \S\ref{FrCC} or Definition \ref{tame} below).
\end{rem}

\begin{thm}[\cite{Cox}, \cite{MGOST}]\label{eqfri}
The frieze $F(a_i)$ is closed of width $m$ if and only if the sequence $(a_i)_i$ is $(m+3)$-periodic and satisfies
$$
0
=\left|\begin{array}{cccccccc}
a_{1}&1&&\\\
1&a_{2}&1&\\
&\ddots&\ddots&\ddots\\
&&1&a_{m+2}\\
\end{array}\right|
=\left|\begin{array}{cccccccc}
a_{2}&1&&\\\
1&a_{3}&1&\\
&\ddots&\ddots&\ddots\\
&&1&a_{m+3}\\
\end{array}\right|,\;\;
1=\left|\begin{array}{cccccccc}
a_{2}&1&&\\\
1&a_{3}&1&\\
&\ddots&\ddots&\ddots\\
&&1&a_{m+2}\\
\end{array}\right|.
$$\end{thm}

This result was obtained in \cite[p307]{Cox} when $a_{i}$ are positive real numbers. The case of arbitrary coefficients is deduced from \cite[\S3]{MGOST} (particular case $k=1$).
In the sequel, we will mainly consider friezes over real or complex numbers. 

When the frieze is closed of width $m$, the infinite array $F(a_{i})$ has $(m+3)$-antiperiodic diagonals (it is a consequence of the periodicity of the coefficients and the recurrence relations along the diagonals), therefore only the first $m$ rows are relevant.  Closed friezes are equivalent to (tame) Coxeter's friezes as defined in the introduction of \S \ref{FrCC}.

In conclusion, the set of real or complex (tame) Coxeter's friezes of width $m=n-3$ is an algebraic subvariety of $\R^n$ or $\C^n$ defined 
by the three polynomial equations of Theorem \ref{eqfri}. 

\subsection{Superperiodic difference equations of order 2}\label{superper}

We consider linear difference equation of order 2, of the form
\begin{equation}\label{eqOrd2}
V_i=a_iV_{i-1}-V_{i-2}
\end{equation}
where the $a_i$, ${i\in \Z}$ are coefficients and $V_i$, ${i\in \Z}$ the unknowns.
This equation is sometimes mentioned in the literature as ``discrete Hill equation'' or ``discrete Sturm-Liouville equation'' or 
``discrete 1-dimensional Shr\"odinger equation''.

Following \cite{Kri} and \cite{MGOST}, such an equation is called \textit{$n$-superperiodic} if all its solutions $(V_i)$ satisfy
$$
V_{i+n}=-V_i
$$
for all ${i\in \Z}$, cf Definition \ref{superdef}.

One shows that the set of superperiodic equations is defined by the same polynomial equations in the coefficients $a_i$'s as the 
closed friezes. In other words, one has the following identification.

\begin{thm}[\cite{MGOTaif}]\label{IsoThm1}
The space of tame Coxeter's friezes of width $m$ is isomorphic, as algebraic variety, to the space of $(m+3)$-superperiodic equations of type \eqref{eqOrd2}.
\end{thm}

This result was also implicitly obtained in \cite{CuHe}.

Note that  superperiodic equations necessarily have periodic coefficients, since the coefficients can be recovered from the solutions.
In the correspondence between friezes and equations of order 2, the entries in the first row of the frieze coincide with the coefficients of the equation, and pairs of consecutive diagonals in the frieze with solutions of the equation from different initial values. We will give more details in the next section.

\subsection{Moduli space of points on the projective line}\label{M0n}

Cross ratios of points on the circle and frieze patterns were already linked in \cite{Cox}. 
Here, we give a different version of such a link.
We explain how the results of \S\ref{width2} (case $n=5$) generalize to any odd $n=m+3$. 
We will show that frieze patterns provide natural coordinate systems
on the (real or complex) moduli space $\cM_{0,n}$ and also on the bigger space
$$\widehat \cM_{0,n}:=\{p_i\in \mathbb{P}^1, i\in \Z, \,p_{i+n}=p_i,\, p_{i}\not=p_{i+1}\}/ \PGL_2.$$


\begin{thm}[\cite{MGOTaif}]
\label{IsoThm2}
If $n$ is odd then the space $\widehat \cM_{0,n}$
is isomorphic to the space of tame Coxeter's frieze patterns of width $n-3$.
\end{thm}

We assume that the spaces are considered over the field of real or complex numbers. In the sequel we will work over $\C$ (the case over $\R$ is similar but requires more care regarding the orientation).
The above theorem is stated in \cite{MGOTaif}, and was proved in a more general form in \cite{MGOST}.
We explain below in details the explicit construction of the isomorphism.
The construction also uses an idea of \cite{OST1}, and general ideas of projective geometry \cite{OvTa}.
\\

\textbf{Construction of the isomorphisms of Theorem \ref{IsoThm1} and \ref{IsoThm2}}.
We fix an odd integer  $n$.
Given an element $p$ of $\widehat \cM_{0,n}$, we explain here how we construct 
a closed frieze $f(p)$ of width $n-3$ and an $n$-superperiodic equation $V(p)$.

First, we choose an $n$-periodic sequence $(p_{i})$ of points in $\bbP^{1}$ representing $p$ modulo $\PGL_{2}$.

\begin{lem}
There exists a unique, up to a sign, lift of the sequence $(p_{i})$
to a sequence $(V_{i})$ of vectors  in $\C^2$, such that $V_{i+n}=-V_{i}$ and
$
\det(V_{i+1}, V_{i})=1,
$
for all $i$.
\end{lem}

\noindent
\begin{proof}
Consider an arbitrary lift of the points $(p_{0},\ldots, p_{n-1})$ to vectors $(\widetilde V_{0},\ldots ,\widetilde V_{n-1})$, and extend the sequence by antiperiodicity.
Since $p_{i}\not=p_{i+1}$, we have: $\det(\widetilde V_{i+1},\widetilde V_{i})\not=0$ for all $i$.
We wish to rescale: $V_{i}=\l_{i}\widetilde V_{i}$, so that $\det(V_{i+1}, V_{i})=1$ for all $i$.
This leads to the following system of $n$ equations:
$$
\begin{array}{rcl}
\l_{i+1}\l_{i}&=&1/\det(\widetilde V_{i+1},\widetilde V_{i}),
\qquad 
i=0,\ldots,n-2,\\[4pt]
\l_{0}\l_{n-1}&=&1/\det(\widetilde V_{n-1},\widetilde V_{0}).
\end{array}
$$
This system admits a unique solution (up to a sign) if and only if $n$ is odd.
Hence the lemma.
\end{proof}

The vectors of the sequence $(V_{i})$ defined in the above lemma satisfy relations of the form 
$V_i=a_iV_{i-1}-V_{i-2}$ with periodic coefficients $a_{i+n}=a_i$.
Moreover the coefficients $(a_{i})$ are determined by~$p$, i.e. independent of the choice of the representative $(p_{i})$. 
We denote by $V(p)$ the corresponding equation \eqref{eqOrd2}. 
The equation $V(p)$ is superperiodic since the two components of the vectors of the sequence $(V_i)$ provide two independent antiperiodic solutions.

In addition, modulo the action of $\GL_2$, we can normalize the lifted sequence of points so that $V_{0}=(0,1)$ and $V_{n-1}=(1,0)$.

The frieze $f(p)$ is defined using the coefficients $(a_i)$ on the first row.
Moreover, the normalized sequence of lifted points  $(V_{i})$ appears in the frieze (and also determines the frieze) as a pair of consecutive diagonals.
In the frieze $f(p)$, one has 
$$
e_{1,i}=V_{i}^{(2)},\;\;e_{2,i}=V_{i}^{(1)}.\;
$$
where $(V_{i}^{(1)},V_{i}^{(2)})$ are the components of $V_{i}$.
One obtains the following picture for $f(p)$:
$$
 \begin{array}{lccccccccccccccccccccccccccc}
&1&&1&& 1&&1&&1&&1&&\\[4pt]
\cdots&&V_{1}^{(2)}&&V_{2}^{(1)}&&a_{3}&&a_{4}&&a_{5}&&\cdots&&
 \\[4pt]
&&&V_{2}^{(2)}&&\ddots&&c_{4}&&c_{5}&&c_{6}&&
 \\[7pt]
&& &&\ddots&&V_{n-3}^{(1)}&&\cdots&&\cdots&&\cdots\\[4pt]
& &&&&V_{n-3}^{(2)}&&V_{n-2}^{(1)}&&\cdots&&\cdots\\[4pt]
&&\cdots&&1&&1&&1&&1&&1&&&&
\end{array}
$$
The entries $e_{i,j}$ in the frieze can be computed directly using the sequence of vectors $(V_{i})$:
$$
e_{i,j}=\det(V_{j}, V_{i-2}).
$$

Let us mention that the second row of the frieze has an important geometric interpretation: it gives cross ratios associated to $p$. More precisely, one has the following proposition.
\begin{prop}
If $p$ is an element of $\widehat \cM_{0,n}$ represented by a $n$-tuple $(p_{0},\ldots, p_{n-1})$ of points in~$\bbP^{1}$, then
the entries in the 2nd row of the frieze $f(p)$ are
$$
e_{i-1,i}=[p_{i-3},p_{i-2},p_{i-1},p_{i}]=
\frac{(p_{i}-p_{i-3})(p_{i-1}-p_{i-2})}{(p_{i}-p_{i-1})(p_{i-2}-p_{i-3})}.
$$
\end{prop}

\begin{proof}
By Proposition \ref{propRec}, pairs of consecutive diagonals in the frieze represent the same sequence of points modulo $\PGL_{2}$, up to cyclic permutations. 
For every $j$, one has
$$
[p_{i-3},p_{i-2},p_{i-1},p_{i}]=
\left[\frac{e_{j,i-3}}{e_{j-1,i-3}},\,
\frac{e_{j,i-2}}{e_{j-1,i-2}},\,
\frac{e_{j,i-1}}{e_{j-1,i-1}},\,
\frac{e_{j,i}}{e_{j-1,i}}
\right].
$$
Choosing $j=i$, one easily computes
$$
[p_{i-3},p_{i-2},p_{i-1},p_{i}]=\left[\frac{-1}{0}, \frac01, \frac{1}{e_{i-1,i-1}}, \frac{e_{i,i}}{e_{i-1,i}}\right]=
\frac{
\frac{1}{e_{i-1,i-1}}
}{
\frac{e_{i,i}}{e_{i-1,i}}-\frac{1}{e_{i-1,i-1}}
}
=e_{i-1,i}.
$$
\end{proof}

\begin{rem}
When $n$ is odd, a point $p\in \widehat \cM_{0,n}$ is characterized by the sequence of the $n$ cross ratios
$$
c_{i}=[p_{i-3},p_{i-2},p_{i-1},p_{i}].
$$
One can recover the first row of the corresponding frieze $f(p)$ directly from this data by solving for $a_{i}$ in the system of equations
$$
1+c_{i}=a_{i}a_{i+1}, \;1\leq i \leq n
$$
where $i$ is considered modulo $n$. 
When $n$ is odd, the system of equations has two sequences of solutions $(a_{i})$ 
with opposite signs. Exactly one of these sequences defines a frieze. 
\end{rem}


\section{Friezes and quivers}\label{FrQuiv}

A first direction to generalize Coxeter's notion of frieze pattern is to define frieze as functions on a repetition quiver. Repetition quivers are classical objects in the theory of representations of quivers. 

Two main alternative conditions may be imposed to the functions on the repetition quivers in order to
 define a frieze.
 One condition, that we call \textit{multiplicative rule}, is a natural generalization of Coxeter's unimodular rule.  
 The other condition is an additive analogue, that we call \textit{additive rule}, which naturally appears in the theory of representations of quivers. 
 
 Multiplicative friezes on repetition quivers were introduced in connection with the recent theory of cluster algebras, \cite{CaCh}, \cite{ARS}, and many results are obtained within this framework \cite{BaMa}, \cite{BaMa2}, \cite{AsDu}, \cite{ADSS}, \cite{KeSc}, \cite{BuDu}, \cite{Ess}.
 
We define the main notions and give the main results that we need from quiver representations and from the theory of cluster algebras, details can be found in classical textbooks or surveys on the subjects, see e.g. \cite{Gab}, \cite{AuReS}, \cite{ASS}, \cite{Schi}, and \cite{Kel}, \cite{Rei}, \cite{GSV}, \cite{Marsh}.

\subsection{Repetition quiver}
Let $\Qc$ be a quiver, i.e. an oriented graph. The set of vertices $\Qc_0$ and the set of arrows $\Qc_1$
are assumed to be finite. We denote by $n$ the cardinality of $\Qc_0$ and often identify this set with the elements $\{1,2,\ldots, n\}$.

The quiver is said to be \textit{acyclic} if it has no oriented cycle.

We denote by $\Qc^{\op}$ the quiver with opposite orientation, i.e. all arrows of $\Qc$ are reversed.

From an acyclic quiver $\Qc$ one constructs the \textit{repetition quiver}
$\Z\Qc$ \cite{Rie}. 
The vertices of $\Z\Qc$ are the couples $(m,i)$, $m\in \Z$,  $i\in \Qc_0$,  and 
for every arrow $i\longrightarrow j$ in $\Qc_1$ one draws the arrows
$$
(m,i) \longrightarrow (m,j)\quad \text{ and } \quad (m,j) \longrightarrow (m+1,i),
$$
for all $m\in \Z$. All the arrows of $\Z\Qc$ are obtained this way.

Note that if $\Qc$ and $\Qc'$ have same underlying unoriented graph, then they have same repetition quivers but with different labels on the vertices.
In particular, one has $\Z\Qc\simeq\Z\Qc^{\op}$.

We denote by $\tau$ the translation on the vertices of $\Z\Qc$ defined by 
$$\tau: (m,i)\mapsto (m-1,i).$$

Similarly, one can define the repetition quiver $\N\Qc$, which is identified with the full subquiver of $\Z\Qc$ with vertices $(m,i)$, 
$i \in \Qc_0$, $m\in \N$.

A copy of $\Qc$ in $\Z\Qc$, with vertices $(m,i)$, $i \in \Qc_0$, for a fixed $m$, is called a \textit{slice} of  $\Z\Qc$.

\begin{ex}\label{exADE}
The Dynkin quivers of type $\mA, \mD, \mE$, i.e. those for which the underlying unoriented graph is a Dynkin diagram of one of these types, 
play an important role in the theory of friezes. Below we fix the labels of the vertices of the Dynkin diagram that we will use throughout the paper.
We choose an orientation so that an edge  $\xymatrix {i\ar@{-}[r]&j }$ is oriented from the smaller index to the larger one.
This notation agrees with the one of \cite{Gab}.\\
\indent
1) Case $\Qc=\mA_{n}$:\\
 \xymatrix  @!0 @R=1.8em @C=1.8pc 
 {
 &
&&
&&
\overset{n}{\bullet}&&
\overset{}{}\ar@{.}[rd]&&
\overset{}{\bullet}\ar[rd]&&
\overset{}{\bullet}\ar[rd]&&
\overset{0n}{\bullet}\ar[rd]&&
\overset{1n}{\bullet}\ar[rd]&&
\overset{}{\bullet}\ar[rd]&&
\overset{}{\bullet}\ar@{.}[rd]&& 
\\
& &
&&
\overset{}{\bullet}\ar[ru]&&
\overset{}{}\ar@{.}[rd]&&
\overset{}{\bullet}\ar[rd]\ar[ru]&&
\overset{}{\bullet}\ar[rd]\ar[ru]&&
\overset{}{\bullet}\ar[rd]\ar[ru]&&
\overset{}{\bullet}\ar[rd]\ar[ru]&&
\overset{}{\bullet}\ar[rd]\ar[ru]&&
\overset{}{\bullet}\ar@{.}[rd]\ar[ru]&&
&&
\\
&
&&
\overset{}{\bullet}\ar@{.}[ru]&&
\overset{}{}\ar@{.}[rd]&&
\overset{}{\bullet}\ar[rd]\ar@{.}[ru]&&
\overset{}{\bullet}\ar[rd]\ar@{.}[ru]&&
\overset{}{\bullet}\ar[rd]\ar@{.}[ru]&&
\overset{}{\bullet}\ar[rd]\ar@{.}[ru]&&
\overset{}{\bullet}\ar[rd]\ar@{.}[ru]&&
\overset{}{\bullet}\ar@{.}[rd]\ar@{.}[ru]&&
&&
\\
&&
\overset{2}{\bullet}\ar[ru]&&
\overset{}{}\ar@{.}[rd]&&
\overset{}{\bullet}\ar[rd]\ar[ru]&&
\overset{}{\bullet}\ar[rd]\ar[ru]&&
\overset{02}{\bullet}\ar[rd]\ar[ru]&&
\overset{12}{\bullet}\ar[rd]\ar[ru]&&
\overset{}{\bullet}\ar[rd]\ar[ru]&&
\overset{}{\bullet}\ar@{.}[rd]\ar[ru]&&
&&
&&
\\
\Qc:& 
\overset{1}{\bullet}\ar[ru]&&
\Z\Qc:&&
\overset{}{\bullet}\ar[ru]&&
\overset{}{\bullet}\ar[ru]&&
\overset{01}{\bullet}\ar[ru]&&
\overset{11}{\bullet}\ar[ru]&&
\overset{}{\bullet}\ar[ru]&&
\overset{}{\bullet}\ar[ru]&&
&&
&& 
\\
&&&&&&&&&&&&&&&&&&&\\
}

2) Case $\Qc=\mD_{n}$:\\
 \xymatrix  @!0 @R=1.8em @C=1.8pc 
 {
 &
&&
&&
\overset{n-1}{\bullet}&&
\overset{}{}\ar@{.}[rd]&&
\overset{}{\bullet}\ar[rd]&&
\overset{}{\bullet}\ar[rd]&&
\overset{0n-1}{\bullet}\ar[rd]&&
\overset{1n-1}{\bullet}\ar[rd]&&
\overset{}{\bullet}\ar[rd]&&
\overset{}{\bullet}\ar@{.}[rd]&& 
\\
& &
&&
\overset{}{\bullet}\ar[ru]\ar[r]&\overset{n}{\bullet}&
\overset{}{}\ar@{.}[rd]\ar@{.}[r]&\ar@{.}[r]&
\overset{}{\bullet}\ar[rd]\ar[ru]\ar[r]&\overset{}{\bullet}\ar[r]&
\overset{}{\bullet}\ar[rd]\ar[ru]\ar[r]&\overset{}{\bullet}\ar[r]&
\overset{}{\bullet}\ar[rd]\ar[ru]\ar[r]&\overset{0n}{\bullet}\ar[r]&
\overset{}{\bullet}\ar[rd]\ar[ru]\ar[r]&\overset{1n}{\bullet}\ar[r]&
\overset{}{\bullet}\ar[rd]\ar[ru]\ar[r]&\overset{}{\bullet}\ar[r]&
\overset{}{\bullet}\ar@{.}[rd]\ar[ru]\ar@{.}[r]&&
&&
\\
&
&&
\overset{}{\bullet}\ar@{.}[ru]&&
\overset{}{}\ar@{.}[rd]&&
\overset{}{\bullet}\ar[rd]\ar@{.}[ru]&&
\overset{}{\bullet}\ar[rd]\ar@{.}[ru]&&
\overset{}{\bullet}\ar[rd]\ar@{.}[ru]&&
\overset{}{\bullet}\ar[rd]\ar@{.}[ru]&&
\overset{}{\bullet}\ar[rd]\ar@{.}[ru]&&
\overset{}{\bullet}\ar@{.}[rd]\ar@{.}[ru]&&
&&
\\
&&
\overset{2}{\bullet}\ar[ru]&&
\overset{}{}\ar@{.}[rd]&&
\overset{}{\bullet}\ar[rd]\ar[ru]&&
\overset{}{\bullet}\ar[rd]\ar[ru]&&
\overset{02}{\bullet}\ar[rd]\ar[ru]&&
\overset{12}{\bullet}\ar[rd]\ar[ru]&&
\overset{}{\bullet}\ar[rd]\ar[ru]&&
\overset{}{\bullet}\ar@{.}[rd]\ar[ru]&&
&&
&&
\\
\Qc:& 
\overset{1}{\bullet}\ar[ru]&&
\Z\Qc:&&
\overset{}{\bullet}\ar[ru]&&
\overset{}{\bullet}\ar[ru]&&
\overset{01}{\bullet}\ar[ru]&&
\overset{11}{\bullet}\ar[ru]&&
\overset{}{\bullet}\ar[ru]&&
\overset{}{\bullet}\ar[ru]&&
&&
&& 
\\
&&&&&&&&&&&&&&&&&&&\\
}

3) Case $\Qc=\mE_{n}$, $n=6,7,8$.\\

 \xymatrix  @!0 @R=1.8em @C=1.8pc 
 {
 &
&&
&&
\overset{n-1}{\bullet}&&
\overset{}{}\ar@{.}[rd]&&
\overset{}{\bullet}\ar[rd]&&
\overset{}{\bullet}\ar[rd]&&
\overset{0n-1}{\bullet}\ar[rd]&&
\overset{1n-1}{\bullet}\ar[rd]&&
\overset{}{\bullet}\ar[rd]&&
\overset{}{\bullet}\ar@{.}[rd]&& 
\\
&&
&&
\overset{}{\bullet}\ar[ru]&&
\overset{}{}\ar@{.}[rd]&&
\overset{}{\bullet}\ar[rd]\ar[ru]&&
\overset{}{\bullet}\ar[rd]\ar[ru]&&
\overset{}{\bullet}\ar[rd]\ar[ru]&&
\overset{}{\bullet}\ar[rd]\ar[ru]&&
\overset{}{\bullet}\ar[rd]\ar[ru]&&
\overset{}{\bullet}\ar@{.}[rd]\ar[ru]&&
&&
\\
& 
&&
\overset{}{\bullet}\ar[ru]\ar[r]&\overset{n}{\bullet}&
\overset{}{}\ar@{.}[rd]\ar@{.}[r]&\ar@{.}[r]&
\overset{}{\bullet}\ar[rd]\ar[ru]\ar[r]&\overset{}{\bullet}\ar[r]&
\overset{}{\bullet}\ar[rd]\ar[ru]\ar[r]&\overset{}{\bullet}\ar[r]&
\overset{}{\bullet}\ar[rd]\ar[ru]\ar[r]&\overset{0n}{\bullet}\ar[r]&
\overset{}{\bullet}\ar[rd]\ar[ru]\ar[r]&\overset{1n}{\bullet}\ar[r]&
\overset{}{\bullet}\ar[rd]\ar[ru]\ar[r]&\overset{}{\bullet}\ar[r]&
\overset{}{\bullet}\ar@{.}[rd]\ar[ru]\ar@{.}[r]&&
&&
\\
&&
\overset{2}{\bullet}\ar@{.}[ru]&&
\overset{}{}\ar@{.}[rd]&&
\overset{}{\bullet}\ar[rd]\ar@{.}[ru]&&
\overset{}{\bullet}\ar[rd]\ar@{.}[ru]&&
\overset{02}{\bullet}\ar[rd]\ar@{.}[ru]&&
\overset{12}{\bullet}\ar[rd]\ar@{.}[ru]&&
\overset{}{\bullet}\ar[rd]\ar@{.}[ru]&&
\overset{}{\bullet}\ar@{.}[rd]\ar@{.}[ru]&&
&&
&&
\\
\Qc:& 
\overset{1}{\bullet}\ar[ru]&&
\Z\Qc:&&
\overset{}{\bullet}\ar[ru]&&
\overset{}{\bullet}\ar[ru]&&
\overset{01}{\bullet}\ar[ru]&&
\overset{11}{\bullet}\ar[ru]&&
\overset{}{\bullet}\ar[ru]&&
\overset{}{\bullet}\ar[ru]&&
&&
&& 
\\
}

\end{ex}

\subsection{Friezes on repetition quivers}
A generalized frieze of type $\Qc$ is a function on the repetition quiver
 $$f: \Z \Qc_{} \rightarrow \A,$$
assigning at each vertex of $\Z\Qc$ an element in a fixed commutative ring with unit  $\A$, so that the assigned values satisfy some ``mesh relations'' read out of the oriented graph $\Z\Qc$.

The function $f$ will be called an \textit{additive frieze} if it satisfies for all  $v\in \Z \Qc_0$,
$$
f(\tau v)+f(v)=\sum_{\substack{\alpha\in\Z\Qc_1:\\[2pt] w\overset{\alpha}{\longrightarrow} v}}\;f(w).
$$
The function $f$ will be called a \textit{multiplicative frieze} if it satisfies for all  $v\in \Z \Qc_0$,
$$
f(\tau v)f(v)=1+\prod_{\substack{\alpha\in\Z\Qc_1:\\[2pt] w\overset{\alpha}{\longrightarrow} v}}\;f(w).
$$
Additive friezes are classical objects in Auslander-Reiten theory, more often called ``additive functions'', see e.g. \cite{Gab} and references therein. Multiplicative friezes naturally appear in \cite{CaCh} and are precisely defined in \cite{ARS}.

\begin{rem}
It is possible to define friezes in a more general way using Cartan matrices or valued quivers, \cite{ARS}.
\end{rem}

\begin{rem}
Other rules for friezes naturally appear in the context of cluster algebras. For instance, 
 \textit{cluster-additive friezes} and \textit{tropical friezes} with  recurrence rules
\begin{equation*}
f(\tau v)+f(v)=\sum_{w\overset{\alpha}{\longrightarrow} v}\;
\max(f(w),0),
\qquad
f(\tau v)+f(v)=\max(\sum_{w\overset{\alpha}{\longrightarrow} v}\;f(w), 0),
\end{equation*}
respectively,
are introduced and studied in \cite{Rin}, and \cite{Guo}.
\end{rem}

\begin{ex}
For $\Qc$ a Dynkin quiver of type $\mA_{m}$, multiplicative friezes coincide with the Coxeter friezes of width $m$, and additive friezes coincide with the patterns
studied in \cite{She}, \cite{Marc}. See also \cite{Gab} where many additive friezes are represented.
We give below examples of friezes over integers, for the type $\mD_5$ and for the Kronecker quiver $ \xymatrix  { \bullet\ar@<-1pt>[r]\ar@<2pt>[r]& \bullet}$.

(1) A multiplicative frieze of type $\mD_{5}$ (computed in \cite{BaMa}):
\begin{center}
\begin{small}
 \xymatrix  @!0 @R=2.1em @C=2.1pc 
 {
&&
\overset{}{}\ar@{.}[rd]&&
\overset{}{2}\ar[rd]&&
\overset{}{1}\ar[rd]&&
\overset{}{4}\ar[rd]&&
\overset{}{5}\ar[rd]&&
\overset{}{6}\ar[rd]&&
\overset{}{1}\ar@{.}[rd]&& 
\\
& 
\overset{}{}\ar@{.}[rd]\ar@{.}[r]&\ar@{.}[r]&
\overset{}{5}\ar[rd]\ar[ru]\ar[r]&\overset{}{1}\ar[r]&
\overset{}{1}\ar[rd]\ar[ru]\ar[r]&\overset{}{2}\ar[r]&
\overset{}{3}\ar[rd]\ar[ru]\ar[r]&\overset{}{2}\ar[r]&
\overset{}{19}\ar[rd]\ar[ru]\ar[r]&\overset{}{10}\ar[r]&
\overset{}{29}\ar[rd]\ar[ru]\ar[r]&\overset{}{3}\ar[r]&
\overset{}{5}\ar@{.}[rd]\ar[ru]\ar[r]&2\ar@{.}[r]&
&&
\\
\overset{}{}\ar@{.}[rd]&&
\overset{}{8}\ar[rd]\ar[ru]&&
\overset{}{2}\ar[rd]\ar[ru]&&
\overset{}{1}\ar[rd]\ar[ru]&&
\overset{}{7}\ar[rd]\ar[ru]&&
\overset{}{11}\ar[rd]\ar[ru]&&
\overset{}{8}\ar@{.}[rd]\ar[ru]&&
&&
&&
\\
& 
\overset{}{3}\ar[ru]&&
\overset{}{3}\ar[ru]&&
\overset{}{1}\ar[ru]&&
\overset{}{2}\ar[ru]&&
\overset{}{4}\ar[ru]&&
\overset{}{3}\ar[ru]&&
&&
&& 
\\
}
\end{small}
\end{center}

(2) An additive frieze of type $\mD_{5}$:
\begin{center}
\begin{small}
 \xymatrix  @!0 @R=2.1em @C=2.1pc 
 {
&&
\overset{}{}\ar@{.}[rd]&&
\overset{}{2}\ar[rd]&&
\overset{}{-1}\ar[rd]&&
\overset{}{1}\ar[rd]&&
\overset{}{-2}\ar[rd]&&
\overset{}{1}\ar[rd]&&
\overset{}{-2}\ar@{.}[rd]&& 
\\
& 
\overset{}{}\ar@{.}[rd]\ar@{.}[r]&\ar@{.}[r]&
\overset{}{1}\ar[rd]\ar[ru]\ar[r]&\overset{}{-1}\ar[r]&
\overset{}{1}\ar[rd]\ar[ru]\ar[r]&\overset{}{2}\ar[r]&
\overset{}{0}\ar[rd]\ar[ru]\ar[r]&\overset{}{-2}\ar[r]&
\overset{}{-1}\ar[rd]\ar[ru]\ar[r]&\overset{}{1}\ar[r]&
\overset{}{-1}\ar[rd]\ar[ru]\ar[r]&\overset{}{-2}\ar[r]&
\overset{}{-1}\ar@{.}[rd]\ar[ru]\ar[r]&1\ar@{.}[r]&
&&
\\
\overset{}{}\ar@{.}[rd]&&
\overset{}{1}\ar[rd]\ar[ru]&&
\overset{}{1}\ar[rd]\ar[ru]&&
\overset{}{0}\ar[rd]\ar[ru]&&
\overset{}{0}\ar[rd]\ar[ru]&&
\overset{}{-1}\ar[rd]\ar[ru]&&
\overset{}{-1}\ar@{.}[rd]\ar[ru]&&
&&
&&
\\
& 
\overset{}{0}\ar[ru]&&
\overset{}{1}\ar[ru]&&
\overset{}{0}\ar[ru]&&
\overset{}{0}\ar[ru]&&
\overset{}{0}\ar[ru]&&
\overset{}{-1}\ar[ru]&&
&&
&& 
}
\end{small}
\end{center}

(3) A multiplicative frieze over the Kronecker quiver:
\begin{center}
\begin{small}
 \xymatrix  @!0 @R=2.1em @C=2.1pc 
 {
 &
 1\ar@<2pt>[rd]\ar@<-2pt>[rd]&&
  2\ar@<2pt>[rd]\ar@<-2pt>[rd]&&
   13\ar@<2pt>[rd]\ar@<-2pt>[rd]&&
    89\ar@<2pt>[rd]\ar@<-2pt>[rd]&&
     610\ar@<2pt>@{.}[rd]\ar@<-2pt>@{.}[rd]&&
 \\
\ar@<2pt>@{.}[ru]\ar@<-2pt>@{.}[ru]&&
  1\ar@<2pt>[ru]\ar@<-2pt>[ru]&&
   5\ar@<2pt>[ru]\ar@<-2pt>[ru]&&
    34\ar@<2pt>[ru]\ar@<-2pt>[ru]&&
     233\ar@<2pt>[ru]\ar@<-2pt>[ru]&&
 }
 \end{small}
 \end{center}
 
(4) An additive frieze over the Kronecker quiver:
\begin{center}
\begin{small}
 \xymatrix  @!0 @R=2.1em @C=2.1pc 
 {
 &
 1\ar@<2pt>[rd]\ar@<-2pt>[rd]&&
3\ar@<2pt>[rd]\ar@<-2pt>[rd]&&
5\ar@<2pt>[rd]\ar@<-2pt>[rd]&&
7\ar@<2pt>[rd]\ar@<-2pt>[rd]&&
9\ar@<2pt>@{.}[rd]\ar@<-2pt>@{.}[rd]&&
 \\
\ar@<2pt>@{.}[ru]\ar@<-2pt>@{.}[ru]&&
2\ar@<2pt>[ru]\ar@<-2pt>[ru]&&
4\ar@<2pt>[ru]\ar@<-2pt>[ru]&&
6\ar@<2pt>[ru]\ar@<-2pt>[ru]&&
8\ar@<2pt>[ru]\ar@<-2pt>[ru]&&
 }
 \end{small}
 \end{center}
\end{ex}

\begin{defn}
Let us fix a set of indeterminates $\{x_1,\ldots, x_n\}$. 
The \textit{generic} additive and multiplicative friezes, denoted by $f_{ad}$ and $f_{mu}$ respectively, are defined
by assigning the value $x_i$ to the vertex $(0,i)$ for all $1\leq i \leq n$. One gets
$$
f_{ad}: \Z \Qc_{} \rightarrow\Z[x_1,\ldots, x_n], \qquad f_{mu}: \Z \Qc_{} \rightarrow \Q(x_1,\ldots, x_n).
$$
We will refer to $x_i$'s as the \textit{initial values} of the friezes.
\end{defn}
Note that these functions are well defined, see e.g. Lemma 3.1 of \cite{AsDu}. 
One can note also that $f_{mu}$ takes values in $ \Q_{sf}(x_1,\ldots, x_n)$ the set of subtraction-free rational fractions, and using the theory of cluster algebras this can be even reduced to $\Z_{\geq 0}[x_1^{\pm 1},\ldots, x_n^{\pm 1}]$ the set of Laurent polynomials with positive integer coefficients.

\begin{rem}\label{fri_nonper}
If $x_1,\ldots, x_n$ are not indeterminates but some given values in a ring $\A$, one may find different multiplicative friezes with same initial values $x_i$'s. 
Indeed, it may happen that $f(\tau v)=0$ for some $v$, and thus the multiplicative rule does not allow us to define uniquely $f(v)$.
Below, we give an example of two different multiplicative friezes on the repetition quiver of $\mA_3$ with same initial values $(0,-1,0)$.

\begin{small}
  \xymatrix  @!0 @R=2.1em @C=2.1pc 
 {
 &
&&
0\ar[rd]&&
0\ar[rd]&&
0\ar[rd]&&
0\ar[rd]&&
0\ar[rd]&& \cdots
\\
\cdots& &
-1\ar[rd]\ar[ru]&&
-1\ar[rd]\ar[ru]&&
-1\ar[rd]\ar[ru]&&
-1\ar[rd]\ar[ru]&&
-1\ar[rd]\ar[ru]&&
-1&&
\\
& 
0\ar[ru]&&
0\ar[ru]&&
0\ar[ru]&&
0\ar[ru]&&
0\ar[ru]&&
0\ar[ru]&& \cdots
\\
 &
&&
0\ar[rd]&&
2\ar[rd]&&
0\ar[rd]&&
4\ar[rd]&&
0\ar[rd]&& \cdots
\\
\cdots& &
-1\ar[rd]\ar[ru]&&
-1\ar[rd]\ar[ru]&&
-1\ar[rd]\ar[ru]&&
-1\ar[rd]\ar[ru]&&
-1\ar[rd]\ar[ru]&&
-1&&
\\
& 
0\ar[ru]&&
1\ar[ru]&&
0\ar[ru]&&
3\ar[ru]&&
0\ar[ru]&&
5\ar[ru]&& \cdots
}
\end{small}
\end{rem}

\subsection{Symmetry of friezes}

A frieze  $f: \Z \Qc_{} \rightarrow \A$ is \textit{periodic},
if there exists an integer $N\geq 1$ such that $f\tau^{-N}=f$. 
The following theorem is a consequence in terms of friezes of classical results from the theory of quiver representations and the theory of cluster algebras. We suggest a proof below.

\begin{thm}\label{frADE}
The friezes  $f_{ad}$ and $f_{mu}$ over a quiver $\Qc$ are periodic if and only if 
$\Qc$ is a Dynkin quiver of type $\mA_{n}$, $\mD_{n}$ or $\mE_{6,7,8}$; 
in these cases the periods\footnote{Note that the period of $f_{ad}$ coincides with the Coxeter number associated to the corresponding Dynkin diagram, and the period of $f_{mu}$ is that number plus two.} are
$$
\begin{array}{l|c|c}
\textup{periods}&f_{ad}&f_{mu}\\[2pt]
\hline
&&\\
\mA_{n}&n+1&n+3\\[2pt]
\hline
&&\\
\mD_{n}&2(n-1)&2n\\[2pt]
\hline
&&\\
\mE_{6,7,8}&12, 18, 30& 14, 20 ,32
\end{array}.
$$
\end{thm}

\begin{proof}
If $\Qc$ is of type $\mA_{}$, $\mD_{}$, $\mE_{}$, the periodicity of the friezes can be established case by case. 

For the frieze $f_{ad}$ the values on a given copy of $\Qc$ are expressed linearly in terms of the values of the previous copy.
If $d=(d_1, \ldots, d_n)$ are on the $m$-th slice of $\Z\Qc$, then the values 
$d'=(d'_1, \ldots, d'_n)$
of the next slice are obtained by applying a linear transformation $\Psi$ to the column vector $d$.
The transformations in type $\mA_{n}$, and $\mD_{n}$, oriented as in Example \ref{exADE}, are respectively
\begin{equation}\label{PhiAD}
\Psi_\mA=\begin{pmatrix}
-1&1&& &\\
-1&0&1&&\\
\vdots&&\ddots&\ddots\\[-4pt]
-1&0&&\ddots&1\\
-1&0&&&0
\end{pmatrix},
\qquad
\Psi_\mD=\begin{pmatrix}
-1&1&& &\\
-1&0&1&&\\
\vdots&&\ddots&\ddots&1\\[-4pt]
-1&0&&\ddots&1\\
-1&0&&&0
\end{pmatrix},
\end{equation}
and similarly one can write down the matrices for each of the types $\mE_{6,7,8}$.
The periodicty of the additive friezes then follows from the fact that each of the transformations $\Psi$ has finite order, which can be easily established, and determined, using 
Cayley-Hamilton's theorem.

For the frieze $f_{mu}$, the periodicty in type $\mA$ is established by Coxeter, cf Theorem \ref{thmPerio}. One can find similar arguments in type $\mD$ and the friezes in the type  $\mE_{6,7,8}$ can be computed by hand. 
One can also interpret the periodicity of $f_{mu}$ as the Zamolodchikov periodicity in the cluster algebra of same type. 
This periodicity has been proved in \cite{FZYsys}.

The difficult part of the Theorem is the necessary condition. We will give arguments using quiver representations and cluster algebras in the next section (Remarks \ref{endPerad} and \ref{endPermu}).
\end{proof}

For the multiplicative friezes of type $\mA_{n}$ (i.e. Coxeter friezes of width $n$) one already knows that the friezes are $\tau^{n+3}$-invariant. 
Moreover, one knows that there is an extra symmetry: the invariant translation factorizes as the square of an invariant glide reflection. 
There is an analogous extra symmetry in each other Dynkin type (that implies the periodicity)  
that can be expressed using the 
``Nakayama permutation'' $\nu$.
Following \cite{Gab} we define  $\nu: \Z\Qc_{0}\to \Z\Qc_{0}$ in each Dynkin case by
\begin{itemize}
\item $\nu(m,i)=(m+i-1,n+1-i)$ in type $\mA_{n}$, ($\nu$ is a glide reflection),
\item $\nu(m,i)=(m+n-2,i)$ in type $\mD_{n}$, with $n$ even,
\item  $\nu(m,i)=(m+n-2,i)$, for $1\leq i \leq n-2$, and  $\nu(m,n-1)=(m+n-2,n)$,  $\nu(m,n)=(m+n-2,n-1)$ in type $\mD_{n}$, with $n$ odd,
\item
 $\nu(m,i)=(m+5,6-i)$, for $1\leq i \leq 5$, and  $\nu(m,6)=(m+5,6)$ in type $\mE_{6}$, 
\item
 $\nu(m,i)=\tau^{-8}(m,i)=(m+8,i)$ in type $\mE_{7}$, 
 \item
 $\nu(m,i)=\tau^{-14}(m,i)=(m+14,i) $ in type $\mE_{8}$.
 \end{itemize}
 
 Note that $\nu$ commutes with $\tau$, and that $\nu^2
=\tau^{-N}$ with $N=n-1, 2(n-2), 10, 16, 28$, in the cases $\mA_{n}$, $\mD_{n}$ and $\mE_{6,7,8}$, respectively.

We also introduce the following other two transformations
$$
\Sigma:=\tau^{-1}\nu, \quad F:=\tau^{-1}\Sigma.
$$

\begin{rem}
The transformations $\tau$, $\nu$, $\Sigma$ and $F$ are the combinatorial equivalent of the 
Auslander-Reiten, Nakayama, Serre and Frobenius functors, respectively, used for the quiver representations.
\end{rem}
\begin{thm}\label{perADE}
Let $\Qc$ be a Dynkin quiver of type $\mA_{n}$, $\mD_{n}$ or $\mE_{6,7,8}$.
\begin{enumerate}
\item
The frieze $f_{ad}$ satisfies
$$
f_{ad}\Sigma = -f_{ad}.
$$
\item
The frieze $f_{mu}$ satisfies
$$
f_{mu}F= f_{mu}.
$$
\end{enumerate}
\end{thm}

This result twill be explained in Remarks \ref{SymAd} and \ref{endPermu} using a certain symmetry in the Auslander-Reiten quiver associated with $\Qc^{\op}$.

Since all additive friezes can be obtained as an evaluation of the frieze $f_{ad}$, one immediately gets the following corollary.
\begin{cor}
 All additive friezes on a repetition quiver of type  $\mA_{}$, $\mD_{}$, $\mE_{}$ are periodic.
\end{cor}
\noindent
There exist ''singular'' multiplicative friezes, which are not evaluations of $f_{mu}$, and may be non-periodic, cf Remark \ref{fri_nonper} where a 
non-periodic
 multiplicative frieze of type $\mA_3$ appears.

\subsection{Quiver representations}
Friezes arise naturally in the theory of quiver representations.
In this context, vertices of the repetition quiver are identified with finite dimensional modules of the path algebra defined over the initial quiver.
The structure becomes more rich. 
Additive or multiplicative friezes of integers can be obtained by taking the dimensions or the Euler characteristics of the Grassmannian of the modules attached to the vertices. This will be developed in the next sections.

In this section we collect briefly some basic facts and Theorems of quiver representations. We refer to
\cite{Gab}, \cite{AuReS}, \cite{ASS}, \cite{Schi}, for details and complete expositions of the subject.\\

Let $\Qc$ be a finite acyclic connected quiver. We work over the field of complex numbers.
A \textit{representation} (or module) of $\Qc$ is a collection of spaces and maps
\begin{itemize}
\item $(M_i)_{i\in \Qc_0}$, where $M_i$ is a $\C$-vector space attached to the vertex $i$,
\item $(f_{ij}:M_i\to M_j)_{i \to j \in \Qc_1}$, where $f_{ij}$ is a $\C$-linear map attached to an arrow $i \to j $.
\end{itemize}
Let $M=(M_i, f_{ij})$ and $M'=(M'_i, f'_{ij})$ be two representations of $\Qc$.
One defines naturally their direct sum as $M\oplus M'=(M_i\oplus M'_i,  (f_{ij}, f'_{ij}))$. 
A morphism from $M$ to $M'$, is a collection of linear maps $(g_i, i\in\Qc_0)$ such that 
all diagrams of the following form commute
$$
 \xymatrix{
 M_i \ar[r]^{f_{ij}} \ar[d]_{g_i}  & M_j \ar[d]^{g_j} \\
  M'_i \ar[r]_{f'_{ij}} & M'_j
  }.
$$ 
The module $M'$ is a  \textit{subrepresentation} of $M$ if there exists a injective morphism from $M'$ to $M$.
A representation is called \textit{indecomposable} if it is not isomorphic to the direct sum of two non-trivial subrepresentations.

We denote by $\rep$ the category of representations of $\Qc$, with objects and morphisms defined as above. 
This category is equivalent to the category of modules $\modQ$, where $\C\Qc$ is the finite dimensional algebra called the \textit{path algebra}.
This algebra is defined as the $k$-vector space with basis set all the paths in $\Qc$ and multiplication given by composition of paths.

In these categories, projective modules play an important role. We define the family of \textit{standard projective modules} $P_i$, indexed by vertices $i\in \Qc_0$. The module $P_i$ has attached to each vertex $j$ the $\C$-vector space $(P_i)_j$ with basis the set of all paths in $\Qc$ from $i$ to $j$.
For each arrow $j\rightarrow \ell$ the linear map $f_{j\ell}: (P_i)_j\to (P_i)_\ell$ is defined on the basis elements by composing the paths from $i$ to $j$ with the arrow $j\rightarrow \ell$. 

Similarly, one can  define the family of \textit{standard injective modules} $I_i$, indexed by vertices $i\in \Qc_0$. The module $I_i$ has attached to each vertex $j$ the $\C$-vector space $(I_i)_j$ with basis the set of all paths in $\Qc$ from $j$ to $i$.
For each arrow $j\rightarrow \ell$ the linear map $f_{j\ell}: (I_i)_j\to (I_i)_\ell$ is defined on the basis elements by sending the paths from $j$ to $i$ starting with the arrow $j\rightarrow \ell$ to the paths obtained by deleting the arrow $j\rightarrow \ell$, and sending the other paths from $j$ to $i$ to 0.

Let us recall classical theorems and definitions in the theory of quiver representations.

\begin{thm}[Gabriel]
There exist only finitely many indecomposable representations of $\Qc$, up to isomorphism, if and only if $\Qc$ is a Dynkin quiver of type $\mA, \mD, \mE$. 
Moreover, if $\Qc$ is of type $\mA, \mD, \mE$ the following map realises a bijection
$$
\begin{array}{rcl}
\{\text{classes of indecomposables of } \rep \}&\longrightarrow& \{\text{positive roots of the root system of } \Qc\}\\[4pt]
[M]&\mapsto & \sum_{i\in \Qc_0} (\dim M_i)\alpha_i
\end{array}
$$
where $\{\alpha_i\}$ is the basis of simple roots in the root system associated to the Dynkin diagram.
\end{thm}

\noindent
\begin{defn}(AR quiver) 
The \textit{Auslander-Reiten quiver} of $\rep$ is the quiver $\G_\Qc$  defined by:
\begin{itemize}
\item vertices: isomorphism classes of indecomposable objects $[M]$,
\item arrows:  ${ [M] \overset{\ell}{\longrightarrow}  [N]}$, if the space of irreducible morphisms from $M$ to $N$ is of dimension $\ell$.
\end{itemize}
\end{defn}
The irreducible morphisms are those that are not compositions, or combinations of compositions, of other non-trivial morphisms.
In other words the AR quiver gives the elementary bricks (modules and morphisms) to construct $\rep$.

The following classical theorem relates the AR quiver $\G_{\Qc}$ (or part of it)  to the repetition quiver over $\Qc^{\op}$ (the quiver with opposite orientation).

\begin{thm}\label{ARfri}
Let $\Qc$ be a finite acyclic connected quiver.
\begin{enumerate}
\item The projective modules all belong to the same connected component $\Pi_\Qc$ of $\G_\Qc$. 
\item In the case when $\Qc$ is a Dynkin quiver of type $\mA, \mD, \mE$, the quiver $\G_\Qc$ is connected and can be embedded in the repetition quiver:
\begin{equation}\label{inj}
\Pi_\Qc \simeq \G_\Qc \hookrightarrow \N\Qc^{\op}.
\end{equation}
The image of $\G_\Qc$ under this injection corresponds to the full subquiver of $ \N\Qc^{\op}$ lying between the vertices $(0,i)$ and $\nu(0,i)$, $i\in \Qc^{\op}_0$. 
For all $i\in \Qc_0$, the projective module $P_i$ in $\rep$ identifies with the vertices $(0,i)$, and the injective module $I_i$ with $\nu(0,i)$.
\item In all other cases, $\G_\Qc$ is not connected. The component $\Pi_\Qc$ is
isomorphic to the repetition quiver:
\begin{equation}\label{bij}
\Pi_\Qc \overset{\sim}{\rightarrow} \N\Qc^{\op}.
\end{equation}
For all $i\in \Qc_0$, the projective module $P_i$ in $\rep$ identifies with the vertices $(0,i)$, in $\N\Qc^{\op}$.\end{enumerate}
\end{thm}

The structure of the graph $\N\Qc^{\op}$ reflects properties between the modules of $\rep$.
Let $M$ and $N$ be indecomposable modules.
An exact sequence $ 0\to N\to E \overset{g}{\rightarrow} M\to 0$ is called \textit{almost split}, if it is not split, if every non-invertible map $X \overset{}{\rightarrow} M$, with $X$ indecomposable, factors through~$g$.

\begin{thm}[Auslander-Reiten]
In $\rep$, for every indecomposable nonprojective module $M$, there exists a unique, up to isomorphisms, almost split sequence 
 $ 0\to N\to E \overset{}{\rightarrow} M\to 0$.
\end{thm}

The \textit{AR translation} $\tau$  is defined on the nonprojective vertices of $\G_{\Qc}$ by $\tau M:=N$  for $M$ and $N$ related by the almost-split sequence $ 0\to N\to E \overset{}{\rightarrow} M\to 0$.
\begin{thm}[Auslander-Reiten]
Under the maps  \eqref{inj} and \eqref{bij} the AR translation and the translation $\tau$ of the repetition quiver coincide.
\end{thm}

Every almost split sequence  $ 0\to \tau M\to E \overset{}{\rightarrow} M\to 0$ leads to the following subquiver of the AR quiver~$\G_\Qc$:
\begin{equation}\label{subAR}
  \xymatrix  @!0 @R=1.8em @C=3.6pc { &E_1 \ar[rdd] &\\& E_2 \ar[rd] & \\ \tau M \ar[ru]\ar[ruu] \ar@{.>}[rd] \ar[rdd] &\vdots&M&\\&\ar@{.>}[ru]&\\& E_\ell\ar[ruu]  &}
\end{equation}
where the $E_i$'s are the indecomposable factors of $E$. There are no other arrows arriving at vertex $M$ or exiting from vertex $\tau M$.

\subsection{Additive friezes and dimension vectors}\label{adSec}

Let $f:  \Z\Qc\to  \Z$ be an additive frieze of type $\Qc$. As usual $n$ stands for the cardinality of $\Qc_0$.
We denote by $f_{m,j}$ the value of $f$ at the vertex $(m,j)$ of $\Z\Qc$, and we denote by $f_m$ the vector of $\Z^n$ of the values of $f$ on the $m$-th slice of $\Z\Qc$, i.e.
$$
f_m=
\begin{pmatrix}
f_{m,1}\\
\vdots\\
f_{m,n}
\end{pmatrix}.
$$
The frieze rule implies that the components of $f_m$ are $\Z$-linear expressions in $f_{m-1}$. 
In other words, there exists a matrix $\Psi_\Qc$, depending only on $\Qc$, satisfying for all $m\in \Z$
$$
f_m=\Psi_\Qc f_{m-1}.
$$
We want to give an expression for $\Psi_\Qc$ (for examples in type $\mA$ and $\mD$ cf \eqref{PhiAD}).
We will use some known results related to the representations of $\Qc$ and $\Qc^{\op}$.\\

The dimension vector of a module $M=(M_i, i\in \Qc_0; f_\a, \a\in \Qc_1)$ is a vector of $\N^{n}$ defined by
$$\underline{\dim}\, M=(\dim M_i)_{i\in \Qc_0}.$$

The alternate sum of dimensions of the spaces in the exact sequence $ 0\to \tau M\to \oplus E_i \overset{}{\rightarrow} M\to 0$ vanishes and leads to the relation
$$
\underline{\dim}\, \tau M + \underline{\dim}\, M=\textstyle \sum_i
 \underline{\dim}\, E_i.$$
 
 This relation allows to compute recursively the indecomposable modules from the projective ones,
 the process is known as ``knitting algorithm''.

The mapping $\underline{\dim}: \Pi_\Qc\to \Z^{n}$ is interpreted as an additive frieze on $\Pi_\Qc$.
Using the map \eqref{inj} or \eqref{bij} this induces an additive frieze from $\Z\Qc^{\op}$ to $\Z^{n}$.
Since we consider friezes from $\Z\Qc$ to $\Z^{n}$, we will use the representations of $\Qc^{\op}$.
By Theorem \ref{ARfri} the standard projective modules $P_i^{\op}$ of $\rep^{\op}$ are attached to the vertices $(0,i)$ of $\Z\Qc$, 
and in type  $\mA$, $\mD$, $\mE$,  the injective modules $I_i^{\op}$ are attached to to the vertex $\nu(0,i)$ of $\Z\Qc$.

Define the additive frieze of dimensions
$$
\underline d: \Z\Qc\to  \Z^{n}
$$
by assigning the initial values 
$ \underline d (0,i) = \udim P_i^{\op}$, for all $i\in \Qc_0$. 
Using the projection $pr_i$ on the $i$-th component of the vectors in $ \Z^{n}$, we define 
a family of additive friezes, indexed by  $i \in \Qc_0$,
$$
d^i:=\, pr_i\circ \underline d\,:  \Z\Qc\to  \Z \;.
$$

\begin{ex}\label{exFrdim}
Let us illustrate the above notions for the following quiver.
$$
\begin{small}
  \xymatrix  
  @!0 @R=1cm @C=1.7cm
  {
  &&2\ar[d]&&&&2\ar[ld]
  \\
\Qc:&1\ar[ru]\ar[r]&3&
&\Qc^{\op}:&1&3\ar[u]\ar[l]
\\
P_1: {\begin{array}{r}
1\\[-4pt]
12
\end{array}}
&P_2: {\begin{array}{r}
1\\[-4pt]
01
\end{array}}
&P_3: {\begin{array}{r}
0\\[-4pt]
01
\end{array}}
&
&P_1^{\op}: {\begin{array}{r}
0\\[-4pt]
10
\end{array}}
&P_2^{\op}: {\begin{array}{r}
1\\[-4pt]
10
\end{array}}
&P_3^{\op}: {\begin{array}{r}
1\\[-4pt]
21
\end{array}}
  }
  \end{small}
$$
where we write the dimension vectors of the modules under the form 
$ {\begin{array}{r}
d_2\\[-4pt]
d_1d_3
\end{array}}
$.

We obtain the following friezes of type $\Qc$:
$$
\begin{small}
  \xymatrix    
  @!0 @R=1.2cm @C=1.8cm
  {
 \underline d:\qquad  \ar@{.>}[rd]&&
{\begin{array}{r}
1\\[-4pt]
10
\end{array}}
  \ar[d]\ar[rd]
  &&
{\begin{array}{r}
2\\[-4pt]
32
\end{array}}
  \ar[d]\ar[rd]
  &&
  {\begin{array}{r}
4\\[-4pt]
43
\end{array}}
\ar[d]  \ar@{.>}[rd]
  &
  \\
  \ar@{.>}[r] \ar@{.>}[rru]& 
 {\begin{array}{r}
0\\[-4pt]
10
\end{array}}
 \ar[r]\ar[ru]&
  {\begin{array}{r}
1\\[-4pt]
21
\end{array}}
\ar[r]\ar[rru]&
{\begin{array}{r}
2\\[-4pt]
21
\end{array}}
  \ar[r]\ar[ru]&
{\begin{array}{r}
3\\[-4pt]
32
\end{array}}
\ar[r]\ar[rru]&
  {\begin{array}{r}
3\\[-4pt]
43
\end{array}}
\ar[r]\ar[ru]&
{\begin{array}{r}
4\\[-4pt]
54
\end{array}}
  \ar@{.>}[r]&
  \\
d^1:\qquad     \ar@{.>}[rd]&&
 1\ar[d]\ar[rd]
  &&
3
  \ar[d]\ar[rd]
  &&
4\ar[d]  \ar@{.>}[rd]
  &
  \\
\ar@{.>}[r] \ar@{.>}[rru]&
1\ar[r]\ar[ru]&
2\ar[r]\ar[rru]&
2\ar[r]\ar[ru]&
3\ar[r]\ar[rru]&
4\ar[r]\ar[ru]&
5
  \ar@{.>}[r]&
   }
  \end{small}
$$
One can see that in the frieze $d^1: \Z\Qc\to \Z$ the first slice
 $\begin{array}{r}1\\[-4pt] 12\end{array}$  coincides with the dimension vector of $P_1$, the next slices give the dimension vectors of the translated of $P_{1}$ through $\tau^{-1}$.
 \end{ex}

The vector dimensions, and thus the friezes $d^i$, can be computed by the mean of the so-called Coxeter transformation.
Let us recall some known results, see e.g. \cite{ASS}, \cite{Schi}.

The \textit{Cartan matrix} $C_\Qc=(c_{ij})_{i,j \in \Qc_0}$ associated with $\Qc$ is given by
$$
c_{ij}= \text{number of paths in } \Qc \text{ from } j \text{ to } i.
$$
The matrix $C_\Qc$ is invertible; its inverse $C_\Qc^{-1}=(b_{ij})$ is given by $b_{ii}=1$ and for $i\not=j$
$$
b_{ij}=-( \text{number of arrows in } \Qc \text{ from } j \text{ to } i).
$$
The \textit{Coxeter transformation} $\Phi_\Qc$ is defined as
$$
\Phi_\Qc=-\; ^tC_\Qc C_\Qc^{-1},
$$
where the superscript $t$ denotes the transpose operation on the matrix $C_\Qc$.

Note that $C_{\Qc^{\op}}=\,^tC_\Qc$ and $\Phi_{\Qc^{\op}}=\Phi_\Qc^{-1}$.
Also it is immediate from the definitions that
$$
\Phi_\Qc\,\udim  P_i= - \udim I_i,
$$
for all standard projective and injective modules. 
In general, one has the following classical theorem,
see e.g. \cite[\S IV 2.9, p116]{ASS}, \cite[Cor 7.16 p190]{Schi}.

\begin{thm}\label{propMN}
Let $M$ and $N$ be an indecomposable modules of $\rep$.
If $M$ is non-projective, and $N$ non-injective, one has
$$
\Phi_\Qc\;\udim  M=\udim \tau M, \qquad   \Phi_\Qc^{-1}\, \udim  N= \udim \tau^{-1} N.
$$
\end{thm}

The above relations interpreted in terms of friezes lead to the following result.
\begin{lem}\label{lemPsi}
For the friezes $d^i$, one has $\Psi_\Qc=\Phi_\Qc^{-1}$, i.e. $d^i_m=\Phi_\Qc^{-1}d^i_{m-1}$, for all $m\in \Z$.
\end{lem}

\begin{proof}
x gives the vector $\udim P_j^{\op}$, and the $i$-th row of $C_{\Qc^{\op}}$ gives $\udim P_i$
 which is also the vector $d^i_0=(d^i_{0,j})$ of the values of the frieze $d^i$ on the copy of $0\times \Qc$ in $\Z\Qc$.
 By Theorem \ref{propMN} one has
$$
\Phi_{\Qc^{\op}}^{-1}\,C_{\Qc^{\op}}=(\udim \tau^{-1} P_1^{\op}, \ldots, \udim \tau^{-1} P_n^{\op})=:B_{\Qc^{\op}}.
$$
The $i$-th row of $B_{\Qc^{\op}}$ gives  the vector $d^i_1=(d^i_{1,j})$ of the values of the frieze $d^i$ on the next copy of $1\times \Qc$ in $\Z\Qc$. 
By transposing the matrices in the above equation one gets
$$
(d^1_1, \ldots, d^n_1)=\,^tB_{\Qc^{\op}}
=\,^tC_{\Qc^{\op}}\, ^t\Phi_{\Qc^{\op}}^{-1}
=\,C_{\Qc^{}}\, ^t\Phi_{\Qc^{}}^{}
=\Phi_\Qc^{-1}C_\Qc
=\Phi_\Qc^{-1}(d^1_0, \ldots, d^n_0).
$$
Hence, the result.
\end{proof}

\begin{prop}\label{fadBase}
The family $(d^i)_{i\in \Qc_0}$ forms a $\Z$-basis of the space of additive friezes from $ \Z\Qc$ to~$\Z$.
The additive frieze $f_{ad}: \Z \Qc_{0} \rightarrow\Z[x_1,\ldots, x_n]$ decomposes as a formal combination
$$
f_{ad}=\sum_{i\in \Qc_0} a_i d^i,
$$
where the coefficients are given by $(a_i)_i=C_\Qc^{-1}(x_i)_i$.
\end{prop}

\begin{proof}
Let $f: \Z \Qc_{0} \rightarrow\Z$ be an additive frieze with initial values given by the column vector $f_0=(f_{0,j})_j$.
The vectors $d^i_0$ are the columns of the invertible matrix $C_\Qc$, in particular they form a $\Z$-basis of $\Z^{\Qc_0}$.
One writes $f_0$ in this basis:
$$
f_0=\sum_i a_i d^i_0=C_\Qc(a_i)_i.
$$
Using Lemma \ref{lemPsi} one obtains the values of $f$ on any slice as
$
f_m=\Phi_\Qc^{-m}\sum_i a_i d^i_0=\sum_i a_i d^i_m.
$
One deduces $
f=\sum_i a_i d^i,
$
with  $(a_i)_i=C_\Qc^{-1}f_0$.
\end{proof}

\begin{ex}
Going back to Example \ref{exFrdim}, one computes the Cartan matrix, its inverse, and the Coxeter transformation
$$
C_\Qc=\begin{pmatrix}
1&0&0\\
1&1&0\\
2&1&1
\end{pmatrix}, 
\qquad
C_\Qc^{-1}=\begin{pmatrix}
1&0&0\\
-1&1&0\\
-1&-1&1
\end{pmatrix}, 
\qquad
\Phi_\Qc^{-1}=-C_\Qc\, ^tC_\Qc^{-1}=\begin{pmatrix}
-1&1&1\\
-1&0&2\\
-2&1&2
\end{pmatrix}.
$$
The formal frieze $f_{ad}$ is
$$
\begin{small}
  \xymatrix    
  @!0 @R=1.5cm @C=2.15cm
  {
   f_{ad}:\quad \ar@{.>}[rd]&&
    x_2\ar[d]\ar[rd]
  &&
  2x_3-x_1
  \ar[d]\ar[rd]
  &&
  {\begin{array}{c}
x_2+3x_3\\[-4pt]
-3x_1
\end{array}}\ar[d]  \ar@{.>}[rd]
  &
  \\
\ar@{.>}[r] \ar@{.>}[rru]&
x_1\ar[r]\ar[ru]&
x_3\ar[r]\ar[rru]&
  {\begin{array}{c}
x_2+x_3\\[-4pt]
-x_1
\end{array}}\ar[r]\ar[ru]&
  {\begin{array}{c}
x_2+2x_3\\[-4pt]
-2x_1
\end{array}}\ar[r]\ar[rru]&
3x_3-2x_1\ar[r]\ar[ru]&
4x_3-3x_1
  \ar@{.>}[r]&
   }
  \end{small}
$$
and decomposes
as
$$
f_{ad}=a_1 d^1+a_2d^2+a_3d^3, \quad \text{ where } 
\begin{pmatrix}
a_1\\a_2\\a_3
\end{pmatrix}=C_\Qc^{-1}
\begin{pmatrix}
x_1\\x_2\\x_3
\end{pmatrix}=
\begin{pmatrix}
x_1\\
x_2-x_1\\
x_3-x_2-x_1
\end{pmatrix}.
$$
For instance on can check this formula for the vertex $(1,1)$:
\begin{eqnarray*}
f_{ad}(1,1)&=&x_1d^1(1,1)+(x_2-x_1)d^2(1,1)+(x_3-x_2-x_1)d^3(1,1)\\
&=&2x_1+2(x_2-x_1)+(x_3-x_2-x_1)\\
&=&-x_1+x_2+x_3.
\end{eqnarray*}
The values of $f_{ad}$ on consecutive copies of $\Qc$ in $\Z\Qc$ are obtained by applying $\Phi_\Qc^{-1}$:
$$
\cdots \overset{\Phi_\Qc^{-1}}{\longrightarrow}
\begin{pmatrix}
x_1\\x_2\\x_3
\end{pmatrix}
\overset{\Phi_\Qc^{-1}}{\longrightarrow}
\begin{pmatrix}
-x_1+x_2+x_3\\
-x_1+2x_3\\
-2x_1+x_2+2x_3
\end{pmatrix}
\overset{\Phi_\Qc^{-1}}{\longrightarrow}
\begin{pmatrix}
-2x_1+x_3\\
-3x_1+x_2+3x_3\\
-3x_1+4x_3
\end{pmatrix}
\overset{\Phi_\Qc^{-1}}{\longrightarrow}
\cdots
$$
\end{ex}

\begin{rem}\label{endPerad}
In the case when $\Qc$ is an acyclic quiver not of Dynkin type $\mA, \mD, \mE$, by a theorem of Auslander \cite{Aus}
one knows that
there exist indecomposable modules in $\rep^{\op}$ of arbitrarily large dimensions.
 Hence, at least one of the functions $d^i$ is not bounded, and the frieze $f_{ad}$ can not be periodic. This is the final argument for the proof of Theorem \ref{frADE} in the additive case.
\end{rem}

\begin{rem}[Proof of Theorem \ref{perADE}, additive case] \label{SymAd}
In the case when $\Qc$ is a Dynkin quiver of type $\mA, \mD, \mE$,
the identifications of the vertices $(0,j)$ and $\nu(0,j)$ of $\Z\Qc$ with the modules $P_j^{\op}$ and $I_j^{\op}$
induce the symmetry of $f_{ad}$ given in Theorem \ref{perADE}. Indeed, the property
$$
d^i (-1,j))_j=\Phi_\Qc(d^i (0,j))_j=-(d^i \nu(0,j))_j.
$$ 
is a reformulation of the property $\Phi_\Qc \udim P_i=-\udim I_i$.
 One therefore deduces the symmetry $d^i\tau^{-1}\nu=-d^i$,  for all $i$.
Applying twice this property one deduces the periodicty.
One can explain that the periods are exactly the Coxeter number of the Dynkin graph $\Qc$ by interpreting the
 Coxeter transformation $\Phi_\Qc$ as the action of a Coxeter element of the Weyl group in the corresponding root system. The order of such elements is precisely the Coxeter number of the graph.
\end{rem}

\subsection{Multiplicative friezes and cluster character}\label{muSec}
The theory of multiplicative friezes is closely related to the theory of cluster algebras. 
One can immediately recognize the entries of the frieze $f_{mu}$ as cluster variables.
We start by collecting some basic definitions and notions from the theory of cluster algebras.

Cluster algebra is a recent theory developed by Fomin and Zelevinsky, \cite{FZ1}-\cite{FZ2}.
 Let us mention the following surveys, notes and books, on the subject \cite{Kel}, \cite{Rei}, \cite{GSV}, \cite{MarshBook}, 
 where one can find more details on what will be exposed below.

Cluster algebras are commutative associative algebras defined by
generators and relations. 
The generators and relations are not given from the beginning.
They are obtained recursively using a combinatorial procedure encoded in a quiver.

We fix here the algebraically closed field $k=\C$.
Our initial data is a quiver  $\Qc$ with no loops and no $2$-cycles, and a set of indeterminates 
$\{x_1,\ldots, x_n\}$. 
As before $n$ stands for the 
cardinality of $\Qc_0$,
and the vertices are labeled with an integer in $\{1, \ldots, n\}$.
The cluster algebra $\A_\Qc(x_1,\ldots, x_n)$ will be defined as a subalgebra of the field of fractions $\C(x_1,\ldots, x_n)$.
The generators and relations of $\A_\Qc$ are given using the recursive procedure called
\textit{seed mutations} that we describe below.

A \textit{seed} is a couple 
$
\Sigma=\left((u_1, \ldots, u_n) , \;\Rc\right),
$
where $\Rc$ is a quiver, without loop and 2-cycle, with $n$ vertices,
and where $u_1, \ldots, u_n$ are free generators of $\C(x_1,\ldots, x_n)$ labeled by the vertices of the graph $\Rc$. 
The \textit{mutation at  vertex}  $k$ of the seed 
$\Sigma$ is a new seed $\mu_k(\Sigma )$ defined by
\begin{enumerate}
\item[\textbullet]
$\mu_k(u_1, \ldots, u_n)=(u_1, \ldots, u_{k-1},u'_k, u_{k+1},\ldots, u_n)$ where
$$
\displaystyle
u'_ku_{k}=\prod\limits_{\substack{\text{arrows in }\Rc\\ i\rightarrow k }}\; u_i
 \quad+\quad 
\prod\limits_{\substack{\text{arrows in }\Rc\\ i\leftarrow k }}\;u_i.
$$
\item[\textbullet] 
$\mu_k(\Rc)$ is the graph obtained from $\Rc$ by applying the following transformations
\begin{enumerate}
\item for each possible path $i\rightarrow k \rightarrow j$ in $\Rc$, add an arrow $i\rightarrow j$,
\item reverse all the arrows leaving or arriving at $k$,
\item remove a maximal collection of 2-cycles.
\end{enumerate}
\end{enumerate}
Mutations are involutions.
\begin{ex}\label{exmut} Example of mutation\\
 \xymatrix  
  @!0 @R=1cm @C=1.5cm
  {
  &&2\ar[d] &&&&&&2\ar[ld]&&& \\
(x_1,x_2,x_3), &1\ar[ru]&3\ar[l]&\ar@{<->}[r]^{\mu_{1}}&&&(\frac{x_2+x_3}{x_1},x_2,x_3), &1\ar[r]&3
}
\end{ex}

\begin{defn}
Starting from the initial seed $\Sigma_0=((x_1,\ldots, x_n), \Qc)$, one produces  $n$ new seeds 
$\mu_k(\Sigma_0)$, $k=1,\ldots, n$. 
One applies all the possible mutations.
The set of rational functions appearing in any of the seeds produced in the mutation process
is called a \textit{cluster}. 
The rational functions in a cluster are called \textit{cluster variables}.
The cluster algebra $\A_\Qc(x_1,\ldots, x_n)$ is defined as the subalgebra of  $\C(x_1,\ldots, x_n)$ generated by all the cluster variables.
\end{defn}

\noindent
Note that if $\left((u_1, \ldots, u_n) , \;\Rc\right)$ is a seed obtained by sequences of mutations from $((x_1,\ldots, x_n), \Qc)$
then the algebras  $\A_\Rc(u_1,\ldots, u_n)$ and $\A_\Qc(x_1,\ldots, x_n)$ are isomorphic.

One of the first surprising result is the so-called \textit{Laurent phenomenon}.
\begin{thm}[\cite{FZ1}]\label{LaurentFZ}
In the cluster algebra $\A_\Qc$ every cluster variable can be written as a Laurent polynomial with integer coefficients in the variable of any given cluster.
\end{thm}

Moreover, the coefficients of the above Laurent polynomials have been conjectured to be positive integers. 
This is proved in the situation we are considering (and in other more general situations), see \cite{KiQi}, \cite{LeSc}.

The following result characterizes the Dynkin quivers in the theory of cluster algebras.

\begin{thm}[\cite{FZ2}]\label{classFZ}
The cluster algebra $\A_\Qc$ has finitely many cluster variables if and only if the initial graph $\Qc$ is mutation-equivalent to a Dynkin quiver of type $\mA,\mD,\mE$.
\end{thm}

Moreover, in type $\mA,\mD,\mE$, the cluster variables are uniquely determined by their monomial denominator (we define uniquely the denominator by  writing the variables as irreducible rational fractions).

\begin{thm}[\cite{FZ2}]
 If $\Qc$ is a Dynkin quiver of type $\mA,\mD,\mE$, then one has a bijection between the set of
 non-initial cluster variables of the algebra $\A_\Qc(u_1, \ldots, u_n)$ and the positive roots of the root system associated to $\Qc$.
 Under this bijection, there is a unique cluster variable with denominator 
 $u_1^{d_1}u_2^{d_2}\ldots u_n^{d_n}$, for each 
 positive root $\sum_{i\in \Qc_0}d_i\alpha_i$. \end{thm}

When combining the above theorem with Gabriel's theorem one obtains
a bijection
\begin{equation}\label{bijxM}
\begin{array}{rcl}
\{\text{indecomposables of } \rep \}/\simeq &\longrightarrow& \{\text{non-initial cluster var. of } \A_\Qc(u_1, \ldots, u_n)\}\\[4pt]
M&\mapsto & x_M
\end{array}
\end{equation}
that gives $x_M$ as the unique cluster variable with denominator $u_1^{d_1}u_2^{d_2}\ldots u_n^{d_n}$, where $(d_i)=\udim M$.

Caldero and Chapoton gave an explicit formula for the variables $x_M$.
Their formula uses the \textit{quiver Grassmannian} $\Gr_e(M)$ defined for a quiver representation $M$ by
$$\Gr_e(M)=\{N \text{ subrepresentation of } M, \;\udim N =e\},$$
for all $e\in \N^{\Qc_0}$. The quiver Grassmannian is a projective subvariety of a product of ordinary Grassmannians.
\begin{thm}[\cite{CaCh}]
Let $\Qc$ be a Dynkin quiver of type $\mA,\mD,\mE$, and $M$ an indecomposable module of $\rep$ with $\udim M=(d_i)$.
One has 
\begin{equation}\label{CCformula}
x_M=\dfrac{1}{u_1^{d_1}u_2^{d_2}\ldots u_n^{d_n}}\sum_{e\in \N^{\Qc_0}}\chi(\Gr_e(M))\prod_{i\in \Qc_0}x_i^{\sum_{j\to i}e_j+\sum_{i\to j}(d_j-e_j)},
\end{equation}
where $\chi$ is the Euler characteristic.
\end{thm}

The function $CC: M \mapsto CC(M)=x_M$ defined by the formula \eqref{CCformula} is known as the Caldero-Chapoton formula and 
often called a ``cluster characacter''.

\begin{prop}[\cite{CaCh}]\label{CaChfri} 
In $\rep$ each exact sequence $0\to \tau M\to \sum E_i \to M\to 0$, where $M$ and $E_i$ are indecomposables leads to a relation
$$
x_{\tau M}x_M=1+\prod x_{E_i}.
$$
\end{prop}

Recall that the exact sequences as above are represented by the picture \eqref{subAR} in the AR-quiver of $\rep$. In other words the evaluation 
of the cluster character $CC$ on the AR quiver gives rise to a piece of multiplicative frieze. This property certainly motivated the definition of generalized multiplicative friezes.\\

\begin{rem}[Proof of Theorems \ref{frADE} and \ref{perADE}, multiplicative case]\label{endPermu}
Let $\Qc$ be a Dynkin quiver of type $\mA,\mD,\mE$.
By Theorem \ref{ARfri}, we identify the AR quiver $\G_{\Qc^{\op}}$ with the full subquiver of $\Z\Qc$
 containing all the vertices between $(m,i)$ and $\nu(m,i)$, $i\in\Qc_0$ 
so that $P^{\op}_i=(m,i)$ and $I^{\op}_i=\nu(m,i)$.
 
Denote by $u_1, \ldots, u_n$ the entries $f_{mu}(m-1,1), \ldots, f_{mu}(m-1,n)$ in the multiplicative frieze $f_{mu}$.
 It is clear from the definition of the frieze rule that all the entries of $f_{mu}$
 are cluster variables of $\A_\Qc(u_1, \ldots, u_n)$.  
One computes by induction the denominators of the entries $f_{mu}(m,i)$ using the frieze rule (this can be done case by case for instance on the quivers given in Example~\ref{exADE}). The
 denominator of $f_{mu}(m-1,i)$ is of the form $u_1^{d_1}u_2^{d_2}\ldots u_n^{d_n}$ with $d_j$ equals to the number of paths from $j$ to $i$
 in $\Qc$.
In other words, $(d_j)_j=\udim P^{\op}_i$. 
By \eqref{bijxM} one gets $f_{mu}(m,i)=x_{P^{\op}_i}$, by  Proposition \ref{CaChfri}, 
one deduces $f_{mu}(M)=x_M$ for all $M$ of $\G_{\Qc^{\op}}$  identified with a vertex of $\Z\Qc$. 
In particular, one has $f_{mu}(\nu(m,i))=f_{mu}(I^{\op}_i)=x_{I^{\op}_i}$.
Similarly, one can compute the  denominators of the entries $f_{mu}(m-2,i)$ and checks that 
the denominator of $f_{mu}(m-2,i)$ is $u_1^{d'_1}u_2^{d'_2}\ldots u_n^{d'_n}$ with $(d'_j)_j=\udim I^{\op}_i$.
 Hence, $f_{mu}(m-2,i)=x_{I^{\op}_i}=f_{mu}(\nu(m,i))$.
This establishes the symmetry $f_{mu}=f_{mu}\nu\tau^{-2}$ stated in Theorem \ref{perADE}. By applying twice this property one obtains the periodicity.

When $\Qc$ is an acyclic quiver, not of type $\mA,\mD,\mE$, one may use similar arguments.
It is more convenient to use the cluster category $\cC_\Qc$, see \cite{BMRRT}, \cite{Rei}. 
To each rigid indecomposable module $M$ of $\cC_\Qc$ one assigns injectively a cluster variable $x_M$. 
The formula \eqref{CCformula} and Proposition \ref{CaChfri} can be generalized to the acyclic case.
The connected component of the Auslander-Reiten quiver of $\cC_\Qc$ containing the projective modules, called \textit{tranjective component}, is isomorphic to $\Z\Qc$. The frieze $f_{mu}$ can be viewed as the evaluation of $CC$ on the transjective component. Therefore the frieze contains infinitely many different entries and can not be periodic. More details can be found in \cite{AsDu}.
\end{rem}

\begin{rem}
In the frieze $f_{mu}$, from the values $u_i$'s on the slice $m\times \Qc_0$, one computes the values on the next copy $(m+1)\times \Qc_0$
by induction using the frieze rule. This induction corresponds to performing a suitable (not unique) sequence of mutations, in which every mutation $\mu_i$ for $i\in \Qc_0$ appears exactly once. At each step the mutation will be performed at a vertex that is a source (i.e. which has only outgoing arrows). In type 
$\mA, \mD, \mE$, with initial oriented quiver given as in Example~\ref{exADE}, a possible sequence is $\mu=\mu_n\cdots\mu_{2}\mu_1$. 
The periodicity of the multiplicative frieze corresponds to the fact that $\mu$ has finite order $h+2$, where $h$ is the Coxeter number associated to the Dynkin quiver. This property is known as \textit{Zamolodochikov periodicity}. It was proved in the Dynkin case in \cite{FZYsys} and in a more general case in \cite{Kel2}, see also \cite{Kel}.
\end{rem}

\section{$\SL_{k+1}$-friezes}\label{SLtil}

 Coxeter's frieze patterns naturally generalize to $\SL_{k+1}$-tilings and $\SL_{k+1}$-friezes.

An $\SL_{2}$-tiling \cite{ARS} is an
 infinite array $(e_{i,j})_{i,j\in\Z}$ satisfying Coxeter's unimodular rule:
all $2\times 2$ minors over adjacent rows and adjacent columns in the array are equal to 1
(in comparison with frieze patterns, the boundary condition of rows of 1's and 0's is removed).
Generalizing this rule to $(k+1)\times (k+1)$ minors of adjacent rows and columns,
we arrive at the notion of $\SL_{k+1}$-tilings~\cite{BeRe}. 
An $\SL_{k+1}$-tiling bounded (from top and bottom) by a row of 1's and
$k$ rows of 0's, is called an $\SL_{k+1}$-frieze
\cite{CoRo},
\cite{BeRe}, \cite{MGOST}.

$\SL_{k+1}$-tilings and friezes are closely related to $T$-systems,
recurrence relations appearing in mathematic physics 
as relations satisfied by a family of transfer matrices in solvable lattice models, \cite{BaRe}, \cite{KNS}.
These systems are also related to the discrete Hirota equation or octahedral recurrence.
$T$-systems were recently studied in connection with the combinatorics of cluster algebras, 
see e.g. \cite{DiKe}, \cite{DiF}, \cite{KeVi}.

A geometric interpretation of $\SL_{k+1}$-friezes leads to the classical
moduli spaces of configurations of points in projective spaces.
The latter spaces are, in turn, closely related to the geometry of Grassmannians. 
$\SL_{k+1}$-friezes can also be interpreted as difference equations 
with special monodromy conditions. 
The three realizations of the same space: that of $\SL_{k+1}$-friezes,
moduli spaces, spaces of difference equations is referred as ``triality'' in \cite{MGOST}. 

\subsection{ $\SL_{k+1}$-tilings and projective duality}

Let $M=(m_{i,j})_{i,j\in \Z}$ be an bi-infinite matrix with coefficients in a arbitrary field of characteristic~$0$.
Define the adjacent minors of $A$ of order $r+1$ based on $(i,j)$
as 
\begin{equation}\label{AdjMin}
M^{(r+1)}_{i,j}=\det\left(
\begin{array}{llll}
m_{i,j}&m_{i,j+1}&\ldots&m_{i,j+r}\\[4pt]
m_{i+1,j}&m_{i+1,j+1}&\ldots&m_{i+1,j+r}\\[4pt]
\ldots& \ldots&& \ldots\\
m_{i+r,j}&m_{i+r,j+1}&\ldots&m_{i+r,j+r}
\end{array}
\right).
\end{equation}

\begin{defn}\cite{BeRe}\label{tame}
(1) An $\SL_{k+1}$-tiling is an infinite matrix $M=(m_{i,j})_{i,j\in \Z}$ for which all adjacent minors of order $k+1$ equals 1, i.e. $M^{(k+1)}_{i,j}=1$ for all $i,j\in \Z$.

(2) An $\SL_{k+1}$-tiling is called \textit{tame} if in addition all adjacent minors of order $k+2$ vanish,
i.e. $M^{(k+2)}_{i,j}=0$ for all $i,j\in \Z$.
\end{defn}

The tameness condition is understood as a condition of genericity on the tiling.
 For instance, every $\SL_{k+1}$-tiling with non-zero adjacent minors of order $k$ is tame, due to the Desnanot-Jacobi, or Sylvester, identity
$$
M^{(r+1)}_{i,j}M^{(r-1)}_{i+1,j+1}=M^{(r)}_{i,j}M^{(r)}_{i+1,j+1}-M^{(r)}_{i,j+1}M^{(r)}_{i+1,j}.
$$

\begin{ex}
\label{sl2CoCox}
Every Coxeter frieze uniquely extends to a tame $\SL_{2}$-tiling, \cite{BeRe}.
For instance the frieze \eqref{exCoCox} extends to 
\begin{center}
\input{CoCoxSL2}
\end{center}
\end{ex}

\begin{defn}\cite{BeRe}\label{defDerived}
(1) The $r$-\textit{derived array} of the $\SL_{k+1}$-tiling $M=(m_{i,j})_{i,j\in \Z}$ is the bi-infinite matrix defined by
$$
\partial_{r}M:=(M^{(r)}_{ij})_{i,j\in \Z}.
$$

(2) The  $k$-derived array is called the \textit{projective dual} of $M$ and denoted by~$M^{*}$.
\end{defn}

The link to classical projective duality will be explained in \S \ref{modsp}.

\begin{prop}[\cite{BeRe}] Let $M$ be a tame $\SL_{k+1}$-tiling.

(i) The projective dual of $M$ is also a tame $\SL_{k+1}$-tiling. 

(ii) One has the following correspondence between the derived arrays of $M$ and $M^{*}$:
$$
(\partial_{r} M)_{i,j}=(\partial_{k+1-r} M^{*})_{i+r-1, j+r-1}.
$$
In particular $(M^{*})^{*}$ and $M$ coincide up to a shift of indices.
\end{prop}

\subsection{$T$-systems}

A $T$-system of type $\mA_k$ is the following recurrence on the variables  $\{T_{\a,u,v}\}_{\a,u,v}$:
\begin{equation}\label{Tsys}
T_{\a,u,v+1}T_{\a,u,v-1}-T_{\a,u+1,v}T_{\a,u-1,v}=T_{\a+1,u,v}T_{\a-1,u,v}
\end{equation}
with $\a\in\{0,1,\ldots, k,k+1\}$, $u,v\in \Z$, and boundary conditions 
\begin{equation}\label{TsysBound}
T_{0,u,v}=T_{k+1,u,v}=1,
\end{equation}
for all $u,v\in \Z$.
It is nothing but the octahedral recurrence subject to boundary conditions.

\setlength{\unitlength}{1844sp}%
\begingroup\makeatletter\ifx\SetFigFont\undefined%
\gdef\SetFigFont#1#2#3#4#5{%
  \reset@font\fontsize{#1}{#2pt}%
  \fontfamily{#3}\fontseries{#4}\fontshape{#5}%
  \selectfont}%
\fi\endgroup%
\begin{picture}(5810,4412)(-3326,-6698)
\put(2301,-6598){\makebox(0,0)[lb]{\smash{{\SetFigFont{8}{14.4}{\rmdefault}{\mddefault}{\updefault}{\color[rgb]{0,0,0}$\a-1,u,v$}%
}}}}
\thicklines
{\color[rgb]{0,0,0}\put(3601,-5011){\line(-1, 0){2700}}
}%
{\color[rgb]{0,0,0}\put(5401,-4111){\line(-2,-1){1800}}
}%
{\color[rgb]{0,0,0}\put(3151,-6361){\line( 1, 1){2250}}
}%
{\color[rgb]{0,0,0}\put(3151,-6361){\line( 1, 3){450}}
}%
{\color[rgb]{0,0,0}\put(3151,-6361){\line(-5, 3){2250}}
}%
\thinlines
{\color{gray}\put(3151,-2761){\line(-1,-3){450}}
}%
\thicklines
{\color[rgb]{0,0,0}\put(3151,-2761){\line( 1,-5){450}}
}%
{\color[rgb]{0,0,0}\put(3151,-2761){\line( 5,-3){2250}}
}%
{\color[rgb]{0,0,0}\put(3151,-2761){\line(-1,-1){2250}}
}%
\thinlines
{\color{gray}
\put(5401,-4111){\line(-1, 0){2700}}
\put(2701,-4111){\line(-2,-1){1800}}
}%
\put(5626,-4111){\makebox(0,0)[lb]{\smash{{\SetFigFont{8}{14.4}{\rmdefault}{\mddefault}{\updefault}{\color[rgb]{0,0,0}$\a,u,v+1$}%
}}}}
\put(-526,-5036){\makebox(0,0)[lb]{\smash{{\SetFigFont{8}{14.4}{\rmdefault}{\mddefault}{\updefault}{\color[rgb]{0,0,0}$\a,u,v-1$}%
}}}}
\put(2026,-4111){\makebox(0,0)[lb]{\smash{{\SetFigFont{8}{14.4}{\rmdefault}{\mddefault}{\updefault}{\color{gray}$\a,u-1,v$}%
}}}}
\put(3726,-5136){\makebox(0,0)[lb]{\smash{{\SetFigFont{8}{14.4}{\rmdefault}{\mddefault}{\updefault}{\color[rgb]{0,0,0}$\a,u+1,v$}%
}}}}
\put(2301,-2636){\makebox(0,0)[lb]{\smash{{\SetFigFont{8}{14.4}{\rmdefault}{\mddefault}{\updefault}{\color[rgb]{0,0,0}$\a+1,u,v$}%
}}}}
{\color{gray}\put(2701,-4111){\line( 1,-5){450}}
}%
\end{picture}%

A $T$-system splits into two independent subsystems 
$$\{T_{\a,u,v}\}=\{T_{\a,u,v}\,:\, \a+u+v\text{ even }\}\sqcup\{T_{\a,u,v}\,:\, \a+u+v \text{ odd }\},$$
each subsystem satisfying the recurrence \eqref{Tsys}. 
In the sequel we will consider $\a+u+v$ even.

\begin{thm}[\cite{BaRe},\cite{KNS}] If $\{T_{\a,u,v}\}_{\a,u,v}$ satisfy \eqref{Tsys} then for all $0\leq \a \leq k$,
$u,v\in \Z$:
$$
T_{\a+1\,,\,u\,,\,v}=\det
\left(
\begin{array}{llll}
T_{1\, ,\,u\, ,\,v-\a}&T_{1\, ,\,u+1\, ,\,v+1-\a}&\cdots&T_{1\, ,\,u+\a\, ,\,v}\\
T_{1\, ,\,u-1\, ,\,v+1-\a}&T_{1\, ,\,u\, ,\,v+2-\a}&\cdots&T_{1\, ,\,u+\a-1+\a\, ,\,v+1}\\
\vdots&&&\vdots\\
T_{1\, ,\,u-\a\, ,\,v}&T_{1\, ,\,u+1-\a\, ,\,v+1}&\cdots&T_{1\, ,\,u ,\,v+\a}
\end{array}
\right).
$$
\end{thm}

Applying the above result with $\a=k$ one deduces that the first layer
 of a $T$-system $\{T_{1,u,v}\}_{u,v}$ forms an $\SL_{k+1}$ tiling, and the next layers are obtained as 
 derived arrays of this tiling. More precisely, as noticed in \cite{BeRe}, one has the following result.
 
\begin{cor}\label{corTsys}
Let $\{T_{\a,u,v}\}_{\a,u,v}$ be a solution of  \eqref{Tsys},\eqref{TsysBound}.
Set
$
M=(m_{i,j})_{i,j}
$
with $m_{i,j}=T_{1,j-i,i+j}$.
Then, $M$ is an  $\SL_{k+1}$-tiling and
$$
\partial_{\a}M_{i,j}=T_{\a,u,v}
$$
for all $\a=1, \ldots, k$ and $v=i+j+\a, u=j-i$.
Conversely, every  $\SL_{k+1}$-tiling gives rise to a solution of  \eqref{Tsys}. 
\end{cor}

\subsection{Periodicity of $\SL_{k+1}$-friezes}

An $\SL_{k+1}$-tiling $F=(f_{i,j})_{i,j}$ is called an \textit{$\SL_{k+1}$-frieze of width $w$}  if, 
in addition to the condition $M^{(k+1)}_{i,j}=1$, it satisfies
the following ``boundary conditions''
$$
\left\{
\begin{array}{rccccl}
f_{i,i-1}&=&f_{i,i+w}&=&1& \hbox{for all}\; i,\\[4pt]
f_{i,i-1-\ell}&=&f_{i,i+w+\ell}&=&0&\hbox{for}\;1\leq \ell \leq k.
\end{array}
\right.
$$

\begin{rem}
$\SL_{k+1}$-friezes of positive numbers satisfying the extra condition that all minors $A^{(r)}_{1,j}=1$ for $2\leq r \leq k$ and $1\leq j \leq w$,
were first considered in \cite{CoRo}.
Such arrays were shown to be $(k+w+2)$-periodic
(this generalizes Coxeter's Theorem \ref{thmPerio}). 
This periodicity holds true for all tame $\SL_{k+1}$-friezes.
\end{rem}

\begin{thm}[\cite{MGOST}]\label{PeriThmGen}
Every tame $\SL_{k+1}$-friezes $(f_{i,j})_{i,j}$  satisfies, for all $i,j$,
$$
f_{i,j}=(-1)^{k}f_{i+k+w+2,j}, \quad \text{and } \quad f_{i,j}=(-1)^{k}f_{i,j+k+w+2}
$$
In particular, for all $i,j$, one has
$$
f_{i,j}=f_{i+k+w+2,j+k+w+2}.
$$
\end{thm}

The above periodicity of tame $\SL_{k+1}$-friezes has been announced in \cite{BeRe};
another proof is also given in \cite{KeVi} in the context of $T$-systems.
It turns out that this periodicity can be interpreted as Zamolodchikov's periodicity for systems of type $\mA_{k}\times \mA_w$, 
established in this case in \cite{Vol}.

\begin{rem}
When the frieze is not tame one may observe different phenomena. 
For instance, the array in Example \ref{wrong_per} can be extended to a 6-periodic $\SL_{2}$-frieze of width 2. The second array in Remark \ref{fri_nonper} leads to a non-periodic $\SL_{2}$-frieze of width 3. For interesting properties of non-tame friezes see \cite{Cun2}.
\end{rem}

We will display the $\SL_{k+1}$-friezes as follows, and often omit the $k$ bordering rows of 0's
(note a slight change in the notation  by a horizontal flip compare to the notation for Coxeter's friezes):

$$
\begin{array}{ccccccccccccc}
&&&& \vdots&&&& \vdots&&&\\
&0&&0&&0&&0&&0&&\ldots\\[6pt]
\ldots&&1&&1&&1&&1&&1&\\[6pt]
&\ldots&&\;f_{0,w-1}&&\;f_{1,w}&&\;f_{2,w+1}&&\ldots&&\ldots\\
&&&\! \iddots&& \iddots&& \iddots&&&&\\
\ldots&& f_{0,1}&&f_{1,2}&&f_{2,3}&&f_{3,4}&&f_{4,5}&\\[6pt]
& f_{0,0}&&f_{1,1}&&f_{2,2}&&f_{3,3}&&f_{4,4}&&\ldots\\[6pt]
\ldots&&1&&1&&1&&1&&1&\\[6pt]
&0&&0&&0&&0&&0&&\ldots\\
&&&& \vdots&&&& \vdots&&&\\
\end{array}
$$
We denote by $\F_{k+1,n}$ the set of tame $\SL_{k+1}$-friezes of width $w=n-k-2$. 

The set $\F_{k+1,n}$ has a natural structure of algebraic variety (for the Coxeter's case $k=1$, see \S\ref{inffri}, and for the general case, see \S \ref{supeq}).

\subsection{Friezes, superperiodic equations and Grassmannians}\label{supeq}

The results of \S\ref{recRelcox} and \S\ref{superper} generalize to $\SL_{k+1}$-friezes.
Consider the following general linear difference equation

\begin{equation}
\label{REq}
V_{i}=a_{i}^1V_{i-1}-a_{i}^{2}V_{i-2}+ \cdots+(-1)^{k-1}a_{i}^{k}V_{i-k}+(-1)^{k}V_{i-k-1},
\end{equation}
with coefficients $a_{i}^{j}\in \R$, where $i\in\Z$ and $1\leq j\leq k$ 
(note that the superscript $j$ is an index, not a power),
and where the sequence $(V_i)$ is the unknown, or solution.
The entries in a tame $\SL_{k+1}$-tiling turn out to be solutions to such equations, \cite{BeRe}, \cite{DiKe}.
For the $\SL_{k+1}$-friezes one has precisely the following.

\begin{thm}[\cite{MGOST}]\label{REqThmGen} 
Given a tame $\SL_{k+1}$-frieze $F=(f_{i,j})_{i,j}$ of width $w$
and let $n=k+w+2$,
for every fixed $i_0$, the sequence $(V_i)_i$ defined by
$$
V_i:=f_{i_0,i}
$$ 
satisfies the equation \eqref{REq} with $n$-periodic coefficients
\begin{equation*}
\label{coeff}
a_{i}^{j}=
\left| 
\begin{array}{rrc}
f_{i-j+1, i-j+1}&\ldots&f_{i-j+1, i}\\[4pt]
1\qquad\ddots&& \vdots\\[2pt]
\qquad\ddots&\quad \ddots& \vdots\\[4pt]
&1&f_{i,i}
\end{array}
\right|.
\end{equation*}
\end{thm}

\begin{rem}
The above coefficients $a_{i}^{j}$ 
are adjacent minors of order $j$ in the array $F$,  denoted by $F_{i-j+1, i-j+1}^{(j)}$, cf. \eqref{AdjMin}.
One also has $a_{i}^{j}=F_{i+2, i+1+w}^{(k-j+1)}$.
 \end{rem}

\begin{defn}[\cite{Kri},\cite{MGOST}] \label{superdef}
An equation of the form \eqref{REq} is called $n$-\textit{superperiodic} if it satisfies the two conditions:
\begin{itemize}
\item  all coefficients are $n$-periodic, 
i.e. $a_{i+n}^{j}=a_{i}^{j}$ for all $i,j$, and 
\item all solutions are $n$-antiperiodic, i.e. satisfy $V_{i+n}=(-1)^kV_i$, for all $i$.
\end{itemize}
 \end{defn}
 We denote by $\E_{k+1,n}$ the set of linear difference equations of order $k+1$ that are $n$-superperiodic. 
 The second condition in Definition \ref{superdef} is actually a condition on the $nk$ coefficients of the equation. It gives rise to $k(k+2)$ independent polynomials relations, see e.g. \cite{MGOST}.  
 The set $\E_{k+1,n}$ is an algebraic subvariety of $\R^{kn}$, or $\C^{kn}$, of codimension $k(k+2)$.
 
 By combining Theorems \ref{PeriThmGen} and \ref{REqThmGen} one deduces that friezes give rise to superperiodic equations. 
 The converse is also true, more precisely one has the following.
 
 \begin{thm}[\cite{MGOST}]
 The spaces $\F_{k+1,n}$ and $\E_{k+1,n}$ are isomorphic algebraic varieties, for all integers $k$ and $n$.
 \end{thm}
 
Let $\Gr_{k+1,n}$ be the Grassmannian,
i.e., the variety of $k+1$ dimensional subspaces in the vector space of dimension $\C^n$,
and $\Gr^o_{k+1,n}\hookrightarrow\Gr_{k+1,n}$ the open subset that can be represented by 
 $(k+1)\times n $ matrices whose adjacent minors of order $k+1$ do not vanish.
A natural embedding of the space of friezes into the Grassmannian:
 $$
 \F_{k+1,n}\, \hookrightarrow \,\Gr^o_{k+1,n}\hookrightarrow \Gr_{k+1,n},
 $$
 is given by ``cutting'' the following $(k+1)\times n $ matrix
\begin{equation}\label{grass}
\left(
\begin{array}{cccccccccccccc}
 1 &   f_{1,1}& \ldots   &\ldots&f_{1,w}&1& \\[4pt]
& \ddots &\ddots &  &&\ddots&\ddots&\\[12pt]
&    & 1 & f_{k+1,k+1} &\ldots& \ldots&f_{k+1,n-1} &1\\[4pt]
\end{array}
\right)
 \end{equation}
in the frieze.
 
\subsection{Moduli space of polygons in the projective space}\label{modsp}
A {\it non-degenerate} $n$-gon is a map 
$$
v:\Z\to\CP^{k}
$$ 
such that $v_{i+n}=v_i$, for all $i$, and no $k+1$ consecutive vertices belong to
the same hyperplane.
We denote by ${\mathcal C}_{k+1,n}$ the space of equivalence classes of 
non-degenerate $n$-gons in~$\CP^{k}$, modulo projective transformations
(i.e., modulo $\PGL_{k+1}$-action).

The Gelfand-McPherson correspondence \cite{GeMc} gives the following identification
$$
{\mathcal C}_{k+1,n}\simeq \Gr^o_{k+1,n}/ (\C^*)^{n-1},
$$
that can be easily understood via choosing a representative for $v$ in
${\mathcal C}_{k+1,n}$, and lifting $(v_1,\ldots,v_n)$ to~$\C^{k+1}$.
Such a lifting  is defined on the class of $v$ up to non-zero multiples, 
i.e. up to the action of the torus $(\C^*)^{n-1}$.

 \begin{thm}[\cite{MGOST}]\label{ThmTriality}
If $k+1$ and $n$ are coprime, then there is an isomorphism of algebraic varieties:
$$
\E_{k+1,n}\simeq \F_{k+1,n}\simeq {\mathcal C}_{k+1,n}.
$$
\end{thm}
 
 The above isomorphism is obtained by the composition of maps:
$\F_{k+1,n}\, \hookrightarrow \,\Gr^o_{k+1,n}\, \twoheadrightarrow \, {\mathcal C}_{k+1,n}$.
If $k+1$ and $n$ are not coprime, then this map is a projection with a non-trivial kernel.
 
 More explicitly,
 starting from an $n$-gon $v$,  there is a unique lift of $v$ to $V=(V_i)_i$ with $V_i\in \C^{k+1}$ such that $V_{i+n}=(-1)^{k}V_{i}$ and
 $$
 \det (V_{i}, V_{i+1}, \ldots, V_{i+k})=1
 $$
 for all $i$, provided $k+1$ and $n$ are coprime (cf. the case $k=1$ explained in \S\ref{M0n}). 
 Moreover the sequence satisfies relations of the form $\eqref{REq}$, with coefficients that are independent of the choice of $v$ modulo $\PGL_{k+1}$.
 Thus the class of $v$ defines a unique superperiodic equation, i.e. an element of $\E_{k+1,n}$.
 Modulo the action of $\SL_{k+1}$ the sequence $V$ can be normalized so that
 $(V_0, V_1, \ldots, V_{n-1})
  \in (\C^{k+1})^n
$
is of the matrix form \eqref{grass}.
This matrix extends to a unique element of $\F_{k+1,n}$, independent of the choice of $v$ modulo $\PGL_{k+1}$.
Conversely, given a frieze $F=(f_{i,j})$ in $\F_{k+1,n}$, every subarray $(f_{i,j})_{r\leq i\leq r+k, 0\leq j \leq n-1}$ of size $(k+1)\times n$
defines the same sequence $(V_0, \ldots, V_{n-1})$ modulo $\SL_{k+1}$ and by projection $\C^{k+1}\to\CP^{k}$ defines a unique element of ${\mathcal C}_{k+1,n}$.

\begin{defn}
Let  $v=(v_i)$  be a non-degenerate $n$-gon  in $\CP^{k}$. For each hyperplane containing the $k$ points $(v_i,   \ldots, v_{i+k-1})$ 
we denote by $v_{i+k}^*$ the corresponding element in $\bbP((\C^{k+1})^*)=\CP^{k}$. 
The sequence $v^*=(v_i^*)$ is called the \textit{projective dual} $n$-gon of $v$. 
\end{defn}

The duality commutes with the action of $\PGL_{k+1}$, so that 
the map $*: {\mathcal C}_{k+1,n}\to {\mathcal C}_{k+1,n}$ is well-defined.
Moreover this duality coincides with the projective duality on the friezes of Definition \ref{defDerived}, i.e. 
the following diagram commutes
$$
\xymatrix{
F\in \F_{k+1,n}\ar_{*}[d]\ar^\simeq[r]&v\in \cC_{k+1,n}\ar^{*}[d]\\
F^*\in \F_{k+1,n}\ar^\simeq[r]&v^*\in \cC_{k+1,n}.}
$$

\begin{rem}
In terms of equations the projective duality gives the following equation:
$$
V^*_{i}=a_{i+k-1}^kV^*_{i-1}-a_{i+k-2}^{k-1}V^*_{i-2}+ 
\cdots+(-1)^{k-1}a_{i}^1V^*_{i-k}+(-1)^{k}V^*_{i-k-1},
$$
dual, or adjoint to equation \eqref{REq}. 
This implies that in terms of friezes the dual array $F^*$ is obtained from $F$ by performing an horizontal reflection and a horizontal shift.
\end{rem}

\subsection{Gale duality on friezes and on difference operators}\label{Gale}

The classical Gale transform is a map 
$$
G: {\mathcal C}_{k+1,n}\to {\mathcal C}_{w+1,n},
$$ 
where $k+w+2=n$, see \cite{Gal}, \cite{EiPo}.
This map is nothing but the duality of the Grassmannians
$\Gr_{k+1,n}\simeq\Gr_{w+1,n}$ combined with the Gelfand-McPherson correspondence.

We define the map $\Gc: {\mathcal F}_{k+1,n}\to {\mathcal F}_{w+1,n}$ 
that we call ``combinatorial Gale transform''.
One has the following commutative diagram:
$$
\xymatrix{
\F_{k+1,n}\ar_{\Gc}[d]\ar@{^{(}->}[r]& \Gr_{k+1,n}\ar@{<->}^{\wr}[d]\ar@{->>}[r]&\cC_{k+1,n}\ar_{G}[d]\\
\F_{w+1,n}\ar@{^{(}->}[r]& \Gr_{w+1,n}\ar@{->>}[r]&\cC_{w+1,n}
.}
$$
If $k+1$ and $n$ are coprime, then the map $\Gc$ tautologically coincides with $G$,
otherwise, this is a non-trivial generalization.

The explicit construction of the combinatorial Gale transform is as follows.
Let $F=(f_{i,j})_{i,j}$ be a tame $\SL_{k+1}$-frieze of width $w$, and let $n=k+w+2$.
Consider the coefficients $a_i^j$ defined in \eqref{coeff}, and form $F^\Gc=(f^\Gc_{i,j})_{i,j}$ the following array
$$
\begin{array}{ccccccccccccccccccccccccccc}
&\ldots&1&&1&&1&&1&&1&&1\\[4pt]
&&&\ldots&& a^1_n&&a^1_1&&a^1_2&& \ldots&& a^1_n&\\[4pt]
&&&& a^2_n&&a^2_1&& a^2_2&&&&a^2_n&\\
&\ldots&& \iddots && \iddots&& \iddots&&&& \iddots&&\ldots\\
&&a^k_n&& a^k_1&&a^k_2&& \ldots&& a^k_n&&\ldots\\[4pt]
&1&&1&&1&&1&&1&&1&&\ldots\\
\end{array}
$$
where $f^\Gc_{i,j}=a_{i-1}^{k-j+i}$.

\begin{thm}[\cite{MGOST}]\label{} 
The array $F^\Gc$ is a tame $\SL_{w+1}$-frieze of width $k$.
\end{thm}

\noindent
We call the frieze $F^\Gc$ the \textit{Gale dual}, of $F$.

Equivalently, this duality can be expressed in terms of superperiodic equations.
Let $V$ be an $n$-superperiodic equation of the form \eqref{REq}. Consider 
the equation $V^\Gc$ of the form
\begin{equation}
\label{TheDualEq}
V^\Gc_i=\alpha_{i}^1V^\Gc_{i-1}-\alpha_{i}^2V^\Gc_{i-2}+ 
\cdots+(-1)^{w-1}\alpha_{i}^wV^\Gc_{i-w}+(-1)^{w}V^\Gc_{i-w-1},
\end{equation}
where the coeffcients are $n$-periodic given by
\begin{equation}
\label{coeffG}
\alpha_{i}^{w-j}=
\left| 
\begin{array}{llllll}
a_{i+1}^1&1&\\[8pt]
a_{i+2}^2&a_{i+2}^1&1&\\[8pt]
\vdots&\ddots&\ddots&\;1\\[6pt]
a_{i+j+1}^{j+1}&\cdots&a_{i+j+1}^{2}&a_{i+j+1}^{1}
\end{array}
\right|.
\end{equation}
The above equation \eqref{TheDualEq} is $n$-superperiodic.
This is the  \textit{Gale dual}, of equation~\eqref{REq}.
Note that the formula~\eqref{coeffG} is a generalization 
of the determinantal formula \eqref{matdet} for the Coxeter's friezes.

A beautiful application of the combinatorial Gale transform
and projective duality was given by Krichever~\cite{Kri} in terms of
the old classical problem of commuting difference operators.
Associate a difference operator to every equation \eqref{REq}:
$$
L:=a_{i}^1\Td-a_{i}^{2}\Td^2+ \cdots+(-1)^{k-1}a_{i}^{k}\Td^k+(-1)^{k}\Td^{k+1},
$$
where $\Td$ is the shift operator acting on the space of infinite sequences $\Td V_i=V_{i-1}$.
The operator~$L$ is superperiodic if the eigenspace 
$\ker(L-\id)$ is contained in the eigenspace $\ker(\Td^n-(-1)^k\id)$.
The projective duality and combinatorial Gale transform then can be viewed
in terms of operators instead of equations.
In particular, one obtains the superperiodic operator $L^{*\Gc}$,
corresponding to the frieze~$F^{*\Gc}$.

\begin{thm}[\cite{Kri}]
If $k+1$ and $n$ are coprime, 
then the operators $L$ and $L^{*\Gc}$ commute.
\end{thm}

\begin{rem}\label{SLbox}
As mentioned in Corollary \ref{corTsys}, an $\SL_{k+1}$-tiling and its derived arrays form a $T$-system.
Such system can be pictured in the 3D-space, inside a $k\times w \times n$ box (with $n=k+w+2$). 
We see $\a$ as the vertical coordinate and $(i,j)$ as coordinates in the horizontal plane. 
Given a frieze $F$, its derived arrays $\partial_2F$,  $\partial_3F, \ldots $ 
lie on horizontal planes $\a=1, 2, 3, \ldots$, respectively.
The top plane $\a=k$ contains the projective dual frieze $F^*$. Up to a shift of indices, the Gale dual frieze $F^\Gc$ is located on the vertical plane $i=j$, and its derived arrays, $\partial_2F^\Gc$,  $\partial_3F^\Gc, \ldots $ lie on vertical planes $j-i= 1, 2, \ldots$, respectively. The last plane, $i-j=w-1$,
consists of the array $(F^\Gc)^*$ which is also equal to $(F^*)^\Gc$ up to a shift of indices.\\

\setlength{\unitlength}{1544sp}%
\begingroup\makeatletter\ifx\SetFigFont\undefined%
\gdef\SetFigFont#1#2#3#4#5{%
  \reset@font\fontsize{#1}{#2pt}%
  \fontfamily{#3}\fontseries{#4}\fontshape{#5}%
  \selectfont}%
\fi\endgroup%
\begin{picture}(17300,6007)(857,-7486)
\put(5926,-7036){\makebox(0,0)[lb]{\smash{{\SetFigFont{9}{14.4}{\rmdefault}{\mddefault}{\updefault}{\color[rgb]{0,0,0}$F$}%
}}}}
\put(14026,-7036){\makebox(0,0)[lb]{\smash{{\SetFigFont{9}{14.4}{\rmdefault}{\mddefault}{\updefault}{\color{gray}$F$}%
}}}}
\thinlines
{\color{gray}\put(10463,-4223){\line(-1,-3){787.600}}
\put(9676,-6586){\line( 1, 0){4950}}
}%
{\color{gray}\put(3038,-4223){\line( 1, 1){1012.500}}
\put(4051,-3211){\line( 1, 0){4275}}
}%
{\color{gray}\put(1801,-5911){\line( 1, 1){2025}}
\put(3826,-3886){\line( 1, 0){4950}}
}%
{\color{gray}\put(1576,-6586){\line( 1, 1){2025}}
\put(3601,-4561){\line( 1, 0){4950}}
}%
{\color{gray}\multiput(9001,-7261)(283.84623,283.84623){10}{\line( 1, 1){145.384}}
}%
\thinlines
{\color[rgb]{0,0,0}\multiput(14401,-7261)(283.84623,283.84623){10}{\line( 1, 1){145.384}}
}%
{\color[rgb]{0,0,0}\multiput(15413,-4223)(283.84623,283.84623){10}{\line( 1, 1){145.384}}
}%
{\color[rgb]{0,0,0}\multiput(10013,-4223)(283.84623,283.84623){10}{\line( 1, 1){145.384}}
}%
\thinlines
{\color{gray}\put(11701,-4561){\line( 1, 0){4500}}
}%
\thicklines
{\color[rgb]{0,0,0}\put(16201,-4561){\line( 1, 0){900}}
\put(17101,-4561){\line( 1, 3){1012.600}}
\put(18113,-1523){\line(-1, 0){5400}}
\put(12713,-1523){\line(-1,-3){450}}
}%
\thinlines
{\color{gray}\put(12263,-2873){\line(-1,-3){562.600}}
}%
\thicklines
{\color[rgb]{0,0,0}\put(11138,-3548){\line( 1, 3){225}}
\put(11363,-2873){\line( 1, 0){5400}}
\put(16763,-2873){\line(-1,-3){1012.600}}
\put(15751,-5911){\line(-1, 0){450}}
}%
{\color[rgb]{0,0,0}\put(10463,-4223){\line( 1, 3){225}}
\put(10688,-3548){\line( 1, 0){5400}}
\put(16088,-3548){\line(-1,-3){1012.600}}
\put(15076,-6586){\line(-1, 0){450}}
}%
\thinlines
{\color{gray}\multiput(3601,-4561)(135.36890,406.10670){8}{\line( 1, 3){ 65.018}}
}%
\thinlines
{\color[rgb]{0,0,0}\multiput(9001,-4561)(135.36890,406.10670){8}{\line( 1, 3){ 65.018}}
}%
{\color[rgb]{0,0,0}\multiput(6301,-7261)(135.36890,406.10670){8}{\line( 1, 3){ 65.018}}
}%
{\color[rgb]{0,0,0}\multiput(901,-7261)(135.36890,406.10670){8}{\line( 1, 3){ 65.018}}
}%
\thicklines
{\color[rgb]{0,0,0}\put(3038,-4223){\line(-1,-1){1687.500}}
\put(1351,-5911){\line( 1, 0){5400}}
\put(6751,-5911){\line( 1, 1){2700}}
\put(9451,-3211){\line(-1, 0){1125}}
}%
{\color[rgb]{0,0,0}\put(9226,-3886){\line(-1, 0){450}}
}%
{\color[rgb]{0,0,0}\put(9001,-4561){\line(-1, 0){450}}
}
{\color[rgb]{0,0,0}\put(1801,-5911){\line(-1,-1){675}}
\put(1126,-6586){\line( 1, 0){5400}}
\put(6526,-6586){\line( 1, 1){2700}}
}%
{\color[rgb]{0,0,0}\put(1576,-6586){\line(-1,-1){675}}
\put(901,-7261){\line( 1, 0){5400}}
\put(6301,-7261){\line( 1, 1){2700}}
}%
{\color[rgb]{0,0,0}\put(4613,-1523){\line(-1,-1){2700}}
\put(1913,-4223){\line( 1, 0){5400}}
\put(7313,-4223){\line( 1, 1){2700}}
\put(10013,-1523){\line(-1, 0){5400}}
}%
{\color[rgb]{0,0,0}\put(9001,-7261){\line( 1, 3){1012.600}}
\put(10013,-4223){\line( 1, 0){5400}}
\put(15413,-4223){\line(-1,-3){1012.600}}
\put(14401,-7261){\line(-1, 0){5400}}
}%
\put(1026,-5461){\makebox(0,0)[lb]{\smash{{\SetFigFont{9}{14.4}{\rmdefault}{\mddefault}{\bfdefault}{\color[rgb]{0,0,0}$k$}%
}}}}
\put(7751,-6136){\makebox(0,0)[lb]{\smash{{\SetFigFont{9}{14.4}{\rmdefault}{\mddefault}{\updefault}{\color[rgb]{0,0,0}$w$}%
}}}}
\put(3376,-7586){\makebox(0,0)[lb]{\smash{{\SetFigFont{9}{14.4}{\rmdefault}{\mddefault}{\updefault}{\color[rgb]{0,0,0}$n$}%
}}}}
\put(15826,-1961){\makebox(0,0)[lb]{\smash{{\SetFigFont{9}{14.4}{\rmdefault}{\mddefault}{\updefault}{\color[rgb]{0,0,0}$F^{\Gc*}=F^{*\Gc}$}%
}}}}
\put(15351,-3311){\makebox(0,0)[lb]{\smash{{\SetFigFont{9}{14.4}{\rmdefault}{\mddefault}{\updefault}{\color[rgb]{0,0,0}$\partial_3F^\Gc$}%
}}}}
\put(14788,-3986){\makebox(0,0)[lb]{\smash{{\SetFigFont{9}{14.4}{\rmdefault}{\mddefault}{\updefault}{\color[rgb]{0,0,0}$\partial_2F^\Gc$}%
}}}}
\put(14596,-4661){\makebox(0,0)[lb]{\smash{{\SetFigFont{9}{14.4}{\rmdefault}{\mddefault}{\updefault}{\color[rgb]{0,0,0}$F^\Gc$}%
}}}}
\put(6788,-4000){\makebox(0,0)[lb]{\smash{{\SetFigFont{9}{14.4}{\rmdefault}{\mddefault}{\updefault}{\color[rgb]{0,0,0}$F^*$}%
}}}}
\put(6001,-5686){\makebox(0,0)[lb]{\smash{{\SetFigFont{9}{14.4}{\rmdefault}{\mddefault}{\updefault}{\color[rgb]{0,0,0}$\partial_3F$}%
}}}}
\put(5788,-6361){\makebox(0,0)[lb]{\smash{{\SetFigFont{9}{14.4}{\rmdefault}{\mddefault}{\updefault}{\color[rgb]{0,0,0}$\partial_2F$}%
}}}}
\thinlines
{\color{gray}\put(11138,-3548){\line(-1,-3){787.600}}
\put(10351,-5911){\line( 1, 0){4950}}
}%
\end{picture}%

\medskip

\noindent
The faces of size $k\times w$ located on the walls $i=\mathrm{const}$ 
give a system of coordinates on $\F_{k+1,n}$. 
For every $i=1, \ldots, n$ the values on the face $(\partial_{\a}F_{i,j})_{\a,j}$ are cluster variables 
of the cluster algebra of type $\mA_k\times \mA_w$ and each face forms a cluster. 
This cluster structure on  $\F_{k+1,n}$ is the one of the Grassmanniann $\Gr_{k+1, n}$, \cite{Sco}, in which the frozen cluster variables are evaluated to 1. This is a generalization of the results of \cite{MGOTaif} established in the case $k=2$.
\end{rem}

\section{Friezes of integers and enumerative combinatorics}\label{comb}

Important problems in the theory of friezes concern the friezes with positive integers.
\begin{itemize}
\item How to construct friezes containing only positive integers?
\item How many such friezes do exist?
\item What do the numbers appearing in the friezes count?
\end{itemize}

In the case of Coxeter's friezes, all these questions have been answered thank to a beautiful correspondence
between friezes and triangulations of polygons due to J.H.Conway.
We explain them in the next sections \S\ref{entieres} and \S\ref{interpret}.
In the more general case of multiplicative friezes over a repetition quiver or in the case of $\SL_k$-friezes some answers to the above questions are known.
We expose them in the sections \S\ref{ZQentiere}, \S\ref{SLenum} and
\S\ref{SLentiere}.

Finally, one can ask a converse question. Can one produce new type of friezes using combinatorial models?
In section \S\ref{diss} we present a variant of frieze obtained from arbitrary dissections of polygons.

\subsection{Coxeter's friezes of positive integers and triangulations of polygons}\label{entieres}
Coxeter formulates the problem of obtaining friezes with positive integers \cite{Cox}. 
He obtains a criterion to characterize such friezes (see Proposition \ref{corPolyLaurent} below)
and connects them to the theory of continued fractions.

From Coxeter's results (see Theorems \ref{thmPoly} and \ref{thmLaurentGen} above) one obtains the following immediate corollaries.

\begin{prop}\label{corLaurent}
(i)
If the first row of a given frieze consists of integers then all the entries of the frieze are integers.

(ii)
If a frieze pattern contains a zig-zag of 1's, then all the entries of the frieze are positive integers.
\end{prop}

Let $x_{1}, \ldots, x_{m}$ be the entries on a diagonal of a frieze, with the convention $x_0=x_{m+1}=1$.
Theorems \ref{thmPoly} and \ref{thmLaurent} lead to the following criterion.

\begin{prop}[\cite{Cox}]\label{corPolyLaurent}
If the entries $x_{1}, \ldots, x_{m}$ are all positive integers satisfying the condition
$
x_{i} \text{ divides } x_{i-1}+x_{i+1},
$
for all $1\leq i \leq m$,
then all entries in the frieze are positive integers, and {\it vice versa}.
\end{prop}

Proposition \ref{corLaurent} gives an easy way to generate a frieze of positive integers: it suffices to set 1's on a zig-zag shape in a frieze and to deduce the rest of the entries by applying the unimodular rule. For instance that is how one can get the frieze \eqref{exCoCox}. 
However, there exist friezes of positive integers that cannot be obtained this way. 
This is the case of the following frieze:
$$
 \begin{array}{ccccccccccccccccccc}
&&1&&1&& 1&&1&&1&&1&&\cdots \\[4pt]
&\cdots&&1&&3&&1&&3&&1&&3&&
 \\[4pt]
&&2&&2&&2&&2&&2&&2&&\cdots
 \\[4pt]
 &\cdots&&3&&1&&3&&1&&3&&1&&
 \\[4pt]
&&1&&1&&1&&1&&1&&1&&\cdots
\end{array}
$$

Classification of frieze patterns with positive integers is related to triangulations of polygons. 
The following theorem is credited to John Conway, cf. \cite{CoRi}.

\begin{thm}[\cite{CoCo}]\label{ThCocox}
Frieze patterns of width $m=n-3$ with positive integers are in one-to-one correspondence with the triangulations of a convex $n$-gon.
If $(a_{1}, a_{2}, \ldots, a_{n})$ is the cycle on the first row of the frieze, then $a_{i}$ is the number of triangles adjacent to the $i$-th vertex in the corresponding triangulated $n$-gon.
\end{thm}

The term \textit{quiddity of order $n$} is introduced in \cite{CoCo} to refer to a sequence $(a_{1}, a_{2}, \ldots, a_{n})$ of positive integers defining the first row of a frieze of width $n-3$, or equivalently,
defining a triangulation of an $n$-gon by its numbers $a_{i }$ of triangles incident to each vertex.

\begin{ex}\label{exQuid}
The frieze \eqref{exCoCox} has quiddity $(4, 2, 1, 3, 2, 2, 1)$ and corresponds to the following triangulated heptagon. 
\begin{center}
\includegraphics[width=7cm]{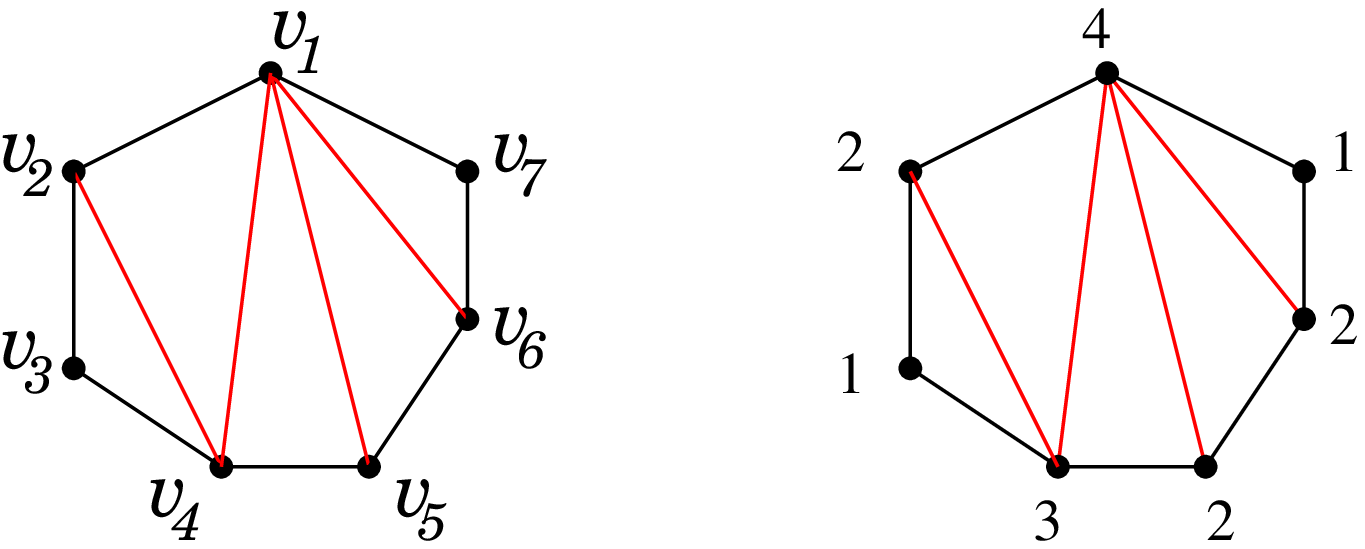}
\label{triang}
\end{center}
\end{ex}

\begin{cor}\label{corCoCox}
For a fixed width $m$, the number of friezes with positive integers is finite, given by the $(m+1)$-th Catalan number 
$$
C_{m+1}=\frac{1}{m+1} {{2(m+1)}\choose{m+1}}.
$$
\end{cor}

Let us mention that relation between the Catalan numbers and matrices of the form~\eqref{matdet}
was found independently in the 80's, see~\cite{Sha} and references therein.

\subsection{Combinatorial interpretations of the entries in a Conway-Coxeter frieze}\label{interpret}

Coxeter's friezes with positive integers are often called \textit{Conway-Coxeter friezes}.

All the entries in a Conway-Coxeter frieze can be combinatorially interpreted using the corresponding triangulated $n$-gon. 
We denote cyclically by $v_1, \ldots, v_n$ the vertices of the triangulated $n$-gon (with convention $v_{i+n}=v_{i-n}=v_{i}$).
Recall that the entries in the frieze are denoted by $e_{i,j}$ as in \eqref{label}, and,
by Theorem~\ref{ThCocox}, the entry $e_{i,i}$ counts 
the number of triangles incident to $v_i$.

We give three elementary ways to compute the value $e_{i,j}$ in the frieze from the  corresponding triangulated $n$-gon. \\

1) Using the ``length'' of the diagonals, defined as follows.

Assume that the sides of the polygon and the diagonals in the initial triangulation are all of length~1. 
The length of all other diagonals can be computed recursively using the Ptolemy rule:
$$
   \xymatrix  @!0 @R=2em @C=2pc { 
& B \ar@{-}[rrd]\ar@{-}[rdd] & \\ 
A\ar@{-}[ru]\ar@{-}[rrr]\ar@{-}[rrd]&&&D&\Longrightarrow&&&&&|AD|\cdot |BC|=|AB| \cdot |CD|+|AC| \cdot |BD|. \\
&&C\ar@{-}[ru]\\
}
$$
(In Euclidean geometry the Ptolemy formula holds true for quadrilateral inscribed in a circle.)

In \cite{CoCo}, it is shown that
$$
e_{i,j}= \text{ the length of the diagonal } [v_{i-1} \; v_{j+1}].
$$

Note that the above formula gives $e_{i,j}=1$ if and only if $ [v_{i-1} \; v_{j+1}]$ is a diagonal in the initial triangulation.\\

2) Using the ``counting procedure''.
This procedure is given in \cite{CoCo}.
\begin{itemize}
\item Choose the vertex $v_i$ and assign the tag 0 at this vertex;

\item All the vertices that are joined to $v_i$ in the triangulation are assigned the tag 1; 

\item  Whenever a triangle has two vertices  tagged by $a$ and $b$, assign the tag $a+b$ to the third vertex;

\item When all the vertices are tagged, one has
\begin{equation}\label{entBCI}
e_{i+1,j-1}= \text{the tag at vertex } v_j.
\end{equation}
\end{itemize}
\begin{ex}
We apply the above procedure in the situation of Example \ref{exQuid} in order to obtain the values of the diagonal $(e_{2, \bullet})$ and
$(e_{3, \bullet})$ of the frieze \eqref{exCoCox}.
\begin{center}
\includegraphics[width=7cm]{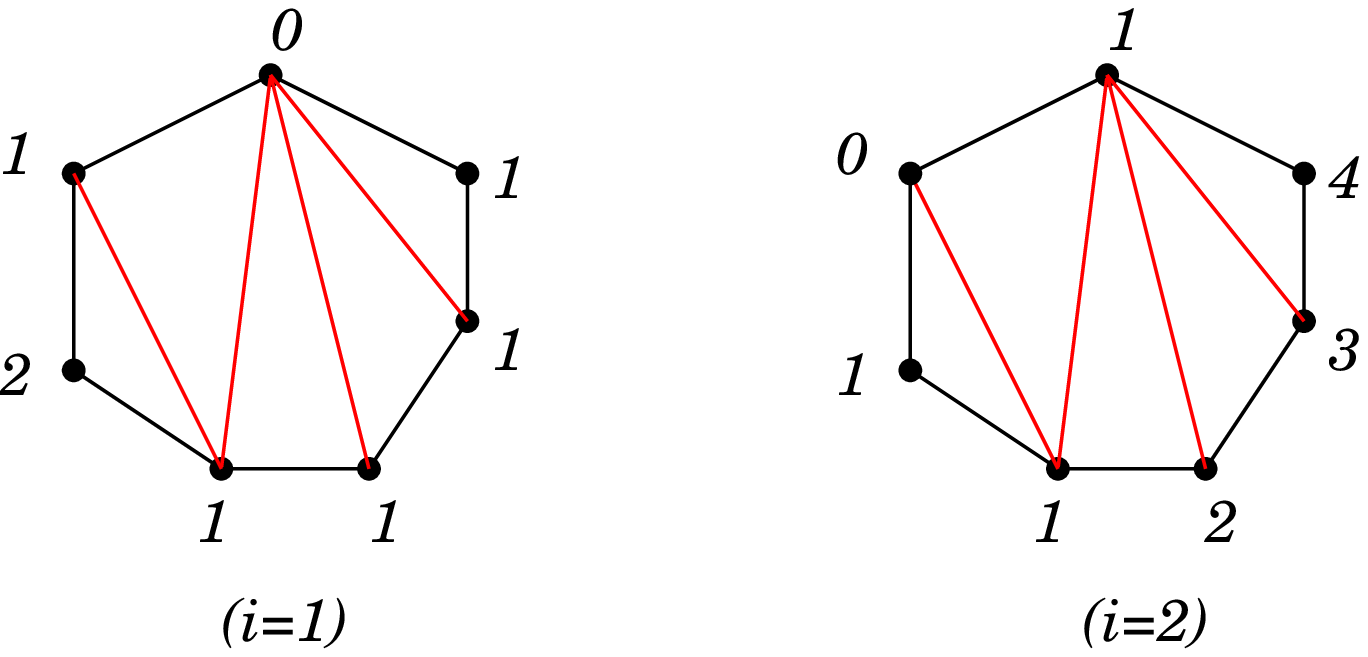}
\label{triang}
\end{center}
\end{ex}

3) Counting the ``admissible paths''. This combinatorial interpretation is given in \cite{BCI}.

An \textit{admissible path} from vertex $ v_{i-1}$ to vertex $v_{j+1}$ is an ordered sequence $(\tau_{i}, \tau_{i+1}, \ldots, \tau_{j})$,
of distinct triangles in the triangulation, such that the triangle $\tau_\ell$ is incident to vertex $v_\ell$. One has:
\begin{equation}\label{entPath}
e_{i,j}= \text{ the number of admissible paths from the vertex } v_{i-1} \text{ to  }v_{j+1}.
\end{equation}

Note that the value $e_{i,i-1}=1$ in the frieze is understood as the unique path $(\,)$, of length 0, between the vertices $v_{i-1}$ and $v_{i}$.

\begin{ex}
In the situation of Example \ref{exQuid}, let $\a, \beta, \gamma, \delta, \epsilon$ be the triangles in the polygon.
\begin{center}
\includegraphics[width=3cm]{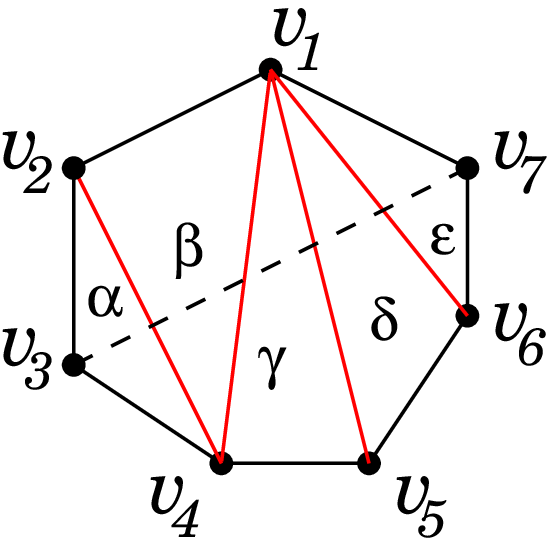}
\label{triang}
\end{center}
Admissibles paths $(\tau_1, \tau_2)$ between vertices $v_7=v_0$ and $v_3$ 
are
$$
(\epsilon, \beta), (\epsilon, \a), (\delta, \beta), (\delta, \a), (\gamma, \beta), (\gamma, \a), (\beta, \a).
$$
Which gives the value $e_{1,2}=7$ in the frieze \eqref{exCoCox}.

Note that the number of admissible paths between two vertices does not depend on the orientation one chooses to go from a vertex to the another.
Indeed, if  $(\tau_1, \tau_2)$ is a path from vertex $v_7$ to $v_3$ then the set of triangles $\{\a, \beta, \gamma, \delta, \epsilon\}\setminus \{\tau_1,\tau_2\}$ defines an admissible path from $v_3$ to $v_7$.
Here the admissible paths $(\tau_4,\tau_5, \tau_6)$ between $v_3$ and $v_7$ are 
$$
(\a, \gamma, \delta), (\beta, \gamma, \delta), (\a, \gamma, \epsilon), (\beta, \gamma, \epsilon), (\a, \delta, \epsilon), (\beta, \delta, \epsilon),
 (\gamma, \delta, \epsilon).
$$
This property gives $e_{i,j}=e_{j+2,i-2}$, which corresponds to the glide symmetry in the frieze.
\end{ex}

\begin{rem}
Other combinatorial models to calculate and interpret the entries of the frieze were suggested by several authors.

(a)
In \cite{Pro} the entries in a frieze are the numbers of perfect matchings on certain bipartite graphs associated to the corresponding triangulation.

(b)
In \cite{Cun} the entries of the friezes are interpreted as root coordinates in some Weyl groupoid.

(c)
Considering Coxeter friezes as friezes on a repetition quiver $\Z\Qc$, with $\Qc$ of type $\mA_{n-3}$, 
the Caldero-Chapoton formula gives the entry $e_{i,j}$ as the Euler characteristic of the total
quiver Grassmannian of the module $M_{i,j}$ of $\rep$ attached to the vertex $(i,j)$ in the AR quiver:
$$
e_{i,j}=\chi(\Gr M_{i,j})=\chi(\sqcup_e \Gr_e M_{i,j})=\sum_e \chi(\Gr_e M_{i,j}).
$$
This provides a geometric interpretation of Proposition \ref{corLaurent}.
See \cite{CaCh}, and \cite{AsDu}.

(d)
 Links between friezes and Farey series were already mentioned in \cite{Cox}.
In \cite{MGOTprep} Conway-Coxeter friezes are classified using cycles in the Farey graph.  
If the triangulated $n$-gon corresponding to the frieze is embedded in the Farey graph so that
$v_{1}=\frac01$ and $v_{n}=\frac10$, then the sequence of vertices $(v_i)$ are  rational points in the Farey graph,
such that the sequence of denominators gives the diagonal $e_{1,\bullet}$ and the sequence of numerators give the diagonal $e_{2,\bullet}$.
For Example \ref{exQuid}, one obtains
\begin{center}
\includegraphics[width=6cm]{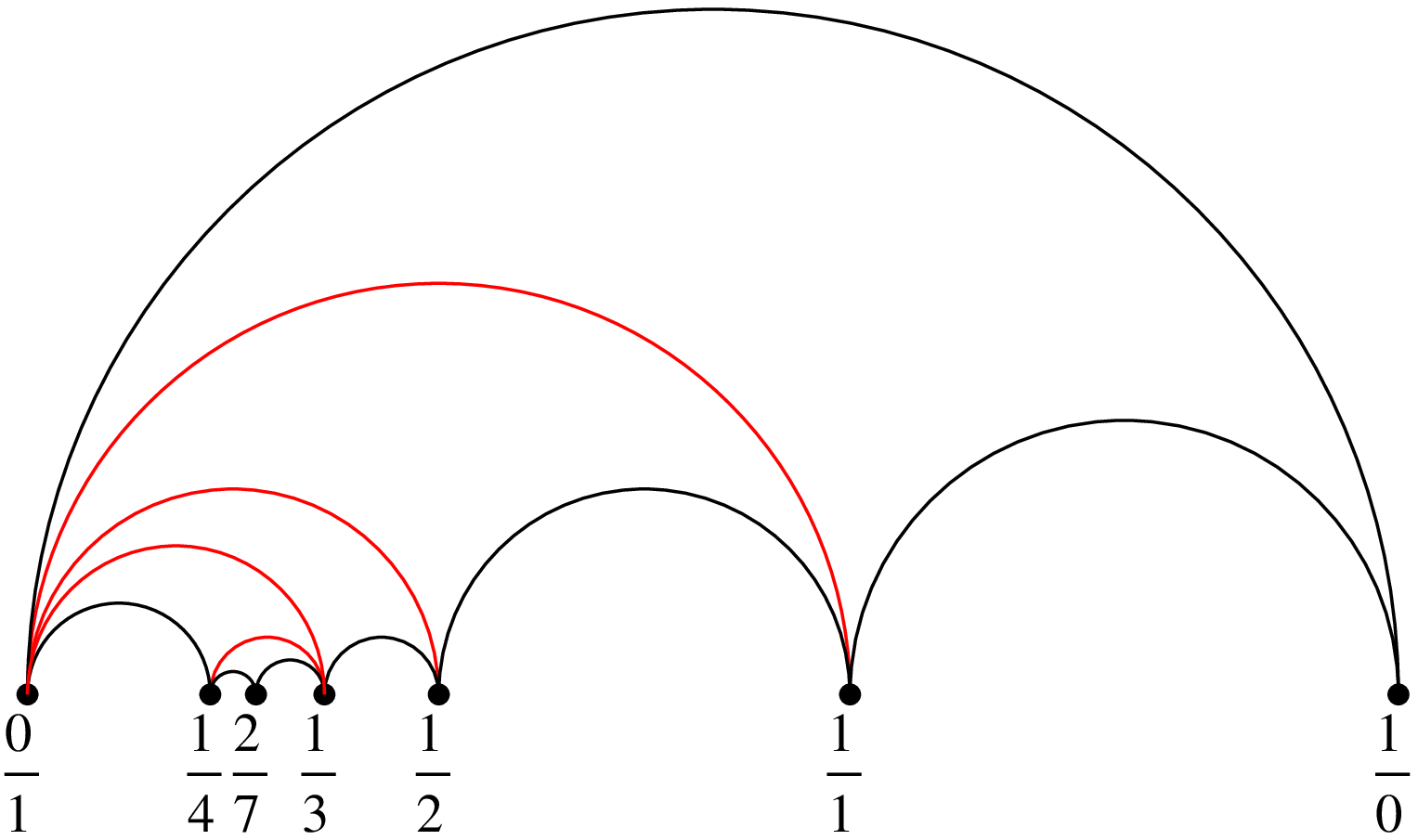}
\label{triangFarey}
\end{center}
which gives the quotient of the diagonal $e_{2,\bullet}=(0,1,2,1,1,1,1)$  by
 $e_{1,\bullet}=(1,4,7,3,2,1,0)$ of \eqref{exCoCox}.
 
 (e) 
The link between friezes, triangulations of polygons and the Farey graph is illustrated by the applet of \cite{SchwApp}.
\end{rem}

\subsection{Multiplicative friezes of type $\Qc$ with integer values}\label{ZQentiere}

Recall that in Section \ref{FrQuiv}, Coxeter's friezes were generalized to multiplicative friezes associated to an acyclic quiver $\Qc$.
Entries in the generic multiplicative frieze $f_{mu}$ over $\Z\Qc$
are cluster variables in the cluster algebra of type $\Qc$.
The general theory and the combinatorial models developed for the cluster algebras 
may be helpful to understand friezes of positive integers.

The result of Proposition \ref{corLaurent} can be generalized in the case of multiplicative friezes over $\Z\Qc$.
A \textit{section} $\Sc$ of $\Z\Qc$ is a full subquiver of $\Z\Qc$ such that $\Z\Sc\simeq \Z\Qc$. 
In other words, a section contains all the vertices of $\Qc$ but not necessarily in the same copy.
The set of variables of $f_{mu}$ lying on a section forms a cluster in the algebra $\A_\Qc$.
By Laurent phenomenon, cf Theorem \ref{LaurentFZ}, one obtains the following corollary, see also \cite{AsDu}.

\begin{cor}
Let $\Qc$ be any acyclic quiver and $f$ a multiplicative frieze over $\Z\Qc$. 
If $f$ takes the constant value $1$ over any given section of $\Z\Qc$, then all the values taken by $f$ are positive integers.
\end{cor}

The converse is not true. A frieze with positive integers does not necessarily have a section of~1's.
A way to construct a frieze of positive integers is to choose an arbitrary cluster in the corresponding cluster algebra and specify all its variables to 1. 
Then all the other cluster variables will evaluate to positive integers. 
Friezes is obtained this way are called \textit{unitary friezes} in \cite{MGjaco}.
However, there exist non-unitary friezes as noticed in the Appendix of \cite{BaMa} and in \cite{MGjaco}.
The number of integral friezes is not the number of clusters in the  corresponding cluster algebra, except for the 
type $\mA_n$ where it is given by a Catalan number, cf Corollary \ref{corCoCox}.

\begin{thm}[\cite{FoPl}]
The number of friezes with positive integers in type $\mD_n$ is
$$
\sum_{m=1}^nd(m) {{2n-m-1}\choose{n-m}},
$$
where $d(m)$ is the number of divisors of $m$.
\end{thm}

The above result is obtained using the combinatorial model of
triangulations of punctured polygons associated with the cluster algebra of type $\mD$.

The exact numbers of friezes with positive integers in type $\mE$ have not been established. 
Several  independent computer programs have enumerated such friezes. The results are:

- 868 friezes found in type $\mE_6$, \cite{Pro}, \cite{MGOTaif}, \cite{FoPl}, \cite{Cun3};

- 4400 friezes found in type $\mE_7$ (under a hypothetical upper bound on the entries), \cite{FoPl};

- 26953 friezes found in type $\mE_8$, \cite{Cun3}; under a hypothetical upper bound on the entries 26592 friezes have beeen found in \cite{FoPl}.\\

Using the finite type classification of cluster algebras, cf Theorem \ref{classFZ}, one obtains the following.
\begin{prop}[\cite{MGjaco}]
If $\Qc$ is an acyclic quiver that is not Dynkin of type $\mA$, $\mD$, $\mE$, then there exist infinitely many multiplicative friezes on $\Z\Qc$ taking positive integer values.
 \end{prop}

\begin{rem} 
(a)
In \cite{BaMa}, the entries in the integral friezes of type $\mD$ are obtained as numbers of certain matchings in a triangulated 
punctured disc.

(b)
In \cite{Fon} the numbers of friezes of Dynkin type taking non-zero integer values are obtained.

(c)
One has the following interesting property of the entries in a multiplicative frieze $f$ over $\Z\Qc$, proved in \cite{ARS} and \cite{KeSc}.
For each $i\in \Qc_0$, the sequence $(f(m,i))_{m\in \Z}$ satisfies a linear recurrence relation if and only if $\Qc$ is a Dynkin or 
affine acyclic quiver.

(d) 
With a more general notion of multiplicative friezes using Cartan matrices
\cite{ARS}, one also gets finitely many integral friezes in the other Dynkin types 
$\mB, \mC$ and $\mG$, \cite{FoPl}.
\end{rem}

\subsection{$\SL_{k+1}$-friezes with positive integer values}\label{SLenum}

Recall that $\SL_{k+1}$-friezes of width $w$ are related to the cluster algebra of type $\mA_k\times \mA_w$, see Remark \ref{SLbox}. 
The problem of constructing $\SL_{k+1}$-friezes with positive integer values is related to the problem of having positive integer values for the cluster variables in type $\mA_k\times \mA_w$. 
The two problems are not \textit{a priori} equivalent as not all cluster variables appear in a $\SL_{k+1}$-friezes.
However, as in the case of multiplicative friezes over a repetition quiver, one has an elementary constuction of $\SL_{k+1}$-friezes with positive integer values by evaluating to 1 all the cluster variables in a chosen cluster.
When the corresponding cluster algebra is of infinite type this produces infinitely many friezes.

\begin{thm}[\cite{MGjaco}]
For $k,w>1$ there exist infinitely many $\SL_{k+1}$-friezes of width $w$ with positive integer values whenever $kw\geq 9$.
\end{thm}

For small values of $k$ and $w$, one has the following known results.

- For  $k=1$: there are finitely many $\SL_{2}$-friezes of width $w$ (this case corresponds to Coxeter's friezes), the number of such friezes is given by a Catalan number, see Corollary \ref{corCoCox}.

- For $w=1$, there are finitely many $\SL_{k+1}$-frieze of width $1$ for every $k$.
This is obtained by Gale duality (see Section \ref{Gale}): an $\SL_{k+1}$-frieze of width $1$ corresponds to an $\SL_{2}$-frieze of width $k$. 

- For $k=2$: there exists exactly 51 $\SL_{3}$-friezes of width 2, conjectured in \cite{Pro} and proved in \cite{MGOTaif}; 868 $\SL_{3}$-friezes of width 3 have been found, \cite{Pro}, \cite{MGOTaif}, \cite{FoPl}, \cite{Cun3}; 26953 $\SL_{3}$-friezes of width 4 have been found \cite{Cun3}.

- For $k=3$: by Gale duality one obtains 868 $\SL_{4}$-friezes of width 2.

\begin{rem}
The quivers $\mA_2\times \mA_3$ and $\mA_2\times \mA_4$, corresponding to $\SL_{3}$-friezes of width 3 and of width 4 respectively, are mutation equivalent to the quivers 
$\mE_{6}$ and $\mE_{8}$ respectively. However the $\SL_{3}$-friezes do not contain all the cluster variables of the corresponding cluster algebras, unlike the multiplicative friezes over the repetition quivers. Therefore the number of $\SL_{3}$-friezes with positive integer values could be greater than the number of multiplicative friezes of the same cluster type.
According to \cite{Cun3}, 868 is the exact number of $\SL_{3}$-friezes of width 3, and by consequent is also the exact number of multiplicative frriezes of type $\mE_{6}$. The number 26953 is not established as the exact number of $\SL_{3}$-friezes of width 4. However all the known 26953
 $\SL_{3}$-friezes  correspond to 26953 multiplicative friezes of type $\mE_{8}$.
\end{rem}



\subsection{$\SL_2$-tilings and triangulations}\label{SLentiere}

Recall that $\SL_2$-tilings are generalizations of Coxeter's friezes by removing the condition of bordering rows of 1's.
They are viewed as bi-infinite matrices $(m_{i,j})_{i,j\in \Z}$ for which every adjacent $2\times 2$ minors are equal to 1.

In \cite{HoJo}, a classification result about $\SL_2$-tilings with positive integer entries is obtained using triangulations of strips.
The $\SL_2$-tilings are assumed to have \textit{enough ones} in the sense that for every couple of indices $(i_0,j_0)$ there exist 
 $i\leq i_0$, $j\geq j_0$ such that $m_{i,j}=1$, and $i'\geq i_0$, $j'\leq j_0$,  such that $m_{i,j}=1$. 
 
The strip can be viewed as an open polygon with infinitely many vertices.
It consists in two disjoint sets of ordered vertices $(v_i)_{i\in \Z}\sqcup(w_i)_{i\in \Z}$ lying on two parallel lines.
A triangulation of the strip is a maximal collection of non-crossing arcs joining either a vertex $v_i$ to a vertex $w_j$ or joining two non consecutive vertices of the same line. A triangulation is called a \textit{good} triangulation of the strip if for every couple of indices $(i_0,j_0)$ there exist an 
arc joining $v_i$ to $w_j$ with $i\geq i_0$ and $j\geq j_0$ and an arc  joining $v_{i'}$ to $w_{j'}$ with $i'\leq i_0$ and $j'\leq j_0$.

\begin{thm}[\cite{HoJo}]
The set of $\SL_2$-tilings of positive integers with enough ones is in bijection with the set of good triangulations of the strip.
\end{thm}

\begin{ex}
Consider the following piece of triangulation of the strip
\begin{center}
\includegraphics[width=10.5cm]{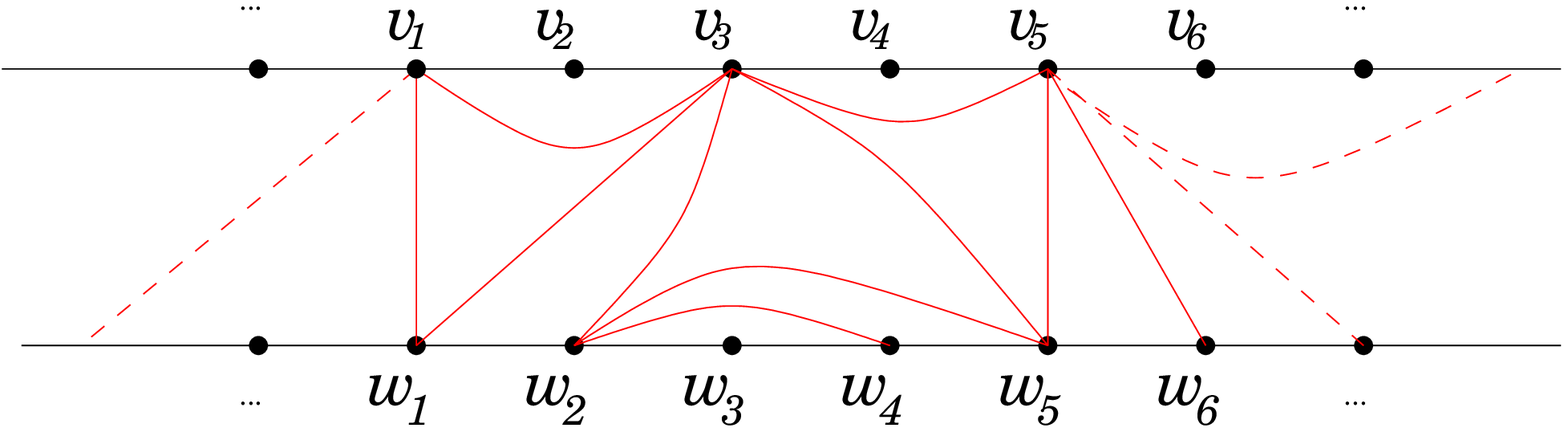}
\end{center}
One can associate an  $\SL_2$-tiling using the same recursive counting procedure as for Conway-Coxeter's friezes, cf \S \ref{interpret} Item 2).
The procedure starts at a vertex $v_i$  labelled by  0, and ends when all the vertices are labelled.
The label at the vertex $w_j$ gives the value $m_{i,j}$ of the tiling. 
For the above piece of strip one obtains the following piece of matrix (where the first row is at the bottom and the row indices increase from bottom to top):
$$
\begin{array}{rrrrrrrr}
&&&\vdots&&&\\[4pt]
&3&2&5&3&1&1&\\
&4&3&8&5&2&3&\\
\cdots&1&1&3&2&1&2&\cdots\\
&2&3&10&7&4&9&\\
&1&2&7&5&3&7\\
&&&\vdots&&&\\
\end{array}
$$
\end{ex}

\begin{rem}
In \cite{BHJ2} the combinatorial model is refined to triangulations of the disc with accumulation points.
It is announced that all $\SL_2$-tilings of positive integers are obtained from triangulations of the disc with four accumulation points.
\end{rem}

Note that  $\SL_2$-tilings containing only positive numbers are never periodic. Indeed, 
every $\SL_2$ matrix $\begin{pmatrix}
a&b\\
c&d
\end{pmatrix}$ of positive numbers satisfies $\frac{a}{c}>\frac{b}{d}$, so that in a  $\SL_2$-tiling 
the ratios of two consecutive rows form a strictly decreasing sequence.

In \cite{MGOTprep} one classifies antiperiodic $\SL_2$-tilings containing a rectangular domain of positive integers.
When writing the array as an infinite matrix, such tilings have the form
$$
 \begin{array}{c|ccc|ccc|ccc|c}
 &&\vdots&&&\vdots&&&\vdots&&\\
 \hline
  &&&&&&&&&&\\[-4pt]
  \cdots &&P&&&-P&&&P&& \cdots\\[4pt]
   \hline
  &&&&&&&&&&\\[-4pt]
  \cdots &&-P&&&P&&&-P&& \cdots\\[4pt]
   \hline
 &&\vdots&&&\vdots&&&\vdots&&
\end{array}
$$
where $P$ is an $m\times n$-matrix with entries that are positive integers.
\begin{thm}[\cite{MGOTprep}]
\label{thethm} 
The set of antiperiodic $\SL_2$-tilings containing a fundamental rectangular domain of positive integers of size $m\times n$
is in a one-to-one correspondence with the set of triples $(q,q', M)$, where
$$
q=(q_0,\ldots,q_{n-1}),
\qquad
q'=(q'_0,\ldots,q'_{m-1})
$$ 
are quiddities of order $n$ and $m$, respectively, and where
$M=\begin{pmatrix}
a&b\\
c&d
\end{pmatrix} \in \SL_{2}(\Z_{>0})$, such that
$$
q_0<\frac{b}{a},
\qquad
q'_0<\frac{c}{a}.
$$
\end{thm}

The explicit construction of the $\SL_2$-tiling is given using a pair of triangulated polygons of quiddities $q$ and $q'$ suitably embedded in the Farey graph according to the matrix $M$,
see \cite{MGOTprep} for details.
\begin{ex}
For the initial data
 $$
 q=(1,2,2,1,3), \qquad q'=(2,1,2,1), \qquad 
 M=\begin{pmatrix}2&5\\[4pt]
 7&18\end{pmatrix}.
 $$
one associates the following two polygons
\begin{center}
\includegraphics[width=12cm]{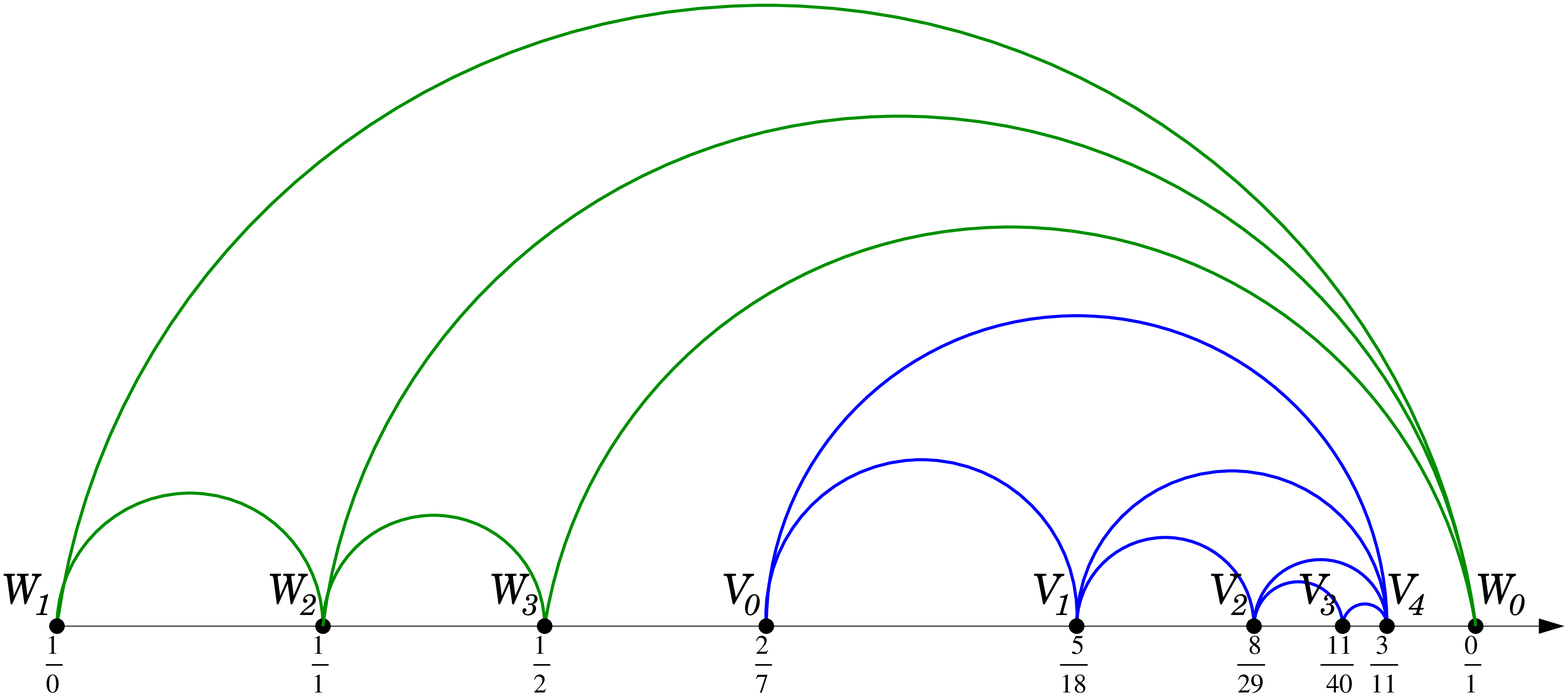}
\end{center}
The triangulated polygon $v$ and $w$ have quiddities $q$ and $q'$ respectively and they are normalized 
so that 
$(v_0, v_1)=(\frac27, \frac{5}{18})$ and $(w_0,w_1)=(\frac01, \frac10)$.
One constructs the corresponding $\SL_2$-tiling $(m_{i,j})$ by antiperiodicity using the formula
$$
m_{i+1,j+1}=a_id_j-b_jc_i, \text{ where } \frac{a_i}{c_i}=w_i, \;  \frac{b_j}{d_j}=v_j, \;0\leq i \leq m-1, \;0\leq j \leq n-1.
$$
One obtains here the following tiling
$$
 \begin{array}{rrrrrrrrrrrrrrrrrrrrrr}
 &\vdots&\vdots&\vdots&\vdots&\vdots&\vdots&\vdots&\vdots&\vdots&\vdots&\\[4pt]
\cdots&2&5&8&11&3&-2&-5&-8&-11&-3&\cdots\\[4pt]
\cdots&7&18&29&40&11&-7&-18&-29&-40&-11&\cdots\\[4pt]
\cdots&5&13&21&29&8&-5&-13&-21&-29&-8&\cdots\\[4pt]
\cdots&3&8&13&18&5&-3&-8&-13&-18&-5&\cdots\\[4pt]
\cdots&-2&-5&-8&-11&-3&2&5&8&11&3&\cdots\\[4pt]
\cdots&-7&-18&-29&-40&-11&7&18&29&40&11&\cdots\\[4pt]
\cdots&-5&-13&-21&-29&-8&5&13&21&29&8&\cdots\\[4pt]
\cdots&-3&-8&-13&-18&-5&3&8&13&18&5&\cdots\\
 &\vdots&\vdots&\vdots&\vdots&\vdots&\vdots&\vdots&\vdots&\vdots&\vdots&
\end{array}
$$
\end{ex}

\subsection{Friezes and dissections of polygons}\label{diss}

A natural way to define ``generalized Conway-Coxeter friezes'' is to use ``generalized triangulations''.
This direction has been investigated in \cite{BaMa}, \cite{BHJ} and \cite{Bes} using triangulations of punctured discs, and $d$-angulations, or more generally, dissections of polygons.

Let us  construct a symmetric matrix $M$
of size $n\times n$, from a frieze pattern of width $m=n-3$. 
The matrix $M$ is obtained by reflecting a fundamental triangular domain of the frieze
along a row of zeros. For instance, for the frieze \eqref{exCoCox} one would get the matrix
$$
\begin{pmatrix}
0&1&4&7&3&2&1\\
1&0&1&2&1&1&1\\
4&1&0&1&1&2&3\\
7&2&1&0&1&3&5\\
3&1&1&1&0&1&2\\
2&1&2&3&1&0&1\\
1&1&3&5&2&1&0
\end{pmatrix}
$$
(this matrix can also be found in the corresponding $\SL_2$-tiling by ignoring the signs, cf Example \ref{sl2CoCox}).

\begin{thm}[\cite{BCI}]\label{thmBCI}
The determinant of the matrix $M$ is independent of the choice of the triangular fundamental domain 
in the frieze. Moreover, one has
$$
\det(M)=-(-2)^{n-2}.
$$
\end{thm}

\begin{rem}
A generalization of Theorem \ref{thmBCI} for a matrix $M$ whose entries are general cluster variables of type $\mA$
can be found in \cite{BaMa2}.
\end{rem}

When considering the associated triangulated $n$-gon, the choice of a fundamental domain corresponds to a choice of cyclic labeling
of the vertices of the polygon. The entries $m_{i,j}$ of the matrix $M$ can be computed directly from the triangulation using 
the formulas \eqref{entBCI} or \eqref{entPath}.

The idea of  \cite{BHJ} and \cite{Bes} is to construct similar matrices $M$ from other types of ``triangulations'' given rise to pieces of more general frieze patterns.

A \textit{dissection} of an $n$-gon is an arbitrary collection of non-crossing diagonals,
or equivalently a collection of
smaller inner polygons $\pi_1, \ldots \pi_\ell$, with arbitrary numbers of edges (that are either sides or diagonals of the polygon),
say $d_1, \ldots, d_\ell$ respectively.
We refer to the dissection as a $(d_1, \ldots, d_\ell)$-\textit{dissection}.
In the particular case when all the inner polygons have same number of edges $d_i=d$, we call the dissection a $d$-\textit {angulation}.
Note that a $d$-angulation of an $n$-gon is possible if $n=d+(\ell-1)(d-2)$.

In what follows \cite{BHJ} is the reference for the case of $d$-angulation and \cite{Bes} the reference for the general case of dissection.

\begin{figure}
\begin{center}
\includegraphics[width=9cm]{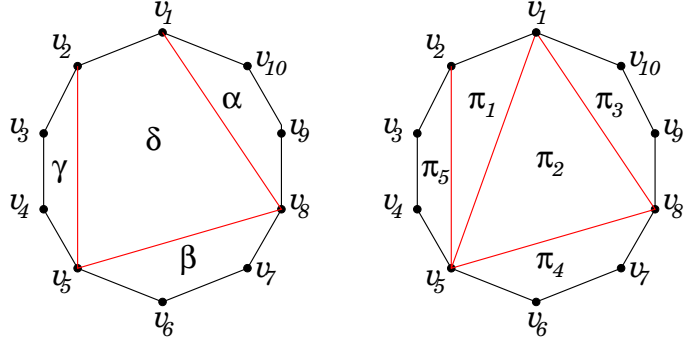}
\end{center}
\caption{A $4$-angulation \textit{(left)} and a $(3,3,4,4,4)$-dissection \textit{(right)} of a decagon}
\label{dissex}
\end{figure}

The notion of admissible path introduced in \S \ref{interpret}, Item 3, is generalized as follows.
\begin{defn}[\cite{BHJ},\cite{Bes}]
Given a dissection of a polygon $(v_i)$ (with convention on the vertices $v_{i+n}=v_{i-n}=v_i$), an \textit{admissible d-path} between vertices $ v_{i-1}$ and $v_{j+1}$ is an ordered sequence $(p_{i}, p_{i+1}, \ldots, p_{j})$,
of inner polygons such that the polygon $p_\ell$ is incident to vertex $\ell$, and for any $d$, a $d$-gon in the dissection does not appear more than $d-2$ times in the sequence.
\end{defn}

The construction of the matrix $M$ is generalized as follows.

\begin{defn}[\cite{BHJ},\cite{Bes}]
Given a dissection $\Dc$ of an $n$-gon $(v_i)$, define the matrix $M_\Dc=(m_{i,j})_{1\leq i,j\leq n}$
by $m_{i,i}=0$ and
$$
m_{i,j}=  \text{ the number of admissible d-paths from the vertex } v_{i-1} \text{ to  }v_{j-1},
$$
for all indices $1\leq i\not=j\leq n$.
\end{defn}
The matrix $M_\Dc=(m_{i,j})_{1\leq i,j\leq n}$ can be viewed as a piece of ``generalized Conway-Coxeter frieze pattern''.
It is shown in \cite{BHJ} that in the case of a $d$-angulation the adjacent $2\times 2$ minors of $M_\Dc$ are either $-1,0$ or $1$.

Theorem \ref{thmBCI} is generalized as follows.
\begin{thm}[\cite{BHJ},\cite{Bes}]
Let $\Dc$  be a  $(d_1, \ldots, d_\ell)$-dissection of an $n$-gon.
The matrix $M_\Dc$ is symmetric. The determinant does not depend on the choice of the labeling of the vertices of the polygon and
$$
\det(M_\Dc)=(-1)^{n-1}\prod_{i=1}^\ell(d_i-1).
$$
\end{thm}

\begin{ex}
For the dissections of Figure \ref{dissex} one obtains the following matrices, respectively
$$
\begin{pmatrix}
0&1&2&2&1&2&2&1&1&1\\
1&0&1&1&1&2&2&1&2&2\\
2&1&0&1&1&3&3&2&4&4\\
2&1&1&0&1&3&3&2&4&4\\
1&1&1&1&0&1&1&1&2&2\\
2&2&3&3&1&0&1&1&3&3\\
2&2&3&3&1&1&0&1&3&3\\
1&1&2&2&1&1&1&0&1&1\\
1&2&4&4&2&3&3&1&0&1\\
1&2&4&4&2&3&3&1&1&0
\end{pmatrix},
\qquad
\begin{pmatrix}
0&1&2&2&1&2&2&1&1&1\\
1&0&1&1&1&3&3&2&3&3\\
2&1&0&1&1&4&4&3&4&4\\
2&1&1&0&1&4&4&3&5&5\\
1&1&1&1&0&1&1&1&2&2\\
2&3&4&4&1&0&1&1&3&3\\
2&3&4&4&1&1&0&1&3&3\\
1&2&3&3&1&1&1&0&1&1\\
1&3&5&5&2&3&3&1&0&1\\
1&3&5&5&2&3&3&1&1&0
\end{pmatrix}.
$$
\end{ex}

\begin{rem}
The generalized frieze patterns coming from dissections of polygons can be obtained by a categorical approach using an analogue of the Caldero-Chapoton formula, \cite{HoJo1}, \cite{HoJo2}.
\end{rem}

\section{Friezes in the literature and open questions}

The notion of friezes on a repetition quiver and the one of $\SL_k$-tilings are ``direct'' generalizations and include 
Coxeter's frieze patterns. Many other generalizations or variants of Coxeter's frieze patterns have been studied recently.

The variations around Coxeter's definition can be made at different levels; one may change  the 
rule or the arithmetic used for the recurrence rule  on the diamond, $
 \begin{array}{c}
b\\[-4pt]
 a\quad d\\[-4pt]
c
\end{array}$,
and/or change the ``shape of the diamond'' 
or one may define an array from a combinatorial model... 

We list below some references where different types of friezes appear

\begin{itemize}
\item {friezes from matrix multiplication}: \cite{She}, \cite{Bro};
\item {additive friezes}: $a+d=b+c$ or $a+d=b+c+1$, \cite{She}, \cite{CoRi}, \cite{Marc}, \cite{Lau};
\item {tropical friezes}: $a+d=\max(b+c,0)$,  \cite{Pro}, \cite{Guo}, \cite{Pec}, \cite{Gra}, \cite{AsDu2}; \\ or $a+d=\max(b,0)+\max(c,0)$, \cite{Rin};
\item {quantum friezes}: $ad-q^\frac12bc=1$, \cite{BuDu};
\item {NIM friezes}: $a\boxplus d= b\boxplus c +1$, \cite{She}, \cite{Lau};
\item {cross-ratio friezes}: $\frac{(b-a)(c-d)}{(a-c)(d-b)}=-1$, \cite{Yao};
\item {continuous friezes} $F(x,y)$: $F\frac{\partial^2}{\partial x \partial y}F-\frac{\partial}{\partial x}F\frac{\partial}{ \partial y}F=1$, \cite{OvTaFrieze};
\item $2$-friezes:  $\begin{array}{c}
b\\[-4pt]
 a\; e\; d\\[-4pt]
c
\end{array} \Rightarrow ad-bc=e$, \cite{Pro}, \cite{MGOTaif}, \cite{MGjaco};
\item super-friezes or $\rm{Osp}(1|2)$-tilings, \cite{MGOTsuper}
\item $\SL_k$-friezes or $\SL_k$-tilings, \cite{CoRo}, \cite{ARS}, \cite{BeRe}, \cite{MGOTaif},\cite{MGOST}, \cite{MGOTprep}, \cite{Cun2};
\item 3D-friezes: $T$-systems, \cite{DiKe}, \cite{DiF} \cite{KeVi};
\item multiplicative friezes on repetition quivers or from Cartan matrices, \cite{CaCh}, \cite{BaMa}, \cite{ARS}, \cite{AsDu}, \cite{ADSS}, \cite{Ess}, \cite{FoPl}, \cite{Fon};
\item friezes and combinatorial models, \cite{BCI} \cite{BaMa}, \cite{BHJ}, \cite{HoJo}, \cite{Bes}, \cite{Cun}, \cite{Pec}, \cite{FoPl}, \cite{Fon}, \cite{Tsc}, \cite{BPT};
\item $\cdots$

\end{itemize}

\bigskip
\noindent
\textbf{Open questions}\\

\noindent
The study of friezes and of variants of friezes may be undertaken in many different directions. 
We suggest below some open questions and open direction of investigation.

The first five questions are related to friezes with positive integer values, 
see \S\ref{ZQentiere} and \S\ref{SLenum}.
The last three questions concern relations between friezes and other fields of mathematics.

\begin{quest}
How many multiplicative friezes with positive integer values are there in type $\mE_{7}$, 
$\mE_{8}$?
\end{quest}

\begin{quest}
How many $\SL_{3}$-friezes of width 4 with positive integer values are there?
\end{quest}

\begin{quest}
How many non-unitary friezes are there for a given type? See \cite{MGjaco}.
\end{quest}

\begin{quest}
Do there exist $\SL_{k}$-friezes with positive integer values for which the cluster variables in the associated cluster algebra that do not appear in the frieze are non-integer rational numbers?
\end{quest}

\begin{quest}
Find combinatorial models, as generalized triangulations, to enumerate $\SL_{k}$-friezes with positive integer entries.
\end{quest}

\begin{quest}
Are the friezes using NIM addition always periodic? See \cite{She}, \cite{Lau}.
\end{quest}

\begin{quest}
Find an analogue of the triality Theorem \ref{ThmTriality} in the supersymmetric case. See \cite{MGOTsuper}.
\end{quest}

\begin{quest}
Superperiodic equations are related to Sturm-Liouville's discrete oscillation theory \cite{ABGO}
\cite[\S4.5]{OvTa}.
Interpret the main results of this theory in terms of friezes, in particular,
 Sturm's separation and comparison theorems.
\end{quest}

\subsection*{Acknowledgement}

I would like to thank J. Conway, V. Ovsienko, S. Tabachnikov, 
P. Le Meur, Y. Palu, C. Riedtmann for enlightening discussions, comments and references.

\bibliographystyle{alpha}
\bibliography{BibFrisesCluster,BiblioMoy3}

\newcommand{\etalchar}[1]{$^{#1}$}
\def\cprime{$'$}
\begin{thebibliography}{MGOST14}

\bibitem[ABGO05]{ABGO}
Ravi~P. Agarwal, Martin Bohner, Said~R. Grace, and Donal O'Regan.
\newblock {\em Discrete oscillation theory}.
\newblock Hindawi Publishing Corporation, New York, 2005.

\bibitem[AD]{AsDu2}
Ibrahim Assem and Gr{\'e}goire Dupont.
\newblock Friezes over semirings and tropicalisations.
\newblock {\em Preprint available online}.

\bibitem[AD11]{AsDu}
Ibrahim Assem and Gr{\'e}goire Dupont.
\newblock Friezes and a construction of the {E}uclidean cluster variables.
\newblock {\em J. Pure Appl. Algebra}, 215(10):2322--2340, 2011.

\bibitem[ADSS12]{ADSS}
Ibrahim Assem, Gr{\'e}goire Dupont, Ralf Schiffler, and David Smith.
\newblock Friezes, strings and cluster variables.
\newblock {\em Glasg. Math. J.}, 54(1):27--60, 2012.

\bibitem[And78]{And}
D{\'e}sir{\'e} Andr{\'e}.
\newblock Terme g\'en\'eral d'une s\'erie quelconque d\'etermin\'ee \`a la
  fa\c{c}on des s\'eries r\'ecurrentes.
\newblock {\em Ann. Sci. \'Ecole Norm. Sup. (2)}, 7:375--408, 1878.

\bibitem[ARS95]{AuReS}
Maurice Auslander, Idun Reiten, and Sverre~O. Smal{\o}.
\newblock {\em Representation theory of {A}rtin algebras}, volume~36 of {\em
  Cambridge Studies in Advanced Mathematics}.
\newblock Cambridge University Press, Cambridge, 1995.

\bibitem[ARS10]{ARS}
Ibrahim Assem, Christophe Reutenauer, and David Smith.
\newblock Friezes.
\newblock {\em Adv. Math.}, 225(6):3134--3165, 2010.

\bibitem[ASS06]{ASS}
Ibrahim Assem, Daniel Simson, and Andrzej Skowro{\'n}ski.
\newblock {\em Elements of the representation theory of associative algebras.
  {V}ol. 1}, volume~65 of {\em London Mathematical Society Student Texts}.
\newblock Cambridge University Press, Cambridge, 2006.
\newblock Techniques of representation theory.

\bibitem[Aus74]{Aus}
Maurice Auslander.
\newblock Representation theory of {A}rtin algebras. {I}, {II}.
\newblock {\em Comm. Algebra}, 1:177--268; ibid. 1 (1974), 269--310, 1974.

\bibitem[BCI74]{BCI}
D.~Broline, D.~W. Crowe, and I.~M. Isaacs.
\newblock The geometry of frieze patterns.
\newblock {\em Geometriae Dedicata}, 3:171--176, 1974.

\bibitem[BD12]{BuDu}
Jean-Philippe Burelle and Gr{\'e}goire Dupont.
\newblock Quantum frieze patterns in quantum cluster algebras of type {$A$}.
\newblock {\em Int. Electron. J. Algebra}, 12:103--115, 2012.

\bibitem[Bes14]{Bes}
Christine Bessenrodt.
\newblock Conway-coxeter friezes and beyond: Polynomially weighted walks around
  dissected polygons and generalized frieze patterns.
\newblock {\em arXiv:1412.1726}, 2014.

\bibitem[BHJ12]{BHJ2}
Christine Bessenrodt, Thorsten Holm, and Peter J{\o}rgensen.
\newblock All {${\SL}_2$}-tilings come from triangulations.
\newblock {\em research report MFO}, 2012.

\bibitem[BHJ14]{BHJ}
Christine Bessenrodt, Thorsten Holm, and Peter J{\o}rgensen.
\newblock Generalized frieze pattern determinants and higher angulations of
  polygons.
\newblock {\em J. Combin. Theory Ser. A}, 123:30--42, 2014.

\bibitem[BM09]{BaMa}
Karin Baur and Robert~J. Marsh.
\newblock Frieze patterns for punctured discs.
\newblock {\em J. Algebraic Combin.}, 30(3):349--379, 2009.

\bibitem[BM12]{BaMa2}
Karin Baur and Robert~J. Marsh.
\newblock Categorification of a frieze pattern determinant.
\newblock {\em J. Combin. Theory Ser. A}, 119(5):1110--1122, 2012.

\bibitem[BMR{\etalchar{+}}06]{BMRRT}
Aslak~Bakke Buan, Robert Marsh, Markus Reineke, Idun Reiten, and Gordana
  Todorov.
\newblock Tilting theory and cluster combinatorics.
\newblock {\em Adv. Math.}, 204(2):572--618, 2006.

\bibitem[BPT15]{BPT}
Karin Baur, Mark~James Parsons, and Manuela Tschabold.
\newblock Infinite friezes.
\newblock {\em arXiv:1504.02695}, 2015.

\bibitem[BR90]{BaRe}
Vladimir Bazhanov and Nikolai Reshetikhin.
\newblock Restricted solid-on-solid models connected with simply laced algebras
  and conformal field theory.
\newblock {\em J. Phys. A}, 23(9):1477--1492, 1990.

\bibitem[BR10]{BeRe}
Fran{\c{c}}ois Bergeron and Christophe Reutenauer.
\newblock {$SL_k$}-tilings of the plane.
\newblock {\em Illinois J. Math.}, 54(1):263--300, 2010.

\bibitem[Bro78]{Bro}
Duane~M. Broline.
\newblock Frieze patterns as matrix multiplication tables.
\newblock In {\em Proceedings of the {N}inth {S}outheastern {C}onference on
  {C}ombinatorics, {G}raph {T}heory, and {C}omputing ({F}lorida {A}tlantic
  {U}niv., {B}oca {R}aton, {F}la., 1978)}, Congress. Numer., XXI, pages
  151--161. Utilitas Math., Winnipeg, Man., 1978.

\bibitem[CC73]{CoCo}
J.~H. Conway and H.~S.~M. Coxeter.
\newblock Triangulated polygons and frieze patterns.
\newblock {\em Math. Gaz.}, 57(400):87--94, 175--183, 1973.

\bibitem[CC06]{CaCh}
Philippe Caldero and Fr{\'e}d{\'e}ric Chapoton.
\newblock Cluster algebras as {H}all algebras of quiver representations.
\newblock {\em Comment. Math. Helv.}, 81(3):595--616, 2006.

\bibitem[CH11]{CuHe}
Michael Cuntz and Istvan Heckenberger.
\newblock Reflection groupoids of rank two and cluster algebras of type {$A$}.
\newblock {\em J. Combin. Theory Ser. A}, 118(4):1350--1363, 2011.

\bibitem[Cox71]{Cox}
H.~S.~M. Coxeter.
\newblock Frieze patterns.
\newblock {\em Acta Arith.}, 18:297--310, 1971.

\bibitem[CR72]{CoRo}
Craig~M. Cordes and D.~P. Roselle.
\newblock Generalized frieze patterns.
\newblock {\em Duke Math. J.}, 39:637--648, 1972.

\bibitem[CR94]{CoRi}
H.~S.~M. Coxeter and J.F Rigby.
\newblock {\em The Lighter Side of Mathematics}, chapter Frieze Patterns,
  Triangulated Polygons and Dichromatic Symmetry, pages 15--27.
\newblock Washington D.C.: The Mathematical Association of America, 1994.

\bibitem[Cun13]{Cun}
Michael Cuntz.
\newblock Frieze patterns as root posets and affine triangulations.
\newblock arXiv:1307.7986, 2013.

\bibitem[Cun15a]{Cun3}
Michael Cuntz.
\newblock Private communication, 2015.

\bibitem[Cun15b]{Cun2}
Michael Cuntz.
\newblock On wild frieze patterns.
\newblock {\em arXiv:1504.07048}, 2015.

\bibitem[DF10]{DiF}
Philippe Di~Francesco.
\newblock The solution of the {$A_r$} {$T$}-system for arbitrary boundary.
\newblock {\em Electron. J. Combin.}, 17(1):Research Paper 89, 43, 2010.

\bibitem[DFK09]{DiKe}
Philippe Di~Francesco and Rinat Kedem.
\newblock Positivity of the {$T$}-system cluster algebra.
\newblock {\em Electron. J. Combin.}, 16(1):Research Paper 140, 39, 2009.

\bibitem[EP00]{EiPo}
David Eisenbud and Sorin Popescu.
\newblock The projective geometry of the {G}ale transform.
\newblock {\em J. Algebra}, 230(1):127--173, 2000.

\bibitem[Ess14]{Ess}
Magnani~Kodjo Essonana.
\newblock Friezes of type {D}.
\newblock arXiv:1405.0548, 2014.

\bibitem[Fon14]{Fon}
Bruce Fontaine.
\newblock Non-zero integral friezes.
\newblock {\em arXiv:1409.6026}, 2014.

\bibitem[FPG14]{FoPl}
Bruce Fontaine and Plamondon Pierre-Guy.
\newblock Counting friezes in type $d_n$.
\newblock {\em arXiv:1409.3698}, 2014.

\bibitem[FR07]{FoRe}
Sergey Fomin and Nathan Reading.
\newblock Root systems and generalized associahedra.
\newblock In {\em Geometric combinatorics}, volume~13 of {\em IAS/Park City
  Math. Ser.}, pages 63--131. Amer. Math. Soc., Providence, RI, 2007.

\bibitem[FZ02a]{FZ1}
Sergey Fomin and Andrei Zelevinsky.
\newblock Cluster algebras. {I}. {F}oundations.
\newblock {\em J. Amer. Math. Soc.}, 15(2):497--529, 2002.

\bibitem[FZ02b]{FZLaurent}
Sergey Fomin and Andrei Zelevinsky.
\newblock The {L}aurent phenomenon.
\newblock {\em Adv. in Appl. Math.}, 28(2):119--144, 2002.

\bibitem[FZ03a]{FZ2}
Sergey Fomin and Andrei Zelevinsky.
\newblock Cluster algebras. {II}. {F}inite type classification.
\newblock {\em Invent. Math.}, 154(1):63--121, 2003.

\bibitem[FZ03b]{FZYsys}
Sergey Fomin and Andrei Zelevinsky.
\newblock {$Y$}-systems and generalized associahedra.
\newblock {\em Ann. of Math. (2)}, 158(3):977--1018, 2003.

\bibitem[Gab80]{Gab}
Peter Gabriel.
\newblock Auslander-{R}eiten sequences and representation-finite algebras.
\newblock In {\em Representation theory, {I} ({P}roc. {W}orkshop, {C}arleton
  {U}niv., {O}ttawa, {O}nt., 1979)}, volume 831 of {\em Lecture Notes in
  Math.}, pages 1--71. Springer, Berlin, 1980.

\bibitem[Gal56]{Gal}
David Gale.
\newblock Neighboring vertices on a convex polyhedron.
\newblock In {\em Linear inequalities and related system}, Annals of
  Mathematics Studies, no. 38, pages 255--263. Princeton University Press,
  Princeton, N.J., 1956.

\bibitem[GM82]{GeMc}
I.~M. Gel{\cprime}fand and R.~D. MacPherson.
\newblock Geometry in {G}rassmannians and a generalization of the dilogarithm.
\newblock {\em Adv. in Math.}, 44(3):279--312, 1982.

\bibitem[Gra13]{Gra}
Jan~E. Grabowski.
\newblock Graded cluster algebras.
\newblock arXiv:1309.6170, 2013.

\bibitem[GSV10]{GSV}
Michael Gekhtman, Michael Shapiro, and Alek Vainshtein.
\newblock {\em Cluster algebras and {P}oisson geometry}, volume 167 of {\em
  Mathematical Surveys and Monographs}.
\newblock American Mathematical Society, Providence, RI, 2010.

\bibitem[Guo13]{Guo}
Lingyan Guo.
\newblock On tropical friezes associated with {D}ynkin diagrams.
\newblock {\em Int. Math. Res. Not. IMRN}, (18):4243--4284, 2013.

\bibitem[HJ13]{HoJo}
Thorsten Holm and Peter J{\o}rgensen.
\newblock {${\SL}_2$}-tilings and triangulations of the strip.
\newblock {\em J. Combin. Theory Ser. A}, 120(7):1817--1834, 2013.

\bibitem[HJ14a]{HoJo1}
Thorsten Holm and Peter J{\o}rgensen.
\newblock Generalised friezes and a modified {C}aldero-{C}hapoton map depending
  on a rigid object.
\newblock {\em Nagoya Math. J., to appear.}, 2014.

\bibitem[HJ14b]{HoJo2}
Thorsten Holm and Peter J{\o}rgensen.
\newblock Generalised friezes and a modified {C}aldero-{C}hapoton map depending
  on a rigid object {II}, 2014.

\bibitem[Jor39]{Jor}
Charles Jordan.
\newblock {\em Calculus of {F}inite {D}ifferences}.
\newblock Hungarian Agent Eggenberger Book-Shop, Budapest, 1939.

\bibitem[Kel10]{Kel}
Bernhard Keller.
\newblock Cluster algebras, quiver representations and triangulated categories.
\newblock In {\em Triangulated categories}, volume 375 of {\em London Math.
  Soc. Lecture Note Ser.}, pages 76--160. Cambridge Univ. Press, Cambridge,
  2010.

\bibitem[Kel13]{Kel2}
Bernhard Keller.
\newblock The periodicity conjecture for pairs of {D}ynkin diagrams.
\newblock {\em Ann. of Math. (2)}, 177(1):111--170, 2013.

\bibitem[KNS94]{KNS}
Atsuo Kuniba, Tomoki Nakanishi, and Junji Suzuki.
\newblock Functional relations in solvable lattice models. {I}. {F}unctional
  relations and representation theory.
\newblock {\em Internat. J. Modern Phys. A}, 9(30):5215--5266, 1994.

\bibitem[KQ14]{KiQi}
Yoshiyuki Kimura and Fan Qin.
\newblock Graded quiver varieties, quantum cluster algebras and dual canonical
  basis.
\newblock {\em Adv. Math.}, 262:261--312, 2014.

\bibitem[Kri14]{Kri}
Igor Krichever.
\newblock Commuting difference operators and the combinatorial {G}ale
  transform.
\newblock arXiv:1403.4629, 2014.

\bibitem[KS11]{KeSc}
Bernhard Keller and Sarah Scherotzke.
\newblock Linear recurrence relations for cluster variables of affine quivers.
\newblock {\em Adv. Math.}, 228(3):1842--1862, 2011.

\bibitem[KV15]{KeVi}
Rinat Kedem and Panupong Vichitkunakorn.
\newblock {$T$}-systems and the pentagram map.
\newblock {\em J. Geom. Phys.}, 87:233--247, 2015.

\bibitem[Lau14]{Lau}
Adrien Laurent.
\newblock Frises.
\newblock Rapport de stage, available online, 2014.

\bibitem[LS13]{LeSc}
Kyungyong Lee and Ralf Schiffler.
\newblock Positivity for cluster algebras.
\newblock arXiv:1306.2415, 2013.

\bibitem[Mar12]{Marc}
Jean-Fran\c{c}ois Marceau.
\newblock Pavages additifs.
\newblock arXiv:1205.5213, 2012.

\bibitem[Mar13a]{Marsh}
Robert~J. Marsh.
\newblock {\em Lecture notes on cluster algebras}.
\newblock Zurich Lectures in Advanced Mathematics. European Mathematical
  Society (EMS), Z\"urich, 2013.

\bibitem[Mar13b]{MarshBook}
Robert~J. Marsh.
\newblock {\em Lecture notes on cluster algebras}.
\newblock Zurich Lectures in Advanced Mathematics. European Mathematical
  Society (EMS), Z\"urich, 2013.

\bibitem[MG12]{MGjaco}
Sophie Morier-Genoud.
\newblock Arithmetics of 2-friezes.
\newblock {\em J. Algebraic Combin.}, 36(4):515--539, 2012.

\bibitem[MGOST14]{MGOST}
Sophie Morier-Genoud, Valentin Ovsienko, Richard Schwartz, and Serge
  Tabachnikov.
\newblock Linear difference equations, frieze patterns, and the combinatorial
  {G}ale transform.
\newblock {\em Forum of Mathematics, Sigma}, 2, 2014.

\bibitem[MGOT12]{MGOTaif}
Sophie Morier-Genoud, Valentin Ovsienko, and Serge Tabachnikov.
\newblock 2-frieze patterns and the cluster structure of the space of polygons.
\newblock {\em Ann. Inst. Fourier (Grenoble)}, 62(3):937--987, 2012.

\bibitem[MGOT14]{MGOTprep}
Sophie Morier-Genoud, Valentin Ovsienko, and Serge Tabachnikov.
\newblock $\mathrm{SL}_2(\mathbb{Z})$-tilings of the torus, {C}oxeter-{C}onway
  friezes and {F}arey triangulations.
\newblock 2014.
\newblock arXiv:1402.5536, Enseign. Math. to appear.

\bibitem[MGOT15]{MGOTsuper}
Sophie Morier-Genoud, Valentin Ovsienko, and Serge Tabachnikov.
\newblock Introducing supersymmetric frieze patterns and linear difference
  operators.
\newblock {\em arXiv:1501.07476}, 2015.

\bibitem[Mui60]{Mui}
Thomas Muir.
\newblock {\em A treatise on the theory of determinants}.
\newblock Revised and enlarged by William H. Metzler. Dover Publications, Inc.,
  New York, 1960.

\bibitem[OST10]{OST1}
Valentin Ovsienko, Richard Schwartz, and Serge Tabachnikov.
\newblock The pentagram map: {A} discrete integrable system.
\newblock {\em Comm. Math. Phys.}, 299(2):409--446, 2010.

\bibitem[OT05]{OvTa}
Valentin Ovsienko and Serge Tabachnikov.
\newblock {\em Projective differential geometry old and new}, volume 165 of
  {\em Cambridge Tracts in Mathematics}.
\newblock Cambridge University Press, Cambridge, 2005.
\newblock From the Schwarzian derivative to the cohomology of diffeomorphism
  groups.

\bibitem[OT15]{OvTaFrieze}
Valentin Ovsienko and Serge Tabachnikov.
\newblock Coxeter's frieze patterns and discretization of the {V}irasoro orbit.
\newblock {\em J. Geom. Phys.}, 87:373--381, 2015.

\bibitem[Pec14]{Pec}
Oliver Pechenik.
\newblock Cyclic sieving of increasing tableaux and small {S}chr\"oder paths.
\newblock {\em J. Combin. Theory Ser. A}, 125:357--378, 2014.

\bibitem[Pro05]{Pro}
James Propp.
\newblock The combinatorics of frieze patterns and {M}arkoff numbers.
\newblock arxiv math/0511633, 2005.

\bibitem[Rei10]{Rei}
Idun Reiten.
\newblock Cluster categories.
\newblock In {\em Proceedings of the {I}nternational {C}ongress of
  {M}athematicians. {V}olume {I}}, pages 558--594. Hindustan Book Agency, New
  Delhi, 2010.

\bibitem[Rie80]{Rie}
Christine Riedtmann.
\newblock Algebren, {D}arstellungsk\"ocher, \"{U}berlagerungen und zur\"uck.
\newblock {\em Comment. Math. Helv.}, 55(2):199--224, 1980.

\bibitem[Rin12]{Rin}
Claus~Michael Ringel.
\newblock Cluster-additive functions on stable translation quivers.
\newblock {\em J. Algebraic Combin.}, 36(3):475--500, 2012.

\bibitem[Sch]{SchwApp}
Richard~Evan Schwartz.
\newblock http://www.math.brown.edu/~res/Java/Frieze/Main.html.

\bibitem[Sch14]{Schi}
Ralf Schiffler.
\newblock {\em Quiver Representations}.
\newblock CMS Books in Mathematics, Springer Verlag, 2014.

\bibitem[Sco06]{Sco}
Joshua~S. Scott.
\newblock Grassmannians and cluster algebras.
\newblock {\em Proc. London Math. Soc. (3)}, 92(2):345--380, 2006.

\bibitem[Sha84]{Sha}
Louis~W. Shapiro.
\newblock Positive definite matrices and {C}atalan numbers, revisited.
\newblock {\em Proc. Amer. Math. Soc.}, 90(3):488--496, 1984.

\bibitem[She76]{She}
G.~C. Shephard.
\newblock Additive frieze patterns and multiplication tables.
\newblock {\em Math. Gaz.}, 60(413):178--184, 1976.

\bibitem[Tsc15]{Tsc}
Manuela Tschabold.
\newblock Arithmetic infinite friezes from punctured discs.
\newblock {\em arXiv:1503.04352}, 2015.

\bibitem[Vol07]{Vol}
Alexandre~Yu. Volkov.
\newblock On the periodicity conjecture for {$Y$}-systems.
\newblock {\em Comm. Math. Phys.}, 276(2):509--517, 2007.

\bibitem[Yao14]{Yao}
Zijian Yao.
\newblock Glick's conjecture on the point of collapse of axis-aligned polygons
  under the pentagram maps.
\newblock {\em arXiv:1410.7806}, 2014.

\end{thebibliography}

\end{document}